\documentclass[12pt]{amsart}

\usepackage[all]{xy}
\usepackage{fullpage}
\usepackage{latexsym}
\usepackage{amsmath}
\usepackage{amsfonts}
\usepackage{amssymb}
\usepackage{amsthm}
\usepackage{eucal}
\usepackage{enumerate,yfonts}
\usepackage{mathrsfs}
\usepackage{graphicx}
\usepackage{graphics}
\usepackage{epstopdf}
\usepackage{amscd}
\usepackage{bbm}
\usepackage{hyperref}
\usepackage{url}
\usepackage{color}
\usepackage{bbm}
\usepackage{cancel}
\usepackage{enumerate}
\usepackage{amsmath,amsthm}
\usepackage{amssymb}
\usepackage{epsfig}
\usepackage{pstricks}
\usepackage{xy}
\usepackage{xypic}
\usepackage{bbm}
\usepackage{inputenc}

\newtheorem{thm}{Theorem}[section]
\newtheorem{corollary}[thm]{Corollary}
\newtheorem{lemma}[thm]{Lemma}

\newtheorem{proposition}[thm]{Proposition}
\newtheorem{prop}[thm]{Proposition}

\newtheorem{thm-dfn}[thm]{Theorem-Definition}

\theoremstyle{definition}
\newtheorem{definition}[thm]{Definition}
\newtheorem{remark}[thm]{Remark}
\newtheorem{example}[thm]{Example}

\numberwithin{equation}{section}

\newcommand{\fg}{{\mathfrak g}}

\newcommand{\fp}{{\mathfrak p}}

\newcommand{\bC}{{\mathbb C}}
\newcommand{\bP}{{\mathbb P}}

\newcommand{\bG}{{\mathbb G}}
\newcommand{\bZ}{{\mathbb Z}}

\newcommand{\mE}{\mathcal{E}}
\newcommand{\mF}{\mathcal{F}}

\newcommand{\mM}{\mathcal{M}}

\newcommand{\mO}{\mathcal{O}}
\newcommand{\mL}{\mathcal{L}}

\newcommand{\mG}{\mathcal{G}}

\newcommand{\sX}{\mathscr{X}}
\newcommand{\sY}{\mathscr{Y}}

\newcommand{\on}{\operatorname}

\newcommand{\lra}{\longrightarrow}

\newcommand{\ra}{\rightarrow}
\newcommand{\la}{\leftarrow}

\newcommand{\is}{\simeq}

\newcommand{\Loc}{\on{LocSys}}
\newcommand{\Bun}{\on{Bun}}

\newcommand{\quash}[1]{}  
\newcommand{\nc}{\newcommand}

\newcommand{\frakg}{{\mathfrak g}}

\newcommand{\bbA}{{\mathbb A}}

\newcommand{\bbC}{{\mathbb C}}

\newcommand{\bbG}{{\mathbb G}}

\newcommand{\bbP}{{\mathbb P}}

\newcommand{\bbR}{{\mathbb R}}

\newcommand{\calB}{{\mathcal B}}

\newcommand{\calE}{{\mathcal E}}
\newcommand{\calF}{{\mathcal F}}

\newcommand{\calK}{{\mathcal K}}
\newcommand{\calL}{{\mathcal L}}
\newcommand{\calM}{{\mathcal M}}
\newcommand{\calN}{{\mathcal N}}
\newcommand{\calO}{{\mathcal O}}

\newcommand{\calS}{{\mathcal S}}
\newcommand{\calT}{{\mathcal T}}

\newcommand{\calV}{{\mathcal V}}

\newcommand{\calX}{{\mathcal X}}

\nc{\al}{{\alpha}} \nc{\be}{{\beta}} \nc{\ga}{{\gamma}}
\nc{\ve}{{\varepsilon}} \nc{\Ga}{{\Gamma}} 
\nc{\La}{{\Lambda}}

\nc{\ad }{{\on{ad }}}

\nc{\aff}{{\on{aff}}} \nc{\Aff}{{\mathbf{Aff}}}
\newcommand{\Aut}{{\on{Aut}}}

\nc{\der}{{\on{der}}}

\nc{\diag}{{\on{diag}}}

\nc{\Fl}{{\calF\ell}}
\newcommand{\Gal}{{\on{Gal}}}
\newcommand{\Gr}{{\on{Gr}}}
\nc{\Hg}{{\on{Higgs}}}
\newcommand{\Hom}{{\on{Hom}}}
\newcommand{\id}{{\on{id}}}
\nc{\Id}{{\on{Id}}}

\nc{\Ind}{{\on{Ind}}}

\nc{\Op}{{\on{Op}}}

\nc{\res}{{\on{res}}}

\newcommand{\Spec}{{\on{Spec}}}

\nc{\tr}{{\on{tr}}}

\nc{\GSp}{{\on{GSp}}} \nc{\GU}{{\on{GU}}} \nc{\SL}{{\on{SL}}}
\nc{\SU}{{\on{SU}}} \nc{\SO}{{\on{SO}}}

\nc{\nh}{{\Loc_{J^p}(\tau')}}
\nc{\bnh}{{\Loc_{\breve J^p}(\tau')}}

\nc{\bU}{{\overline{U}}} \nc{\IC}{{\on{IC}}}

\newcommand{\beqn}{\begin{equation*}}
\newcommand{\eeqn}{\end{equation*}}

\newcommand{\beq}{\begin{equation}}
\newcommand{\eeq}{\end{equation}}

\nc{\QM}{QM}

\quash{
\topmargin-0.5cm \textheight22cm \oddsidemargin1.2cm \textwidth14cm}
\begin{document}
\title{Affine Matsuki correspondence for sheaves}

       \author{Tsao-Hsien Chen} 
        
       \address{Department of Mathematics, University of Chicago, Chicago, 60637, USA}
       \email{chenth@math.uchicago.edu}
       \author{David Nadler} 
        
        \address{Department of Mathematics, UC Berkeley, Evans Hall,
Berkeley, CA 94720}
        \email{nadler@math.berkeley.edu}
       
\maketitle
\begin{abstract}
We lift the affine Matsuki correspondence between real and symmetric loop group orbits in affine Grassmannians to an equivalence of derived categories of 
sheaves. 
 In analogy with the finite-dimensional setting, our arguments depend upon the Morse theory of energy functions
obtained from symmetrizations of coadjoint orbits. The additional fusion structures  of the affine setting lead to further equivalences with Schubert constructible 
derived categories of sheaves on real affine Grassmannians.

\end{abstract}
\setcounter{tocdepth}{2} \tableofcontents

\section{Introduction}

This is the first of several papers devoted to the geometry and representation theory of real loop groups $LG_\bbR$, i.e.,~groups of maps from the circle into a real reductive (not necessarily compact) Lie group $G_\bbR$. Some of our primary motivations established here or in the sequels include the following:

\begin{enumerate}

\item  A lift of the {\em affine Matsuki correspondence}~\cite{N1} between real and symmetric loop group orbits in affine Grassmannians to an equivalence of derived categories of 
sheaves.

  \item A lift of the {\em Kostant-Sekiguchi correspondence}~\cite{S} between real and symmetric nilpotent orbits to an equivariant stratified homeomorphism (see \cite{CN2}).

\item  The development of a {\em representation theory of real loop groups} from the well-known setting of compact groups~\cite{PS} to general reductive groups.

\end{enumerate}

Of the preceding goals, the current paper establishes (1) which in turn provides the geometry underlying our approach to (2). It also introduces and establishes basic properties of the moduli of quasi-maps that play a fundamental role in (3). 

In what immediately follows, we describe the main results of (1) in more detail, including the remarkable relation of real and symmetric loop group orbits in affine flag varieties to 
real affine Schubert geometry. We then sketch some of the applications to (2), (3) and other topics to be established in sequel papers.

\subsection{Matsuki correspondence for sheaves}\label{flag manifold}
We begin by recalling the Matsuki correspondence for sheaves~\cite{MUV}. 
It intertwines
the Beilinson-Bernstein localization~\cite{BB} of Harish Chandra $(\frakg, K)$-modules with the Kashiwara-Schmid localization~\cite{KS} of (infinitesimal classes of) admissible representations of $G_\bbR$.

Let $\bbG_\mathbb R$ be a connected real reductive algebraic group, and
 $\bbG=\bbG_\bbR\otimes_\bbR\bbC$  its complexification.
From this starting point, one  constructs the following diagram of Lie groups
\beq\label{eq:diamond}
\xymatrix{&G&\\
K\ar[ru]\ar[ru]&G_\mathbb R\ar[u]&G_c\ar[lu]\\
&K_c\ar[lu]\ar[u]\ar[ru]&}
\eeq
Here $G=\bbG(\bC)$ and $G_\bbR=\bbG_\bbR(\bbR)$ are the Lie groups of complex and real points respectively, 
$K_c$ is a maximal compact subgroup of $G_\mathbb R$, with complexification $K$,
and $G_c$ is the maximal compact subgroup of $G$ containing $K_c$.

%

Let $\calB \simeq G/B$ be the flag manifold of Borel subgroups of $G$.  The groups
$K$ and $G_\bbR$ act on $\calB$ with finitely many orbits and 
the classical Matsuki correspondence~\cite{M}  provides
 an anti-isomorphism of orbit posets  
\beq
|K\backslash\calB|\longleftrightarrow|G_\bbR\backslash\calB|
\eeq
between the sets  of  $K$-orbits and $G_\bbR$-orbits  on 
$\calB$, each ordered with respect to orbit closures. The correspondence matches a $K$-orbit $\mO^+$ with the unique  
 $G_\bbR$-orbit $\mO^-$ such that the intersection 
$\mO=\mO^+\cap\mO^-$ is a non-empty $K_c$-orbit.

The Matsuki correspondence for
 sheaves~\cite{MUV}, as conjectured by Kashiwara, lifts this anti-isomorphism of posets  to an equivalence  
 \beq\label{Matsuki correspondence for shvs}
 D_c(K\backslash\calB)\is D_c(G_\bbR\backslash\calB)
 \eeq
between the bounded constructible $K$-equivariant and $G_\bbR$-equivariant   derived categories of $\mathcal B$.
The main ingredient of the proof is a Morse-theoretic interpretation and refinement 
of the Matsuki correspondence due to Uzawa. 

\quash{
Choose a highest weight representation 
$V_\lambda$ of $G$ associated to 
an integral, regular dominant weight $\lambda$. Then the corresponding 
line bundle $\mL_\lambda$ on $\calB$ gives us an embedding 
$\calB\hookrightarrow\mathbb P(V_\lambda)$ and the Fubini-Study metric on 
$\mathbb P(V_\lambda)$ restrict to a $G_c$-invariant 
metric on $\calB$.
We have the following refined version of Matsuki correspondence:

\begin{thm}\cite{MUV}
There is a $K_c$-invariant Bott-Morse function 
$f:\calB\to\bbR$ such that its gradient flow
$\phi_t$
with respect to the the metric above 
satisfies the following properties:
\begin{enumerate}
\item
The critical manifold of $f$, that is, the fixed point set of $\phi_t$, consists of 
finitely many $K_c$-orbits $\mO$.
The flow preserves 
$K$-and $G_\bbR$-orbits on $\calB$.
\item
The limit $\underset{t\ra\pm\infty}\lim\phi_t(x)$ exists for any $x\in\calB$.
For any critical $K_c$-orbit $\mO$, the set $\mO^+=\{x\in\calB|\ \underset{t\ra\infty}\lim\phi_t(x)\in\mO\}$ is a single 
$K$-orbit, the set $\mO^-=\{x\in\calB|\ \underset{t\ra-\infty}\lim\phi_t(x)\in\mO\}$
is a single $G_\bbR$-orbit, and the bijection $\mO^+\longleftrightarrow\mO^-$ 
between $K$-and $G_\bbR$-orbits is the Matsuki correspondence.

\end{enumerate}
\end{thm}
}

\subsection{Affine Matsuki correspondence for sheaves}
Now let us turn to the affine setting.  

Let $\calO = \bbC[[t]]$ be the ring of formal power series, and $\calK = \bbC((t))$ the field of formal Laurent series. Let $D= \Spec \calO$ be the formal disk, and $D^\times = \Spec \calK$ the formal punctured disk.
Let $\bC[t,t^{-1}]$ be the ring of Laurent polynomials so that $\bbG_m = \Spec  \bC[t,t^{-1}]$.
%
%

In place of diagram~\eqref{eq:diamond}, we take the diagram of loop groups
\beq\label{eq:aff diamond}
\xymatrix{&G(\calK)&\\
K(\calK)\ar[ru]\ar[ru]&LG_\mathbb R\ar[u]&LG_c\ar[lu]\\
&LK_c\ar[lu]\ar[u]\ar[ru]&}
\eeq
Here $G(\calK)$ and $K(\calK)$ are the formal loop groups of maps $D^\times \to G$ and $D^\times  \to K$ respectively, $LG_\bbR, LG_c$, and $LK_c$ 
are the subgroups of the polynomial loop group $LG=G(\bC[t,t^{-1}])$  of those maps 
 that take the unit circle $S^1$ into 
$K_c, G_c$, and $G_\bbR$ respectively.

In this paper, the role of the flag manifold $\calB \simeq G/B$ will be played 
by the affine Grassmannian $\Gr=G(\calK)/G(\calO)$.\footnote{Throughout this paper, we will be concerned exclusively with the topology of $\Gr$ and
related moduli and ignore their potentially non-reduced structure.}  (In a sequel paper~\cite{CN1}, we will extend many of our results to the affine flag manifold  
$\Fl=G(\calK)/I$, where $I \subset G(\calO)$ is an Iwahori subgroup. Our focus in this paper is the remarkable connection between the Matsuki correspondence for  the affine Grassmannian  and real  Schubert geometry as highlighted in Sect.~\ref{Relation to Schubert geometry} below.)

The paper~\cite{N1} establishes a Matsuki correspondence
for the affine Grassmannian: there is an anti-isomorphism of orbit posets  
\beq\label{AMat}
|K(\calK)\backslash\Gr|\longleftrightarrow|LG_\bbR\backslash\Gr|
\eeq
between the sets of $K(\calK)$-orbits and $LG_\bbR$-orbits on 
$\Gr$,
each ordered with respect to orbit closures. The correspondence matches a $K(\calK)$-orbit $\mO_K$
with the unique
$LG_\bbR$-orbit $\mO_\bbR$  such that the intersection 
$\mO_c=\mO_K\cap\mO_\bbR$ is a non-empty $LK_c$-orbit. 

Furthermore, the paper~\cite{N1} provides an explicit parametrization of the orbit posets
(see Sect.~\ref{orbits} for a review).

The first main result of this paper is the following Morse-theoretic interpretation and refinement of the Matsuki correspondence for 
the affine Grassmannian:

\begin{thm}[Theorem \ref{flow} below]\label{first result}
There is a $LK_c$-invariant function 
$E:\Gr\to\bbR$ and $LG_c$-invariant metric on $\Gr$ such that the associated  gradient $\nabla E$ and gradient-flow
$\phi_t$  satisfy the following:
\begin{enumerate}
\item The critical locus $\nabla E = 0$ is a disjoint union of 
$LK_c$-orbits.
\item
The gradient-flow $\phi_t$
preserves the $K(\calK)$-and $LG_\bbR$-orbits.
\item
The limits $\lim_{t\ra\pm\infty}\phi_t(\gamma)$ of the gradient-flow exist for any $\gamma\in\Gr$.
For each  $LK_c$-orbit $\mO_c$ in the critical locus,
the stable and unstable sets 
\beq\label{eq:morse slows}
\xymatrix{
\mO_K=\{\gamma\in\Gr|\underset{t\ra\infty}\lim_{}\phi_t(\gamma)\in\mO_c\}
&
\mO_\bbR=\{\gamma\in\Gr|\underset{t\ra-\infty}\lim_{}\phi_t(\gamma)\in\mO_c\}
}\eeq
are a single $K(\calK)$-orbit and $LG_\bbR$-orbit respectively.
\item
The correspondence between orbits $\mO_K\longleftrightarrow\mO_\bbR$ defined by \eqref{eq:morse slows} recovers the affine Matsuki correspondence~\eqref{AMat}.

\end{enumerate}

\end{thm}


Using the above refinement of the affine  Matsuki correspondence~\eqref{AMat}, we prove 
a Matsuki correspondence for sheaves on the affine Grassmannian
in analogy with \eqref{Matsuki correspondence for shvs}.
In order to make sense of the bounded constructible $K(\calK)$-and $LG_\bbR$-equivariant  derived categories 
of $\Gr$, we give  
moduli interpretations of the quotient stacks $K(\calK)\backslash\Gr$
and $LG_\bbR\backslash\Gr$.  For simplicity, for the rest of the introduction,
except Sect. \ref{KS},
we 
assume $K$ is connected.

Let us first discuss the quotient $LG_\bbR\backslash\Gr$. 
Consider the moduli stack $\Bun_\bbG(\mathbb P^1)$ of $\bbG$-bundles on the projective line $\bbP^1$,
and its standard real form $\Bun_{\bbG_\bbR}(\bbP_\bbR^1)$ of $\bbG_\bbR$-bundles on 
the real projective line $\bbP_\bbR^1$. The real points of 
$\Bun_{\bbG_\bbR}(\bbP_\bbR^1)$  naturally form a real analytic stack 
we denote by $\Bun_\bbG(\mathbb P^1)_\bbR$. 
In general, it is disconnected
and we denote by $\Bun_\bbG(\mathbb P^1)_{\bbR,\alpha_0}$
the union of those components consisting of 
$\bbG_\bbR$-bundles on $\bbP_\bbR^1$ that are 
trivializable at the point $\infty \in \bbP^1_\bbR$ (see Sect.~\ref{uniform} for details).

In Proposition~\ref{uniformizations at complex x}, we prove the following:
\beq\label{Moduli int for LG_R}
\begin{gathered}
\text{The quotient $LG_\bbR\backslash\Gr$ is a real analytic stack isomorphic to $\Bun_\bbG(\bbP^1)_{\bbR,\alpha_0}$.}
\end{gathered}
\eeq
Thus the real analytic stack
$LG_\bbR\backslash\Gr$ is locally of finite type and we have a well-defined 
category of sheaves on it.
\begin{definition}
Let $D_c(LG_\bbR\backslash\Gr)$ 
 be the bounded constructible derived category of  sheaves on $LG_\bbR\backslash\Gr$. 
We set $D_!(LG_\bbR\backslash\Gr)$ 
to  be the full subcategory of 
$D_c(LG_\bbR\backslash\Gr)$ consisting of all  complexes that are extensions by zero off of finite type substacks.
\end{definition}

Next let us discuss the quotient 
$K(\calK)\backslash\Gr$. In general, the $K(\calK)$-orbits on $\Gr$ are 
neither finite-dimensional nor finite-codimensional 
(unlike the $LG_\bbR$-orbits on $\Gr$ which are finite-codimensional).
Thus there is not a naive
approach  to sheaves on 
$K(\calK)\backslash\Gr$ with traditional methods.
To overcome this, we use the observation that
 the quotient $LK_c\backslash\Gr$ is 
a real analytic 
ind-stack of ind-finite type, i.e.,
an inductive limit of 
real analytic stacks of finite type. We will take a certain subcategory of sheaves on 
$LK_c\backslash\Gr$ as a replacement for 
$K(\calK)$-equivariant sheaves on $\Gr$. 

To give more details,  
denote by $z\mapsto \bar z$ the  standard conjugation of $\bbP^1$ with real form $\bbP^1_\bbR$, and equip  $(\bbP^1)^2$ with the conjugation $(z_1, z_2) \mapsto (\bar z_2, \bar z_1)$. Its real points  $(\bbP^1)^2_\bbR$   form a real analytic space  isomorphic to $\bbP^1(\bbC)$ via the projection $(z_1, z_2) \mapsto z_1$.

Next, introduce the ind-stack of 
quasi-maps $QM^{(2)}(\bbP^1, G, K)$ classifying  
$(z_1, z_2, \mE,\sigma)$ where $(z_1, z_2)$ is a point of $(\bbP^1)^2$, 
$\mE$ is a $\bbG$-bundle on $\bbP^1$, and 
$\sigma$ is a section 
$\bbP^1\setminus \{z_1, z_2\}\to\mE\times^G G/K$. 
The given  conjugations on $(\bbP^1)^2$, $G, K$ induce a conjugation on 
$QM^{(2)}(\bbP^1, G, K)$, and  we denote by 
$QM^{(2)}(\bbP^1, G, K)_\bbR$ the real analytic 
ind-stack of its real points.
There is  a natural projection  $QM^{(2)}(\bbP^1, G, K)_\bbR\to
\Bun_G(\bbP^1)_\bbR$, $(z_1, z_2, \mE,\sigma) \mapsto \mE$, and we write 
$QM^{(2)}(\bbP^1, G, K)_{\bbR, \alpha_0}$ for the pre-image of the 
components $\Bun_G(\bbP^1)_{\bbR,\alpha_0}$.

We also have the natural projection  $QM^{(2)}(\bbP^1, G, K)_\bbR\to (\bbP^1)^2_\bbR \is \bbP^1(\bbC)$,
$(z_1, z_2, \mE,\sigma) \mapsto z_1$. 
For $z \in \bbP^1(\bbC)$,  denote by 
$QM^{(2)}(\bbP^1,z, G, K)_\bbR$ the fiber of  $QM^{(2)}(\bbP^1, G, K)_\bbR$ over $z$, and  
by $QM^{(2)}(\bbP^1, z, G, K)_{\bbR, \alpha_0}$ the intersection of the fiber with $QM^{(2)}(\bbP^1, G, K)_{\bbR, \alpha_0}$.
In particular, for the (non-real)  point $i\in \bbP^1(\bbC)$, we have the  fiber 
$QM^{(2)}(\bbP^1, i, G, K)_\bbR$, and the intersection $QM^{(2)}(\bbP^1, i, G, K)_{\bbR, \alpha_0}$.

In section \ref{QMaps}, we prove the following:
\beq
\begin{gathered}
\text{
The quotient $LK_c\backslash\Gr$ is a real analytic ind-stack isomorphic to $QM^{(2)}(\bbP^1, i, G, K)_{\bbR, \alpha_0}$.}
\end{gathered}
\eeq
Thus the real analytic ind-stack
$LK_c\backslash\Gr$ is  of ind-finite type and we have a well-defined 
category of sheaves on it.

Finally, denote by $\calS$ the stratification  of $LK_c\backslash\Gr$ with strata the
$LK_c$-quotients of $K(\calK)$-orbits.
\quash{Note the strata 
are finite-dimensional.}
By Theorem \ref{first result}, each $K(\calK)$-orbit 
$\calO_K$ deformation retracts to an $LK_c$-orbit 
$\mO_c$.
This suggests the following definition.

\begin{definition}

Let $D_c(LK_c\backslash\Gr)$ 
be the bounded constructible derived category of  sheaves on $LK_c\backslash\Gr$. 
We set $D_c(K(\calK)\backslash\Gr)$ to be the full subcategory
of $D_c(LK_c\backslash\Gr)$
of complexes constructible with 
respect to the stratification $\calS$.
\end{definition}

We are now ready to state our second main result,
the affine Matsuki correspondence for sheaves.

\begin{thm}[Theorem \ref{AM} below]\label{second result}
There is an equivalence of categories 
\[
\xymatrix{
\Upsilon:D_c(K(\calK)\backslash\Gr)\ar[r]^-\sim &  D_!(LG_\bbR\backslash\Gr)
}\]
\end{thm}

\begin{remark}
The category $D_c(K(\calK)\backslash\Gr)$, respectively $D_!(LG_\bbR\backslash\Gr)$,
is generated by standard, respectively costandard, objects,
and 
the equivalence $\Upsilon$ maps standard objects to 
costandard objects.
\end{remark}

\begin{remark}
Our first two main results, Theorems~\ref{first result} and~\ref{second result}, admit natural generalizations from the affine Grassmannian to any (partial) affine flag manifold. (Beyond  the affine Grassmannian, we do not know whether the orbit posets have as simple a parameterization as recounted in Sect.~\ref{orbits}.) The addition of Iwahori and other level structures offers further interesting geometry, especially in families as we vary their place along the curve, and we postpone details to the sequel paper~\cite{CN1}.
\end{remark}

\begin{remark}
In place of the standard conjugation
  $z\mapsto \bar z$ of the projective line $\bbP^1$, we could take the ``antipodal" 
  conjugation $z\mapsto -\bar z^{-1}$ whose real points are empty. In place of 
    diagram~\eqref{eq:aff diamond}, we would find the diagram of ``twisted" loop groups
\beq\label{eq:tw aff diamond}
\xymatrix{&G(\calK)&\\
G_\theta(\calK)\ar[ru]\ar[ru]&L_\eta G \ar[u]&L_{\eta_c} G\ar[lu]\\
&L_{\eta_c} K\ar[lu]\ar[u]\ar[ru]&}
\eeq
Here $G_\theta(\calK)$ is the subgroup of maps $\gamma:D^\times \to G$ such that  $\gamma(-z) = \theta(\gamma(z))$ where $\theta:G\to G$ is the involution that cuts out $K\subset G$. Similarly, $L_{\eta} G$, respectively $L_{\eta_c} G$,
is the subgroup of  maps $\gamma:\bC^\times \to G$ such that $\gamma(-\bar z^{-1}) = \eta(\gamma(z))$,  
respectively $\gamma(-\bar z^{-1}) = \eta_c(\gamma(z))$, 
where $\eta:G\to G$, respectively  $\eta_c:G\to G$,   is the conjugation that cuts out $G_\bbR\subset G$,
respectively $G_c\subset G$.
Lastly, $L_{\eta_c} K$ is  the subgroup of  maps $\gamma:\bC^\times \to K$ such that $\gamma(-\bar z^{-1}) = \eta_c(\gamma(z))$,  
and  is the intersection of any two of the above three subgroups.

We expect statements analogous to our first two main results, Theorems~\ref{first result} and~\ref{second result}, along with the moduli interpretations that underlie them, to hold with  this setup as well.
\end{remark}

\subsection{Relation to Schubert geometry}\label{Relation to Schubert geometry}

In this section, we state our third main result, a remarkable connection between
the Matsuki correspondence for  the affine Grassmannian  and real  Schubert geometry.

Let $\calO_\bbR = \bbR[[t]]$, and $\calK_\bbR = \bbR((t))$.
Consider the real affine Grassmannian
$\Gr_\bbR=G_\bbR(\calO_\bbR)/G_\bbR(\calO_\bbR)$.
The group $G_\bbR(\calO_\bbR)$, respectively 
$G_\bbR(\bbR[t^{-1}])$, acts on $\Gr_\bbR$ with finite-dimensional, respectively finite-codimensional, orbits.

Recall the uniformization of real analytic stacks of \eqref{Moduli int for LG_R}:
\beq
LG_\bbR\backslash\Gr \is \Bun_\bbG(\bbP^1)_{\bbR,\alpha_0}
\eeq
In its construction, we view $\Gr$ as based at the (non-real) point  $i\in \bbP^1(\bbC)$.
When we instead focus on the (real) point $0\in \bbP^1(\bbR)$, we obtain an alternative uniformization:
\beq G_\bbR(\bbR[t^{-1}])\backslash\Gr_\bbR\is\Bun_{\bbG}(\bbP^1)_{\bbR,\alpha_0}
\eeq

Let $D_c(G_\bbR(\calO_\bbR)\backslash\Gr_\bbR)$ be the bounded constructible $G_\bbR(\calO_\bbR)$-equivariant  derived category of sheaves on the ind-scheme $\Gr_\bbR$. 
Let $D_c(G_\bbR(\bbR[t^{-1}])\backslash\Gr_\bbR)$ be the bounded constructible derived category of sheaves on
the stack
 $G_\bbR(\bbR[t^{-1}])\backslash\Gr_\bbR$, and set
  $D_!(G_\bbR(\bbR[t^{-1}])\backslash\Gr_\bbR)$ 
to  be the full subcategory of complexes that are extensions by zero off of finite type substacks.

Recall the affine Matsuki correspondence for sheaves of Theorem~\ref{second result}: 
\beq\label{eq:mat recalled}
\xymatrix{
\Upsilon:D_c(K(\calK)\backslash\Gr)\ar[r]^-\sim & D_!(LG_\bbR\backslash\Gr)
}
\eeq
In Proposition~\ref{real RT}, we show the Radon transform provides an analogous equivalence:
\beq
\label{eq:radon}
\xymatrix{
\Upsilon_\bbR:D_c(G_\bbR(\calO_\bbR)\backslash\Gr_\bbR)\ar[r]^-\sim &  D_!(G_\bbR(\bbR[t^{-1}])\backslash\Gr_\bbR)
}
\eeq
The main ingredient in the proof of the equivalence $\Upsilon_\bbR$ is 
the natural $\bbR_{>0}$-action on $\Gr_\bbR$ (as Morse theory is the main ingredient in the proof of Theorem~\ref{second result}).

Our third main result is a nearby cycles equivalence intertwining the Matsuki correspondence of \eqref{eq:mat recalled}
and the Radon transform of \eqref{eq:radon}. 
To state it, recall the quasi-map family $QM^{(2)}(\bbP^1, G, K)_\bbR \to \bbP^1(\bbC)$. Recall as well for the (non-real) point $i\in \bbP^1(\bbC)$,
 the identification of the fiber:
\beq
LK_c\backslash\Gr\is QM^{(2)}(\bbP^1, i, G, K)_{\bbR, \alpha_0}.
\eeq
For the (real) point $0\in \bbP^1(\bbC)$, we have an analogous
identification of the fiber:
\beq
K_c\backslash\Gr_\bbR\is QM^{(2)}(\bbP^1, 0, G, K)_{\bbR, \alpha_0}.
\eeq
Taking nearby cycles in the family descends to a functor
\beq
\begin{xymatrix}{
\Psi:D_c(K(\calK)\backslash\Gr)  \ar[r]& D_c(G_\bbR(\calO_\bbR)\backslash\Gr_\bbR)
}
\end{xymatrix}
\eeq

\begin{thm}[Theorem \ref{equ} and \ref{diagram} below]
\label{third result}
There is a canonical commutative square of equivalences
\beq
\begin{xymatrix}{
D_c(K(\calK)\backslash\Gr)\ar[d]^-{\Upsilon}\ar[r]^-{\Psi}&D_c(G_\bbR(\calO_\bbR)\backslash\Gr_\bbR)\ar[d]^-{\Upsilon_\bbR}\\
D_!(LG_\bbR\backslash\Gr)\ar[r]^-{\Psi_\bbR}&D_!(G_\bbR(\bbR[t^{-1}])\backslash\Gr_\bbR)}
\end{xymatrix}
\eeq
where the equivalence $\Psi_\bbR$ is given by transport along the 
uniformization isomorphisms
\beq
LG_\bbR\backslash\Gr\is \Bun_\bbG(\bbP^1)_{\bbR,\alpha_0} \is G_\bbR(\bbR[t^{-1}])\backslash\Gr_\bbR
\eeq

\end{thm}

\subsection{Further directions}

In this final section of the introduction, we discuss results to appear in sequel papers that build 
on those of the current paper.

\subsubsection{Kostant-Sekiguchi correspondence}\label{KS}
One of our primary motivations for studying the affine Matsuki correspondence is its application to the
Kostant-Sekiguchi correspondence. This will be the subject of the sequel paper~\cite{CN2} which we briefly survey here. 

Let $\fg$, $\fg_\bbR$ and $\frak k$
be the Lie algebras of 
$G$, $G_\bbR$ and $K$ respectively, and introduce the Cartan decomposition  $\fg=\frak k\oplus\fp$.

Let $\calN \subset \fg$ be the nilpotent cone, and introduce
the real nilpotent cone $\calN_\bbR=\calN\cap\fg_\bbR$, and the 
$\fp$-nilpotent cone $\calN_\fp=\calN\cap\fp$. The adjoint actions of
$G$, $G_\bbR$ and $K$ preserve $\calN$,  $\calN_\bbR$  and $\calN_\fp$
respectively and have finitely many orbits.

The celebrated Kostant-Sekiguchi correspondence~\cite{S}
is a poset isomorphism 
\beq
|K\backslash\calN_\fp|\longleftrightarrow|G_\bbR\backslash\calN_\bbR|
\eeq
between sets of $K$-orbits and $G_\bbR$-orbits, each ordered with respect to orbit closure. 

Let $\mO \subset \calN_\fp$ be a $K$-orbit and $\mO'\subset \calN_\bbR $ the 
corresponding $G_\bbR$-orbit under the Kostant-Sekiguchi correspondence.
The papers~\cite{SV,V} establish the following remarkable result: 
\beq\label{nilpotent orbits}
\begin{gathered}
\text{There is a real analytic $K_c$-equivariant isomorphism 
$\mO\is\mO'$.}
\end{gathered}
\eeq

Now recall from Theorem~\ref{third result} the  equivalence 
\beq
\begin{xymatrix}{
\Psi:D_c(K(\calK)\backslash\Gr)  \ar[r]^-\sim & D_c(G_\bbR(\calO_\bbR)\backslash\Gr_\bbR)
}
\end{xymatrix}
\eeq
given by nearby cycles in the quasi-map family $QM^{(2)}(\bbP^1, G, K)_\bbR \to \bbP^1(\bbC)$. 

In the sequel paper~\cite{CN2}, we show the quasi-map family 
in fact admits a topological trivialization providing a $K_c$-equivariant homeomorphism
\beq\label{eq: homeo}
\begin{xymatrix}{
\Omega K_c \backslash\Gr  \ar[r]^-\sim & \Gr_\bbR
}
\end{xymatrix}
\eeq

We also introduce another quasi-map family  $QM^{(2)}(\calX, G, K)_\bbR \to \bbA^1(\bbC)$ induced by a degeneration $\calX\to \bbA^1$  of the projective line $\bbP^1$ to a nodal curve $\bbP^1 \vee \bbP^1$.
Its restriction to an open subspace can be modeled by the flat family of quotients
\beq\label{eq: quots family}
\begin{xymatrix}{
\Omega K_c \backslash\Gr  \ar@{~>}[r] & K(\bbC[t^{-1}])_1\backslash \Gr
}
\end{xymatrix}
\eeq
where $K(\bbC[t^{-1}])_1\subset K(\bbC[t^{-1}])$ is the kernel of evaluation at $\infty$.
 
Now we arrive at our intended application to the 
Kostant-Sekiguchi correspondence.
It is well-known~\cite{L} that when $G$ is of type $A$, the nilpotent cone $\calN$ embeds in the affine Grassmannian $\Gr$.
Furthermore, this induces  embeddings of the real and $\fp$-nilpotent cones:
\beq
\begin{xymatrix}{
\calN_\bbR \subset \Gr_\bbR & \calN_\fp \subset K(\bbC[t^{-1}])_1\backslash \Gr
}
\end{xymatrix}
\eeq
Applying the geometry of \eqref{eq: homeo} and \eqref{eq: quots family}, we obtain in type A a lift of
the Kostant-Sekiguchi correspondence:
\beq
\begin{gathered}\label{eq: kos-sek homeo}
\text{There is a $K_c$-equivariant
orbit-preserving homeomorphism 
$\calN_\fp\is\calN_\bbR$.}
\end{gathered}
\eeq
Thanks to the compatibility of  our constructions with inner automorphisms and Cartan involutions, we are in fact able to deduce  \eqref{eq: kos-sek homeo} for all classical types from the case of type $A$.

\subsubsection{Comparison of dual groups}

The paper  \cite{N1} associates to each real form $G_\bbR\subset G$ 
a reductive subgroup $H_{real}^\vee \subset G^\vee$ of the  dual group.\footnote{While the notation suggests
regarding $H_{real}^\vee$
 itself as a  dual group, we do not know of a concrete role for its dual group.} 
The construction of $H_{real}^\vee$ is via Tannakian formalism: its tensor category of finite-dimensional representations 
$\textup{Rep}(H_{real}^\vee)$ is realized as a certain full subcategory $Q_\bbR \subset D_c(G_\bbR(\calO_\bbR)\backslash \Gr_\bbR)$ of perverse sheaves on the real affine Grassmannian $\Gr_\bbR$. 

On the other hand, the papers~\cite{GN1, GN2} associate to every spherical subgroup $K\subset G$ a reductive
subgroup $H_{sph}^\vee \subset G^\vee$ of the dual group. Again,
the construction of $H_{sph}^\vee$ is via Tannakian formalism: its tensor category of finite-dimensional representations 
$\textup{Rep}(H_{sph}^\vee)$  can be realized as a certain full subcategory $Q_K \subset D_c(K(\calK)\backslash \Gr)$ of perverse sheaves where as usual we understand $D_c(K(\calK)\backslash \Gr)$ as complexes on a quasi-map space with target $G/K$.

When $K\subset G$ is the symmetric subgroup of a real form $G_\bbR\subset G$, we may ask whether the above two subgroups $H^\vee_{real}, H^\vee_{sph}\subset G^\vee$ coincide. Their inclusions into $G^\vee$ are determined under Tannakian
formalism by the respective tensor functors of restriction  
 $\textup{Rep}(G^\vee) \to \textup{Rep}(H_{real}^\vee)$, $\textup{Rep}(G^\vee) \to \textup{Rep}(H_{sph}^\vee)$.
By construction, these functors correspond to functors from the Satake category $\textup{Sat}_G \subset D_c(G(\calO)\backslash \Gr)$  to the respective categories $Q_\bbR, Q_K$.

Details of the following compatibility will be given in \cite{CN1}.
Recall from Theorem~\ref{third result} the  nearby cycles equivalence
\beq
\begin{xymatrix}{
\Psi:D_c(K(\calK)\backslash\Gr)  \ar[r]^-\sim & D_c(G_\bbR(\calO_\bbR)\backslash\Gr_\bbR)
}
\end{xymatrix}
\eeq
given by nearby cycles in the quasi-map family $QM^{(2)}(\bbP^1, G, K)_\bbR \to \bbP^1(\bbC)$

\begin{thm}
The functor $\Psi$ restricts to the horizontal tensor equivalence in a commutative diagram of tensor functors
\beq
\begin{xymatrix}{
& \textup{Sat}_G \ar[dl]\ar[dr]&\\
Q_K  \ar[rr]^-\sim && Q_\bbR
}
\end{xymatrix}
\eeq
\end{thm}

\subsubsection{Representation theory of real loop groups}

Let us briefly sketch here another motivation for the results of this paper.

Recall the usual Matsuki correspondence for sheaves~\cite{MUV} intertwines
the Beilinson-Bernstein localization~\cite{BB} of Harish Chandra $(\frakg, K)$-modules with the Kashiwara-Schmid localization~\cite{KS} of (infinitesimal classes of) admissible representations of $G_\bbR$.  
We seek an analogous geometric approach to the representation theory of  real loop groups.

For a compact real form $G_c \subset G$, the  positive energy representation theory of the real loop group $LG_c$ 
offers analogues of the many beautiful geometric and combinatorial aspects of the representation  theory of $G_c$ itself.
For example, there is a Borel-Weil-Bott construction of irreducibles, a Weyl-Kac character formula
via localization, BGG resolutions
via Schubert geometry, among other now standard results~\cite{PS}. Furthermore, there is the celebrated fusion structure on level $k$ representations as
organized by rational conformal field theory.

In comparison,    
 for a non-compact real form $G_\bbR \subset G$, 
relatively little representation theory of the real loop group $LG_\bbR$  has been developed. 
 This is so even though there are longstanding motivations coming from 
 Chern-Simons theory for non-compact gauge group. 
  With the results of this paper in hand, one might hope to engineer  a representation theory of $LG_\bbR$
 by suitably ``globalizing" $K(\calK)$-equivariant and  $LG_\bbR$-equivariant sheaves on affine flag varieties. 
 Unfortunately, traditional global sections constructions,
following Beilinson-Bernstein or Kashiwara-Schmid,
   appear either to produce no new representations  or to lead to  semi-infinite pathologies. 
  
  We expect a theory of admissible representations of $LG_\bbR$
   to fit within the framework
 of representations on pro-vector spaces. More specifically, we conjecture the derived categories of equivariant sheaves in the affine Matsuki
 correspondence are  equivalent to categories of $LG_\bbR$-representations that are admissible in the sense that they are pro-objects
 in the positive energy representations of $LK_c$. 
 We plan to approach this in future work.

\subsection{Organization}
In Section \ref{orbits}, we recall the parametrization of 
$K(\calK)$-orbits and $LG_\bbR$-orbits on the affine Grassmannian
and the statement of the affine Matsuki correspondence.
We also establish some 
geometric properties for those orbits. In Section \ref{Morse flow}, we construct the 
Matsuki flow on the affine Grassmannian and we give a 
Morse-theoretic interpretation and refinement of the Matsuki correspondence for 
the affine Grassmannian. In Section \ref{Real BD}, we study real forms 
of Beilinson-Drinfeld Grassmannians. In Sections \ref{uniform} and \ref{QMaps}, we study moduli stacks of real bundles on $\bbP^1$ and quasi-maps. We study uniformizations for
those moduli stacks and use them to provide moduli interpretations for various 
quotients of the affine Grassmannian by subgroups of the loop group. 
In Section \ref{Affine Matsuki}, we prove the affine Matsuki correspondence for sheaves (Theorem \ref{second result}). 
In Section \ref{nearby cycles and Radon TF}, we prove the 
nearby cycles equivalences and 
the Radon transform equivalence (Theorem \ref{third result}).
In Section \ref{s:hecke}, we study the compatibility of Hecke actions.
In Appendix \ref{Real stacks}, we discuss real analytic stacks and categories of sheaves 
on real analytic stacks.

\subsection{Acknowledgements} 
T.H. Chen would like to thank the Max Planck Institute for 
Mathematics for support, hospitality, and a nice research environment.
D. Nadler would like to thank the Miller Institute for its inspiring environment.
The research of
T.H. Chen is supported by NSF grant DMS-1702337
and that of D. Nadler  by NSF grant DMS-1502178.

\section{$K(\calK)$ and $LG_\bbR$-orbits on $\Gr$}\label{orbits}
In this section we study 
$K(\calK)$ and $LG_\bbR$-orbits on the affine Grassmannian $\Gr$.

\subsection{Loop groups}
The real forms $G_\mathbb R$ and $G_c$ of $G$ correspond to 
anti-holomorphic involutions
$\eta$ and $\eta_c$. The involutions 
$\eta$ and $\eta_c$ commutes with each other and 
$\theta:=\eta\eta_c=\eta_c\eta$ is an involution of $G$. We have $K=G^\theta$, 
$G_\mathbb R=G^\eta$, and $G_c=G^{\eta_c}$.
We fix a maximal split tours $S_\mathbb R\subset G_\mathbb R$ and 
a maximal torus $T_\mathbb R$ such that $S_\mathbb R\subset T_\mathbb R$.
We write $S$ and $T$ for the complexification of $S_\mathbb R$ and $T_\mathbb R$.
We denote by $\Lambda_T$  the lattice of coweights of $T$ and 
$\Lambda_S$ the lattice of real coweights. We write $\Lambda_T^+$
the set of dominant coweight with respect to the Borel subgroup $B$ and 
define $\Lambda_S^+:=\Lambda_S\cap\Lambda_T^+$.
For any $\lambda\in\Lambda_T$ we define $\eta(\lambda)\in\Lambda_T$ as
\[\eta(\lambda):\bC^\times\stackrel{c}\ra\bC^\times\stackrel{\lambda}\ra T\stackrel{\eta}\ra T,\]
where $c$ is the complex conjugation of $\bC^\times$ with respect to $\mathbb R^\times$.
The assignment $\lambda\ra\eta(\lambda)$ defines an involution on $\Lambda_T$, which 
we denote by $\eta$,
and $\Lambda_S$ is the fixed points of $\eta$.

Let 
$LG:=G(\bC[t,t^{-1}])$ be the (polynomial) loop group associated to $G$. We define 
the following involutions on $LG$:
for any $(\gamma:\bC^\times\ra G)\in LG$ we set
\[\eta^\tau(\gamma):\bC^\times\stackrel{\tau}\ra\bC^\times\stackrel{c}\ra\bC^\times\stackrel{\gamma} \ra G\stackrel{\eta}\ra G\]
\[\eta^\tau_c(\gamma):\bC^\times\stackrel{\tau}\ra\bC^\times\stackrel{c}\ra\bC^\times\stackrel{\gamma} \ra G\stackrel{\eta_c}\ra G.\]
Here $\tau(x)=x^{-1}$ is the the inverse map.
Denote by $\calK=\bC((t))$ and $\mO=\bC[[t]]$.
We have the following diagram
\[
\xymatrix{&G(\calK)&\\
K(\calK)\ar[ru]^\theta\ar[ru]&LG_\mathbb R\ar[u]^{\eta^\tau}&LG_c\ar[lu]_{\eta_c^\tau}\\
&LK_c\ar[lu]\ar[u]\ar[ru]&}
\]
Here $LG_\mathbb R$ and $LG_c$ are the fixed points subgroups of the 
involutions
$\eta^\tau$ and $\eta^{\tau}_c$ on $LG$ respectively.
Equivalently, $LG_\mathbb R$ (resp. $LG_c$) is the 
subgroup of $LG$ consisting of maps that take the 
unit circle $S^1\subset\bC$ to $G_\bbR$ (resp. $G_c$).
We define the based loop group $\Omega G_c$ to be the subgroup of 
$LG_c$ consisting of maps that take $1\in S^1$ to $e\in G_c$.

\subsection{The based loop spaces $\Omega X_c$}\label{based for X_c}
We define $X\subset G$
(resp. $X_c\subset G_c$) to be the identity component of the 
fixed point subspace of the involution $\tilde\theta=\theta^{-1}$ on $G$ (resp. $G_c$).
The map $\pi:G\to X, \pi(g)=\tilde\theta(g)g$ induces a $G$-equivariant isomorphism 
$K\backslash G\is X$ (resp. $G_c$-equivariant isomorphism $K_c\backslash G_c\is X_c$).
We define the loop space $LX_c$ be the subspace of 
$LG_c$ consisting of 
maps that takes $S^1$ into $X_c$. We define the based loop space 
$\Omega X_c$ to be the subspace of 
$LX_c$ consisting of maps that takes $1\in S^1$ to $e\in X_c$.

\subsection{Real affine Grassmannians}\label{real affine Gr}
We recall results from \cite{N1} about the real affine Grassmannian.
Let $\Gr:=G(\calK)/G(\mathcal O)$ be the affine Grassmannian for $G$
and $\Gr_\mathbb R:=G_\mathbb R(\calK_\mathbb R)/G_\mathbb R(\mO_\mathbb R)$
be the real affine Grassmannian. For any $\lambda\in\Lambda_T^+$ we denote by 
$S^\lambda$ and $T^\lambda$ the $G(\mO)$ and $G(\bC[t^{-1}])$-orbit 
of $t^\lambda\in\Gr$. 
The orbits $S^\lambda$ and $T^\mu$ on $\Gr$ are transversal and 
the intersection $C^\lambda=S^\lambda\cap T^\lambda$ is isomorphic to 
the flag manifold $G/P^\lambda$ where the parabolic subgroup $P^\lambda$ is the 
stabilizer of $\lambda$. 
The affine Grassmannian $\Gr$ is the disjoint union of the orbits $S^\lambda$ (resp. $T^\lambda$)
for $\lambda\in\Lambda_T^+$
\[\Gr=\bigsqcup_{\lambda\in\Lambda_T^+} S^\lambda\ \ \ \ (resp.\ \ \ 
\Gr=\bigsqcup_{\lambda\in\Lambda_T^+} T^\lambda)\]
and we have 
\[\overline S^\lambda=\bigsqcup_{\mu\leq\lambda} S^\mu\ \ \ \ (resp.\ \ \ 
T^\lambda=\bigsqcup_{\lambda\leq\mu} T^\mu).\]
The intersection of $S^\lambda$ (resp. $T^\lambda$) with $\Gr_\bbR$ is nonempty if and only if 
$\lambda\in\Lambda_S^+$ and we write $S_\bbR^\lambda$ (resp. $T_\bbR^\lambda$), $\lambda\in\lambda_S^+$ for the intersection. We define $C_\bbR^\lambda$ to be the intersection of 
$S_\bbR^\lambda$ and $T_\bbR^\lambda$. $S_\bbR^\lambda$ (resp. $T_\bbR^\lambda$)
is equal to the $G_\bbR(\mO_\bbR)$-orbit (resp. $G_\bbR(\bbR[t^{-1}])$-orbit) of $t^\lambda$ and 
$C_\bbR^\lambda$ is isomorphic to the real flag manifold 
$G_\bbR/P_\bbR^\lambda$ where the parabolic subgroup $P_\bbR^\lambda\subset G_\bbR$ is the 
stabilizer of $\lambda$. 
The real affine Grassmannian $\Gr$ is the disjoint union of the orbits $S_\bbR^\lambda$ (resp. $T_\bbR^\lambda$)
for $\lambda\in\Lambda_T^+$
\[\Gr_\bbR=\bigsqcup_{\lambda\in\Lambda_S^+} S_\bbR^\lambda\ \ \ \ (resp.\ \ \ 
\Gr_\bbR=\bigsqcup_{\lambda\in\Lambda_S^+} T_\bbR^\lambda)\]
and we have 
\[\overline S_\bbR^\lambda=\bigsqcup_{\mu\leq\lambda} S_\bbR^\mu\ \ \ \ (resp.\ \ \ 
T_\bbR^\lambda=\bigsqcup_{\lambda\leq\mu} T_\bbR^\mu).\]

\subsection{The energy flow on $\Omega G_c$}\label{energy flow}
We recall the construction of energy flow on $\Omega G_c$ following 
\cite[Section 8.9]{PS}.
For any $\gamma\in LG_c$ and $v\in T_\gamma LG_c$ we denote by 
$\gamma^{-1}v\in L\fg_c$ (resp. $v\gamma^{-1}\in L\fg_c$) the
image of $v\in T_\gamma LG_c$ under the 
isomorphism $T_\gamma LG_c\is T_eLG_c\is L\fg_c$
induced by the left action (resp. right action). 

Fix a $G_c$-invariant metric $\langle,\rangle$ on $\fg_c$.
Observe that the formula 
\[\omega(v,w):=\int_{S^1} \langle(\gamma^{-1}v)',\gamma^{-1}w\rangle d\theta
\]
defines a left invariant symplectic form on $T_\gamma\Omega G_c$. 
According to \cite[Theorem 8.6.2]{PS}, the composition 
$\Omega G_c\to G(\mathcal K)\to\Gr$ defines a diffeomorphism
\[
\Omega G_c\is\Gr.
\]
Let $J_\gamma$ be the automorphism of $T_\gamma\Omega G_c$ which 
corresponds to multiplication by $i$ in terms of the complex structure on 
$\Gr$. The formula 
$g(v,w)=\omega(v,J_\gamma w)$
defines a positive inner product on $T_\gamma\Omega G_c$ and 
the K$\ddot{\on{a}}$hler form on $T_\gamma\Omega G_c$ is 
given by $g(v,w)+i\omega(v,w)$.
Finally, for any smooth function $F:\Omega G_c\ra\bbR$ there corresponds 
so-called Hamiltonian vector field $R(\gamma)$ and 
gradient vector field $\nabla F(\gamma)$ on $\Omega G_c$ characterized by 
\[\omega(R(\gamma),v)=dF(\gamma)(u),\ g(\nabla F(\gamma),u)=
dF(\gamma)(u).\]
Consider the energy function on $\Omega G_c$:
\beq\label{energy function}
E:\Omega G_c\ra\bbR,\ \gamma\ra (\gamma',\gamma')_\gamma=\int_{S^1}\langle\gamma^{-1}\gamma',\gamma^{-1}\gamma'\rangle d\theta.
\eeq

We have the following well-known facts.
\begin{prop}\cite{P,PS}\label{PS}
\begin{enumerate}
\item The Hamiltonian vector field of $E$ 
is equal to the vector field induced by the rotation flow 
$\gamma_a(t)=\gamma(t+a)\gamma(a)^{-1}$ and 
is given by 
$\gamma\ra R(\gamma)=\gamma'-\gamma\gamma'(0)$. The gradient vector field of $E$ is equal to 
$\nabla E=-J\circ R$. 
\item 
The critical locus $\nabla E=0$ is the disjoint union 
$\bigsqcup_{\lambda\in\Lambda_T^+}C^\lambda$ of $G_c$-orbits of 
$\lambda\in\Omega G_c$.
\item
The gradient flow $\psi_t$ of $\nabla E$  
preserves the orbits $S^\lambda$ and $T^\lambda$. For each 
critical orbit $C^\lambda$, we have 
 \[S^\lambda=\{\gamma\in\Omega G_c|\underset{t\to\infty}{\lim}\psi_t(\gamma)\in C^\lambda\} \ \ \ \ T^\lambda=\{\gamma\in\Omega G_c|\underset{t\ra-\infty}\lim_{}\psi_t(\gamma)\in C^\lambda\}.\]
That is $S^\lambda$ and $T^\lambda$ are the stable and unstable 
manifold of $C^\lambda$.
\end{enumerate}
\end{prop}

\subsection{Component groups of $\Gr_\bbR$}
The diffeomorphism 
$\Omega G_c\is\Gr$
 induces a diffeomorphism on the $\eta$-fixed points 
$(\Omega G_c)^\eta\is\Gr_\bbR$.
Let $\Omega_{\on{top}} G_c$ and $\Omega_{\on{top}} X_c$
be the (topological) based loop spaces of $G_c$ and $X_c$.
Note that for any $\gamma\in(\Omega_{\on{top}} G_c)^\eta$ 
we have $\gamma(-1)\in K_c$ and the map  
$e^{i\theta}\to\gamma'(e^{i\theta}):=\pi\circ\gamma(e^{i\theta/2})$
defines a map
$\gamma':S^1\to X_c$, that is, $\gamma'\in\Omega_{\on{top}} X_c$.
According to \cite{M} the composition 
\[q:(\Omega G_c)^\eta\to(\Omega_{\on{top}} G_c)^\eta\to \Omega_{\on{top}} X_c\]
is a homotopic equivalence where the first map is the natural inclusion and the second map
is given by $\gamma\to\gamma'$. 
Since $X_c$ is a deformation retract of $X$, 
the map $q$ induces 
an isomorphism 
\beq\label{components of real Gr}
\pi_0(\Gr_\bbR)\is\pi_0((\Omega G_c)^\eta)\is\pi_0( \Omega_{\on{top}} X_c)\is\pi_1(X_c)\is
\pi_1(X).
\eeq

\subsection{Parametrization of $K(\calK)$ and $LG_\bbR$-orbits}
We recall  
results from \cite{N2} about the parametrization of $K(\calK)$ and $LG_\bbR$-orbits on 
$\Gr$.
Consider the following diagram
\[\pi_1(G)\stackrel{\pi_*}\ra\pi_1(X)\stackrel{[-]}\la\Lambda_S^+,\]
where the first map is that induced by the map $\pi:G\ra X$ and 
the second map $[-]$ assigns to a loop its homotopy class.

\begin{definition} 
We define $\calL\subset\Lambda_S^+$ to be the 
inverse image of $\pi_*(\pi_1(G))$ along the map
$[-]$.
\end{definition}
\begin{remark}
If $K$ is connected, then we have $\mL=\Lambda_S^+$.
\end{remark}
\begin{proposition}[\cite{N1}]\label{parametrization}
We have the following.
\begin{enumerate}
\item There is a bijection \[|K(\mathcal K)\backslash\Gr|\longleftrightarrow\mL\]
between $K(\mathcal K)$-orbits on $\Gr$ and $\mL$
characterized by the following properties:
Let $\mO_K^\lambda$ be the $K(\mathcal K)$-orbits corresponding to $\lambda\in\mL$.
Then for any $\gamma\in\mO_K^\lambda$, thought of as an element in 
$\Omega G_c$, satisfies  
$\tilde\theta(\gamma)\gamma\in G(\bC[t])t^\lambda G(\bC[t])$.
In addition, we have 
$\overline\mO_K^\lambda=\bigsqcup_{\mu\leq\lambda}\mO_K^\mu$.

\item There is a bijection \[|LG_\bbR\backslash\Gr|\longleftrightarrow\mL\]
between $LG_\bbR$-orbits on $\Gr$ and $\mL$
characterized by the following property:
Let $\mO_\bbR^\lambda$ be the $LG_\bbR$-orbits corresponding to $\lambda\in\mL$.
Then for any $\gamma\in\mO_\bbR^\lambda$, thought of as an element in 
$\Omega G_c$, satisfies
$\tilde\eta^\tau(\gamma)\gamma\in G(\bC[t^{-1}])t^\lambda G(\bC[t])$.
In addition, we have 
$\overline\mO_\bbR^\lambda=\bigsqcup_{\lambda\leq\mu}\mO_\bbR^\mu$.

\item The correspondence 
\beq\label{AMS correspondence}
|K(\mathcal K)\backslash\Gr|\longleftrightarrow |LG_\bbR\backslash\Gr|,\ \ \mO_K^\lambda
\longleftrightarrow\mO_\bbR^\lambda
\eeq
provides an order-reversing 
isomorphism from the poset $|K(\mathcal K)\backslash\Gr|$ to the poset $|LG_\bbR\backslash\Gr|$
(with respect to the closure ordering). In addition, for each 
$K(\mathcal K)$-orbit $\mO_K^\lambda$, $\mO_\bbR^\lambda$ is the 
unique $LG_\bbR$-orbit such that \[\mO_c^\lambda:=\mO_K^\lambda\cap\mO_\bbR^\lambda\]
is a single $LK_c$-orbit.

\end{enumerate}
\end{proposition}

We will call~\eqref{AMS correspondence} the Affine Matsuki correspondence.

\begin{corollary}\label{rotation}
The $K(\calK)$-orbits and $LG_\bbR$-orbits are stable under the 
rotation flow $\gamma_a(t)$ (see Proposition \ref{PS}).
\end{corollary}
\begin{proof}
We give a proof for the case of $K(\calK)$-orbits.
The proof for the $LG_\bbR$-orbits is similar.
Let $\mO_K^\lambda$ be a $K(\calK)$-orbit and let 
$\gamma=\gamma(t)\in\mO_K^\lambda$. 
By Proposition \ref{parametrization}, we need to show that 
$\tilde\theta(\gamma_a)\gamma_a\in G(\bC[t])t^\lambda G(\bC[t])$.
A direct computation shows that 
$\tilde\theta(\gamma_a)\gamma_a=\theta(\gamma(a))\tilde\theta(\gamma(t+a))\gamma(t+a)\gamma(a)^{-1}$. Note that  $\tilde\theta(\gamma(t+a))\gamma(t+a)\in G(\bC[t])t^\lambda G(\bC[t])$ as 
$\gamma(t)\in\mO_K^\lambda$, the desired claim follows.

\end{proof}

\quash{
The natural inclusion 
$S\to X$ induces an imbedding 
$\Lambda_S^+\hookrightarrow X(\calK)$.

\begin{proposition}\label{parametrization of K(K)-orbits}
The embedding $\Lambda_S^+\hookrightarrow X(\calK)$ 
induces an bijection 
\[\Lambda_S^+\longleftrightarrow |X(\mathcal K)/G(\mO)|\]
between $\Lambda_S^+$ and $G(\mO)$-orbits in $X(\mathcal K)$
and, under the above bijection, 
 the image of the natural inclusion \[|K(\calK)\backslash\Gr|\to| X(\mathcal K)/G(\mO)|\] is the subset $\mL\subset\Lambda_S^+$.
\end{proposition}
}

\subsection{Geometry of $K(\calK)$ and $LG_\bbR$-orbits}
For $\lambda\in\Lambda_S^+$, we define $P^\lambda\subset\Omega X_c$ to be the intersection 
of $\Omega X_c$ with the orbit $S^\lambda\subset\Omega G_c\is\Gr$, and we define 
$Q^\lambda\subset\Omega X_c$ to be the intersection 
of $\Omega X_c$ with the orbit $T^\lambda\subset\Omega G_c\is\Gr$. We define 
$B^\lambda$ to be the intersection 
of $\Omega X_c$ with $C^\lambda\subset\Omega G_c\is\Gr$.
The projection map $\pi:G\to X, g\to \tilde\theta(g)g$ induces 
a projection $\pi:\Omega G_c\to\Omega X_c$. 

\begin{lemma}
$P^\lambda$ is a vector bundle over $B^\lambda$.
\end{lemma}
\begin{proof}
By \cite[Proposition 6.3]{N1},
the restriction of the energy function $E$ to 
$P^\lambda$ is Bott-Morse and $B^\lambda$ is the 
only critical manifold. The lemma follows.
\end{proof}

We define $\Omega X_c^0$ be the union of components 
of $\Omega X_c$ in 
$\pi_*(\pi_1(G))\subset\pi_1(X)=\pi_0(\Omega X_c)$.

\begin{lemma}\label{component X_c}
We have $\Omega X_c^0=\bigcup_{\lambda\in\mL} P^\lambda$.
\end{lemma}
\begin{proof}
Let $\lambda\in\Lambda_S^+$.
It suffices to show that 
$P^\lambda\subset \Omega X_c^0$ if and only if $\lambda\in\mL$.
We have $t^\lambda\in B^\lambda$ and it follows from the definition 
of the map $[-]:\Lambda_S^+\to\pi_1(X)=\pi_0(\Omega X_c)$
that 
$t^\lambda$ lies in the 
component of 
$\Omega X_c^0$
corresponding to $[\lambda]\in\pi_0(\Omega X_c)$ (here $[\lambda]$ is the image of 
$\lambda$ under $[-]$). 
It implies $t^\lambda\in\Omega X_c^0$ if and only if $\lambda\in\mL$.
Since 
$B^\lambda=K_c\cdot t^\lambda$ and $\pi:G\to X$ is $K$-equivariant, 
it implies $B^\lambda\subset\Omega X_c^0$ if and only if $\lambda\in\mL$.
Finally, since $P^\lambda$ is a vector bundle over $B^\lambda$ 
we conclude that $P^\lambda\subset\Omega X_c^0$ if and only if $\lambda\in\mL$.
The lemma follows.

\end{proof}

\begin{proposition}\label{torsors}
We have the following.
\begin{enumerate}
\item
The projection $\pi:\Omega G_c\to\Omega X_c$ maps $\mO_K^\lambda$ into 
$P^\lambda$ and the resulting map 
$\mO_K^\lambda\to P^\lambda$ is a principal $\Omega K_c$-bundle over $P^\lambda$.
\item
The projection $\pi:\Omega G_c\to\Omega X_c$ maps $\mO_\bbR^\lambda$ into 
$Q^\lambda$ and the resulting map 
$\mO_\bbR^\lambda\to Q^\lambda$ is a principal $\Omega K_c$-bundle over $Q^\lambda$.

\item 
We have $\pi(\Omega G_c)=\Omega X_c^0$ and the resulting map 
$\pi:\Omega G_c\to \Omega X_c^0$ is a principal $\Omega K_c$-bundle over $\Omega X_c^0$.
\end{enumerate}
\end{proposition}
\begin{proof}
Fix $\lambda\in\mL$.
Proposition \ref{parametrization} together with the fact that 
$\tilde\theta=\tilde\eta^\tau$ on $\Omega G_c$
imply  
$\pi(\mO_K^\lambda)\subset P^\lambda$ and $\pi(\mO_\bbR^\lambda)\subset Q^\lambda$.
Note that $P^\lambda$ is a vector bundle over $B^\lambda$ and 
$\pi(C^\lambda)=B^\lambda$ as $\pi$ is $LK_c$-equivariant and 
$LK_c$ (resp. $K_c$) acts transitively on 
$C^\lambda$ (resp. $B^\lambda$). Thus the image $\pi(\mO_K^\lambda)$ meets every 
connected component of $P^\lambda$ and, by \cite[Proposition 6.4]{N2}, we have 
$P^\lambda\subset\pi(\Omega G_c)$. 
It implies 
$\pi(\mO_K^\lambda)=P^\lambda$ and 
part (1) follows. 
For part (2) 
we observe that  $Q^\lambda=\bigcup_{\lambda\leq\mu,\mu\in\Lambda_S^+} Q^\lambda\cap P^\mu$. 
Since $B^\lambda=Q^\lambda\cap P^\lambda$ is in the closure of $Q^\lambda\cap P^\mu$,
Lemma \ref{component X_c} implies 
$Q^\lambda=\bigcup_{\lambda\leq\mu,\mu\in\mL} Q^\lambda\cap P^\mu$ and 
part (1) implies $Q^\lambda\subset\pi(\Omega G_c)$, hence 
$Q^\lambda=\pi(\mO_{\bbR}^\lambda)$. Part (2) follow. 
Part (3) follows from part (1) and Lemma \ref{component X_c}.

\end{proof}

\begin{corollary}\label{transversal}
$K(\calK)$ and $LG_\bbR$-orbits on $\Gr$ are transversal. 
\end{corollary}
\begin{proof}
By Proposition \ref{torsors}, it suffices to show that the strata
$P^\lambda$ and $Q^\mu$ in $\Omega X_c$ are transversal. 
This follows from the fact that the orbits 
$S^\lambda$ and $T^\lambda$ 
on $\Omega G_c$ 
are transversal and both $S^\lambda, T^\lambda$ are invariant under the 
involution $\tilde\theta$ on $\Omega G_c$ as 
$\tilde\theta=\tilde\eta^\tau$ on $\Omega G_c$ and 
$S^\lambda$ (resp. $T^\lambda$) is $\tilde\theta$-invariant 
(resp. $\tilde\eta^\tau$-invariant). 

\end{proof}

\subsection{The components $\Gr_{\bbR}^0$}
We define $\Gr_{\bbR}^0$
be the union of the components of $\Gr_\bbR$ in the image 
$\pi_*(\pi_1(G))\subset \pi_1(X)\stackrel{(\ref{components of real Gr})}=\pi_0(\Gr_\bbR)$.

\begin{lemma}\label{Gr^0_R}
We have $\Gr_{\bbR}^0=\bigcup_{\lambda\in\mL} S_{\bbR}^\lambda$.
\end{lemma}
\begin{proof}
Let $\lambda\in\Lambda_S^+$.
It suffices to show that 
$S_{\bbR}^\lambda\subset \Gr_{\bbR}^0$ if and only if $\lambda\in\mL$.
We have $t^\lambda\in S_{\bbR}^\lambda$ and it follows from (\ref{components of real Gr}) that 
$t^\lambda$ lies in the 
component of 
$\Gr_\bbR$
corresponding to $[\lambda]\in\pi_0(\Gr_\bbR)=\pi_1(X)$. 
It implies $t^\lambda\in\Gr^0_\bbR$ if and only if $\lambda\in\mL$.
Since 
$G_\bbR/P_\bbR^\lambda=K_c\cdot t^\lambda$ and $\pi:G\to X$ is $K$-equivariant, 
it implies $G_\bbR/P_\bbR^\lambda\subset\Gr^0_\bbR$ if and only if $\lambda\in\mL$.
Finally, since $S_\bbR^\lambda$ is a vector bundle over $G_\bbR/P_\bbR^\lambda$ 
we conclude that $S_\bbR^\lambda\subset\Gr_\bbR^0$ if and only if $\lambda\in\mL$.
The lemma follows.
\end{proof}

\begin{definition}
We define $D_c(G_\bbR\backslash\Gr_\bbR)$ to be the bounded constructible derived categories of  sheaves on $G_\bbR\backslash\Gr_\bbR$. 
We set $D_c(G_\bbR(\mO_\bbR)\backslash\Gr_\bbR)$ 
to be the full subcategory
of $D_c(G_\bbR\backslash\Gr)$ of complexes constructible with 
respect to the $G_\bbR(\mO_\bbR)$-orbits stratification.
We set  
 $D_c(G_\bbR(\mO_\bbR)\backslash\Gr_\bbR^0)$ be the full subcategory  
of $D_c(G_\bbR(\mO_\bbR)\backslash\Gr_\bbR)$ of complexes supported on the components 
$\Gr_\bbR^0$.

\end{definition}


\section{The Matsuki flow}\label{Morse flow}
In this section we construct a Morse flow on the 
affine Grassmannian, called the Matsuki flow, and we use it to give a 
Morse-theoretic interpretation and refinement of the 
affine 
Matsuki correspondence.

\subsection{The Matsuki flow on $\Gr$}\label{Matsuki flow}
The Cartan decomposition $\fg_\bbR=\frak k_\bbR\oplus\fp_\bbR$ 
induces a decomposition of
$\fg_c=\frak k_c\oplus i\fp_\bbR$, $\fg_\bbR=\frak k_c\oplus\fp_\bbR$
and the corresponding loop algebra 
$L\fg=L\frak k\oplus L\fp$, $L\fg_c=L\frak k_c\oplus L(i\fp_\bbR)$,
$L\fg_\bbR=L\frak k_c\oplus L\fp_\bbR$.

Recall the non-degenerate bilinear form $(,)_\gamma$ on 
$T_\gamma LG_c$ 
\[(v_1,v_2)_\gamma:=\int_{S^1}\langle\gamma^{-1}v_1,\gamma^{-1}v_2\rangle d\theta.\]
Let $\gamma\in LG_c$ and $T_\gamma (LK_c\cdot\gamma)\subset T_\gamma LG_c$ be the 
tangent space of the $LK_c$-orbit $LK_c\cdot\gamma$ through $\gamma$. 
The bilinear form above induces an 
orthogonal decomposition 
\[T_\gamma LG_c=T_\gamma LK_c\cdot\gamma\oplus (T_\gamma LK_c\cdot\gamma)^\bot\] and 
for any vector $v\in T_\gamma LG_c$ we write 
$v=v_0\oplus v_1$ where $v_0\in T_\gamma LK_c\cdot\gamma$, 
$v_1\in (T_\gamma LK_c\cdot\gamma)^\bot$.
 Note that we have 
\beq\label{decomp}
\gamma^{-1}v_0\in\on{Ad}_{\gamma^{-1}}L\frak k_c,\ \ 
\gamma^{-1}v_1\in\on{Ad}_{\gamma^{-1}}L(i\fp_\bbR).
\eeq
Recall that the loop group $\Omega G_c$ can be identified with a ``co-adjoint'' orbit
in $LG_c$ via the embedding
\[\Omega G_c\hookrightarrow L\fg_c,\ \gamma\ra\gamma^{-1}\gamma'.\]
Consider the following  functions on $\Omega G_c$
\[E:\Omega G_c\ra\bbR,\ \gamma\ra (\gamma',\gamma')_\gamma=\int_{S^1}\langle\gamma^{-1}\gamma',\gamma^{-1}\gamma'\rangle d\theta,\] 
\[E_0:\Omega G_c\ra\bbR,\ \gamma\ra
(\gamma'_0,\gamma'_0)_\gamma=
\int_{S^1}\langle\gamma^{-1}\gamma'_0,\gamma^{-1}\gamma'_0\rangle d\theta,\]
\[E_1:\Omega G_c\ra\bbR,\ \gamma\ra
(\gamma'_1,\gamma'_1)_\gamma=
\int_{S^1}\langle\gamma^{-1}\gamma'_1,\gamma^{-1}\gamma'_1\rangle d\theta.\]

\quash{
\begin{lemma}
We have  
$E=E_0+E_1$ and $E_1$ is constant on $LK_c$-orbits on 
$\Omega G_c=\Gr$.
\end{lemma}
\begin{proof}
The first claim follows from the definition and the second claim follows from the equality 
$(k\gamma)^{-1}(k\gamma)'_1=
\gamma^{-1}\gamma'_1
$, for $k\in LK_c,\gamma\in LG_c$.
\end{proof}

\begin{remark}
The functions $E$ and $E_0$ are not $LK_c$-invariant.
\end{remark}

\begin{lemma}
The Hamiltonian vector field on $\Omega G_c$ which 
corresponds to $E_1$ is given by 
\[\gamma\ra (\gamma^{-1}\gamma')_1+[(\gamma^{-1}\gamma')_1,
(\gamma^{-1}\gamma')_0]\in L\fg_c\on{mod}\fg_c\is T_\gamma\Omega G_c.\]

\end{lemma}}
Note that $E$ is the energy function in \eqref{energy function}.

\begin{lemma}\label{4}
Recall the map $\pi:\Omega G_c\ra\Omega G_c, \gamma\ra\theta(\gamma)^{-1}\gamma$.
We have 
\beq
4E_1=E\circ\pi:\Omega G_c\ra\bbR.
\eeq
In particular, the function $E_1$ is $LK_c$-invariant.

\end{lemma}
\begin{proof}
Write $||v||=\langle v,v\rangle$ for $v\in \fg_c$.
For any $\gamma\in\Omega G_c$ we have 
\[E\circ\pi(\gamma)=\int_{S^1}||\pi(\gamma)^{-1}\pi(\gamma)'||d\theta=
\int_{S^1}||\gamma^{-1}\gamma'-\gamma^{-1}\eta(\gamma)'\eta(\gamma)^{-1}\gamma|| d\theta.\]
Note that $\gamma^{-1}\gamma'-\gamma^{-1}\theta(\gamma)'\theta(\gamma)^{-1}\gamma=2\gamma^{-1}\gamma'_1$
, hence we have $||\gamma^{-1}\gamma'-\gamma^{-1}\theta(\gamma)'\theta(\gamma)^{-1}\gamma||=4||\gamma^{-1}\gamma'_1||$.
The lemma follows.

\quash{
\[\int||\gamma^{-1}\gamma'||+||\gamma^{-1}\eta(\gamma'\gamma^{-1})\gamma||
-2(\gamma^{-1}\gamma',\gamma^{-1}\eta(\gamma'\gamma^{-1})\gamma)_{\on{kil}}d\theta.\]
Since $(,)_{\on{kill}}$ is $G_c$ and $\eta$-invariant,
we have 
\[||\gamma^{-1}\gamma'||+||\gamma^{-1}\eta(\gamma'\gamma^{-1})\gamma||=
||\gamma^{-1}\gamma'||+||\eta(\gamma'\gamma^{-1})||=2||\gamma^{-1}\gamma'||.\]
On the other hand, it follows from (\ref{decomp}) that   
\[(\gamma^{-1}\gamma',\gamma^{-1}\eta(\gamma'\gamma^{-1})\gamma)_{\on{kil}}=
||\gamma^{-1}\gamma'_0||-
||\gamma^{-1}\gamma'_1||.\]
All together, we arrive 
\[E\circ\pi(\gamma)=\int2||\gamma^{-1}\gamma'||-2(
||\gamma^{-1}\gamma'_0||-
||\gamma^{-1}\gamma'_1||)d\theta=\int 4||\gamma^{-1}\gamma'_1||d\theta
=4E_1(\gamma).\]}

\end{proof}

\begin{lemma}\label{hamiltonian}
The Hamiltonian vector field on $\Omega G_c$ which 
correspond to $E_1$ (resp. $E_0$) is given by 
\[\gamma\ra R_1(\gamma)=\gamma'_1-\gamma\gamma_1'(0)\ \ (resp.\  \gamma\ra R_0=\gamma'_0-\gamma\gamma_0'(0)).\] 
In particular, we have 
\[\gamma^{-1}R_1(\gamma)\in\on{Ad}_{\gamma^{-1}}Li\frak p_\bbR+i\frak p_\bbR
\ \ (resp.\  
\gamma^{-1}R_0(\gamma)\in\on{Ad}_{\gamma^{-1}}L\frak k_\bbR+\frak k_\bbR).\]

\end{lemma}
\begin{proof}
Since $R_0(\gamma)+R_1(\gamma)=R(\gamma)=\gamma'-\gamma\gamma(0)'$,
it is enough to show that $R_1(\gamma)=\gamma'_1-\gamma\gamma_1'(0)$.
Let $\gamma\in\Omega G_c$, $x=\pi(\gamma)=\theta(\gamma)^{-1}\gamma$,
and $u\in T_\gamma\Omega G_c$. 
According to Proposition \ref{PS} and Lemma \ref{4},
we have 
\[4dE_1(\gamma)(u)\stackrel{}=\pi^*dE(\gamma)(u)=
dE(x)(\pi_*u)=
\omega(x',\pi_*u)=\omega(x^{-1}x',x^{-1}\pi_*u).\]
Using the equalities  $x^{-1}x'=2\gamma^{-1}\gamma'_1$,  
$x^{-1}\pi_*u=2\gamma^{-1}u_1$, and the fact that $\langle\gamma^{-1}\gamma'_1,(\gamma^{-1}u_0)'\rangle=0$, we get 
 \[4dE_1(\gamma)(u)=4\omega(
\gamma^{-1}\gamma'_1,\gamma^{-1}u_1)=4\int_{S^1} \langle
\gamma^{-1}\gamma'_1,(\gamma^{-1}u_1)'\rangle d\theta=4\int_{S^1} \langle
\gamma^{-1}\gamma'_1,(\gamma^{-1}u)'\rangle d\theta\]
\[=4\int_{S^1} \langle
\gamma^{-1}\gamma'_1-\gamma_1'(0),(\gamma^{-1}u)'\rangle d\theta=
4\omega(R_1(\gamma),u).\]
The lemma follows.

\end{proof}

Let $\Omega G_c=\bigcup_{\lambda\in\calL}\mO_K^\lambda$ and 
$\Omega G_c=\bigcup_{\lambda\in\calL}\mO_\bbR^\lambda$ be the $K(\calK)$-orbits and 
$LG_\bbR$-orbits stratifications of $\Omega G_c$. 
Let $\mO_c^\lambda=\mO_K^\lambda\cap\mO_\bbR^\lambda$ which is a single 
$LK_c$-orbit. 

\begin{prop}\label{Morse-Bott}
Let $E_1:\Omega G_c\ra\bbR$ be the function above and  $\nabla E_1$
be the corresponding gradient vector field. 
\begin{enumerate}
\item $\nabla E_1$ is tangential to both 
$\mO^\lambda_K$ and  $\mO^\lambda_\bbR$,
\item The union $\bigsqcup_{\lambda\in\mathcal L}\mO_c^\lambda$ is the critical manifold of $\nabla E_1$.
\item For any $\gamma\in\mO_c^\lambda$, let 
$T_\gamma\Omega G_c=T^+\oplus T^0\oplus T^-$
be the orthogonal direct sum decomposition into the positive, zero, and 
negative eigenspaces of the Hessian $d^2E_1$. We have 
\[T_\gamma\mO_K^\lambda=T^+\oplus T^0,\ T_\gamma\mO_\bbR^\lambda=
T^-\oplus T^0.\]
\end{enumerate}

\end{prop}
\begin{proof}
Proof of (1). 
We first show that $\nabla E_1$ is tangential to $\mO_\bbR^\lambda=\Omega G_c\cap
LG_\bbR t^\lambda G(\bC[t])$.
Since the tangent space $T_\gamma\mO_\bbR^\lambda$ at $\gamma\in\mO_\bbR^\lambda$
is identified, by 
left translation, with the space
\[\Omega\fg_c\cap(\on{Ad}_{\gamma^{-1}}L\fg_\bbR+\fg(\bC[t]))\subset \Omega\fg_c\] it suffices to show that 
$\gamma^{-1}\nabla E_1(\gamma)\in\on{Ad}_{\gamma^{-1}}L\fg_\bbR+\fg(\bC[t])$.
Recall that, by Proposition \ref{PS}, we have $\gamma^{-1}\nabla E_1(\gamma)=J(\gamma^{-1}R_1(\gamma))
$.
Note that $J(v)+iv\in\fg(\bC[t])$ for $v\in L\fg$ and by Lemma \ref{hamiltonian} we have 
\[i\gamma^{-1}R_1(\gamma)=i(\gamma^{-1}(\gamma'_1-\gamma\gamma_1'(0))
\in\on{Ad}_{\gamma^{-1}}L\fp_\bbR+\fp_\bbR.\]  
All together, 
we get 
\[\gamma^{-1}\nabla E_1(\gamma)=
-i\gamma^{-1}R_1(\gamma)+(J(\gamma^{-1}R_1(\gamma))+
i\gamma^{-1}R_1(\gamma))\in\on{Ad}_{\gamma^{-1}}L\fp_\bbR+\fg(\bC[t])\]
which is contained in $\on{Ad}_{\gamma^{-1}}L\fg_\bbR+\fg(\bC[t])$. We are done.  The same argument as above, replacing $LG_\bbR$ by
$K(\calK)$, shows that the gradient field 
$\nabla E_0$ of $E_0$ is tangential to $\mO_K^\lambda$.
Since, by Corollary \ref{rotation}, the orbit $\mO_K^\lambda$ is a complex submanifold of $\Omega G_c=\Gr$
invariant under the 
rotation flow $\gamma_a(t)$, it follows from Proposition \ref{PS} that 
$\nabla E$ is tangential to $\mO_K^\lambda$. Since 
$\nabla E_1=\nabla E-\nabla E_0$, we conclude that 
$\nabla E_1$ is also tangential to $\mO_K^\lambda$. This finishes 
the proof of (1).

Proof of (2) and (3). 
Let $\Omega X_c^0$ be the components of $\Omega X_c$ in lemma \ref{component X_c}.
By proposition \ref{torsors} and lemma \ref{4}, the function 
$E_1$ factors as 
\[E_1:\Omega G_c\stackrel{\pi}\ra
\Omega X_c^0 \subset\Omega G_c\stackrel{E}\ra\bbR.\]
Thus to  
prove (2) and (3), it is enough to prove following:
\begin{enumerate}[(i)]
\item 
The union $\bigsqcup_{\lambda} B^\lambda$ is the critical manifold 
of the restriction $E$ to $\Omega X_c^0$,
\\
\item 
For $\gamma\in B^\lambda$
we have $T_\gamma P^\lambda=W^+\oplus W^0$, 
$T_\gamma Q^\lambda=W^-\oplus W^0$, where 
$T_\gamma\Omega X^0_c=W^+\oplus W^0\oplus W^-$ is the 
orthogonal direct sum decomposition into the positive, zero, and negative eigenspaces of the Hessian 
$E|_{\Omega X_c^0}$.
\end{enumerate}

By Proposition \ref{PS}, we have 
$T_\gamma S^\lambda=U^+\oplus U^0$, 
$T_\gamma T^\lambda=U^-\oplus U^0$, where 
$T_\gamma\Omega G_c=U^+\oplus U^0\oplus U^-$
is the orthogonal direct sum decomposition into the positive, zero, and negative eigenspaces of the Hessian 
$E$. Note that $\tilde\theta$ induces a linear map on 
$T_\gamma\Omega G_c$, which we still denoted by $\tilde\theta$,
and we have $T_\gamma\Omega X_c=
(T_\gamma\Omega G_c)^{\tilde\theta}$ is the fixed point subspace.
So to prove (i) and (ii) it suffices to show that the subspaces 
$T_\gamma S^\lambda$ and $T_\gamma T^\lambda$ are 
$\tilde\theta$-invariant. It is true, since 
$\tilde\theta=\tilde\eta^\tau$ on $\Omega G_c$ and 
$S^\lambda$ (resp. $T^\lambda$) is $\tilde\theta$-invariant 
(resp. $\tilde\eta^\tau$-invariant). 
This finished the proof of (2) and (3).

\end{proof}

\quash{
\begin{remark}
It is worth comparing this with the finite dimensional situation in \cite{MUV}. 
The functions $E_0,E_1$ here are the affine analogy of the Morse-Bott functions
$f^+,f^-:\calB\to\bbR$ in \cite[Section 3]{MUV}. 
There the sum $f=f^++f^-$ is a constant function on $\calB$ 
\end{remark}
}

\begin{thm}\label{flow}
The gradient $\nabla E_1$ and gradient-flow
$\phi_t$ associated to the $LK_c$-invariant function 
$E_1:\Gr\to\bbR$ and the $LG_c$-invariant metric 
$g(,)$ satisfy the following:
\begin{enumerate}
\item The critical locus $\nabla E_1 = 0$ is the disjoint union of 
$LK_c$-orbits $\bigsqcup_{\lambda\in\mL}\mO_c^\lambda$
\item
The gradient-flow $\phi_t$
preserves the $K(\calK)$-and $LG_\bbR$-orbits.
\item
The limits $\underset{t\ra\pm\infty}\lim\phi_t(\gamma)$ of the gradient-flow exist for any $\gamma\in\Gr$.
For each  $LK_c$-orbit $\mO_c^\lambda$ in the critical locus,
the stable and unstable sets 
\beq\label{eq:morse slows}
\xymatrix{
\mO^\lambda_K=\{\gamma\in\Gr|\underset{t\ra\infty}\lim_{}\phi_t(\gamma)\in\mO^\lambda_c\}
&
\mO^\lambda_\bbR=\{\gamma\in\Gr|\underset{t\ra-\infty}\lim_{}\phi_t(\gamma)\in\mO^\lambda_c\}
}\eeq
are a single $K(\calK)$-orbit and $LG_\bbR$-orbit respectively.
\item
The correspondence between orbits $\mO^\lambda_K\longleftrightarrow\mO^\lambda_\bbR$ defined by \eqref{eq:morse slows} recovers the affine Matsuki correspondence~\eqref{AMS correspondence}.

\end{enumerate}

\end{thm}

\begin{proof}
Part (1) and (2) follows from Proposition \ref{Morse-Bott}.
The $LK_c$-invariant function $E_1$, respectivley the $LG_c$-invariant metric $g(,)$, and the flow $\phi_t$, descends to a $K_c$-invariant Morse-Bott function
$\underline E_1:\Omega K_c\backslash\Gr\to\bbR$, respectivley a  
$K_c$-invariant metric $\underline g(,)$ on $\Omega K_c\backslash\Gr$,
and a flow $\underline\phi_t$. 
Since the function $\underline E_1$ is bounded below and 
the quotient $\Omega K_c\backslash\mO_K^\lambda$ is finite dimensional with $\overline{\Omega K_c\backslash\mO_K^\lambda}=
\bigcup_{\mu\leq\lambda}\Omega K_c\backslash\mO_K^\mu$, 
Proposition \ref{Morse-Bott} and 
standard results for gradient flows (see, e.g., \cite[Proposition 1.19]{AB} or \cite[Theorem 1]{P})
imply that the limit $\underset{t\ra\pm\infty}\lim\underline\phi_t(\gamma)$ 
exists for any $\gamma\in\Omega K_c\backslash\Gr$ and 
$\Omega K_c\backslash\mO_K^\lambda$
is the stable manifold for $\Omega K_c\backslash\mO_c^\lambda$
and $\Omega K_c\backslash\mO_\bbR^\lambda$ is the unstable manifold for $\Omega K_c\backslash\mO_c^\lambda$.
Part (3) and (4) follows.
\end{proof}

We will call the gradient flow $\phi_t:\Gr\to\Gr$ the 
Matsuki flow on $\Gr$.

\section{Real Beilinson-Drinfeld Grassmannians}\label{Real BD}
In this section we recall some basic facts about 
Real Beilinson-Drinfeld Grassmannians. 
The main reference is \cite{N2}.
\subsection{Beilinson-Drinfeld Grassmannians}
Let $\Sigma$ be a smooth curve over $\bC$.
Consider the functor 
$G(\mO)_{\Sigma^n}$ from the category of affine schemes to sets
\[S\ra G(\mO)_{\Sigma^n}(S):=\{(x,\phi)|x\in \Sigma^n(S), \phi\in G(\hat\Gamma_x)\}.\]
Here $\hat\Gamma_x$ is the formal completion of the graphs $\Gamma_x$ of 
$x$ in $\Sigma\times S$.
Similarly, we define 
$G(\calK)_{\Sigma^n}$ to be the functor 
from the category of affine schemes to sets
\[S\ra G(\calK)_{\Sigma^n}(S):=\{(x,\phi)|x\in \Sigma^n(S), \phi\in G(\hat\Gamma_x^0)\}.\]
Here $\hat\Gamma_x^0:=\hat\Gamma'_x-\Gamma_x$ and 
$\hat\Gamma'_x=\on{Spec}(A_x)$ is the spectrum of  
ring of functions $A_x$ of $\hat\Gamma_x$.
$G(\mO)_{\Sigma^n}$ is represented by a formally smooth group scheme over $\Sigma^n$
and $G(\calK)_{\Sigma^n}$ is represented by a formally smooth group ind-scheme over $\Sigma^n$.

Consider the functor $LG_{\Sigma^n}$ that assigns to an affine scheme $S$
the set of sections 
\[S\ra LG_{\Sigma^n}(S)=\{(x,\gamma)|x\in \Sigma^n(S),\ \gamma\in G(\Sigma\times S-\Gamma_x).\}.\]
There is a natural map $LG_{\Sigma^n}\ra G(\calK)_{\Sigma^n}$
sending $(x,\gamma)$ to $(x,\phi=\gamma|_{\hat\Gamma^0_x})$, 
where $\gamma|_{\hat\Gamma^0_x}$ is the restriction of 
the section $\gamma:\Sigma\times S-\Gamma_x$ to $\hat\Gamma^0_x$.

The quotient ind-scheme 
\[\Gr_{\Sigma^n}:=G(\calK)_{\Sigma^n}/G(\mO)_{\Sigma^n}\] 
is 
called the Beilinson-Drinfeld Grassmannian. We have 
\[\Gr_{\Sigma^n}(S)=\{(x,\mE,\phi)|x\in \Sigma^n(S), \mE\text{\ a}\ G\text{-torsor\ on\ } 
\Sigma\times S, \phi\text{\ a trivialization of}\ \mE\text{\ on}\ \Sigma\times S-\Gamma_x\}.\]

\subsection{Real forms}
From now we assume $\Sigma=\mathbb P^1=\bC\cup\infty$. We 
write $\Gr^{(n)}=\Gr_{\Sigma^n}, G(\calK)^{(n)}=G(\calK)_{\Sigma^n}$, etc.
Let $c:\mathbb P^1\ra \mathbb P^1$ be the complex conjugation.
Consider the following anti-holomorphic involution 
$c^{(2)}:(\mathbb P^1)^2\ra (\mathbb P^1)^2, c^{(2)}(a,b)=(c(b),c(a))$.
The involution $c^{(2)}$ together with the involution 
$\eta$ on $G$ defines anti-holomorphic involutions on 
$G(\mO)^{(2)}$, $G(\calK)^{(2)}$, and  
$LG^{(2)}$ and we write 
$G(\mO)^{(2)}_\bbR, G(\calK)^{(2)}_\bbR$, and  $LG_\bbR^{(2)}$
for the corresponding real analytic spaces of real points.
We define 
$\Gr^{(2)}_\bbR=G(\mO)_{\bbR}^{(2)}\backslash G(\calK)^{(2)}_\bbR$ a real form of 
$\Gr^{(2)}$.

\begin{lemma}\label{real forms of various loop groups}
We have the following:
\begin{enumerate}
\item
There are canonical isomorphisms
\[LG^{(2)}_\mathbb R|_{0}\is G_\mathbb R(\mathbb R[t^{-1}]),\ LG^{(2)}_\mathbb R|_{i\mathbb R^\times}\is LG_\mathbb R\times i\mathbb R^\times,\]

\item There are canonical isomorphisms
\[\Gr^{(2)}_\mathbb R|_{0}\is\Gr_\mathbb R,\ \Gr^{(2)}_\bbR|_{
i\mathbb R^\times}\is\Gr\times i\bbR^\times\]
compatible with the natural action of $LG^{(2)}_\bbR$ on $\Gr^{(2)}_\bbR$.
\end{enumerate}
\end{lemma}
\begin{proof}
The isomorphism in (1) is the restriction of 
the natural isomorphisms $LG^{(2)}|_{i\bbR^\times}\is 
LG^{(2)}|_i\times i\bbR^\times\is LG\times i\bbR^\times$ and $LG^{(2)}|_0\is G(\bC[t^{-1}])$. 
Here we regard $i\bbR\subset\bC\times\bC, z\to (z,-z)$
and the isomorphism $LG^{(2)}|_i\is LG, \gamma(z)\to\gamma(t)$ is induced by the 
change of coordinate $t=\frac{z-i}{z+i}$
of $\bP^1$ sending $i$ to $0$, $-i$ to $\infty$, and $\infty$ to $1$. 
The isomorphism in (2) is the restriction of 
the factorization isomorphism $\Gr^{(2)}|_{i\bbR^\times}\is 
\Gr^{(2)}|_i\times i\bbR^\times\is \Gr\times\Gr\times i\bbR^\times$ and $\Gr^{(2)}|_0\is \Gr$. Here the isomorphism $\Gr^{(2)}|_i\times i\bbR^\times\is \Gr\times\Gr\times i\bbR^\times$ is induced by the above coordinate $t=\frac{z-i}{z+i}$.

 \end{proof}

\begin{lemma}\label{poly in K_c}
Assume $G_\bbR$ is compact. 
We have 
$G_\bbR(\mathbb R[t^{-1}])=G_\bbR$ and $G_\bbR(\calK_\mathbb R)=G_\bbR(\mO_\mathbb R)$.
\end{lemma}
\begin{proof}
Note that the real affine Grassmannian $\Gr_\bbR$ 
for a compact group $G_\bbR$ is equal to a point 
and $G_\bbR(\mathbb R[t^{-1}])_1=\{\gamma\in G_\bbR(\mathbb R[t^{-1}])|\gamma(\infty))=e\}$ 
is an open $G_\bbR(\mathbb R[t^{-1}])$-orbit in $\Gr_\bbR$ (see Section \ref{real affine Gr}). 
Hence $G_\bbR(\calK_\mathbb R)=G_\bbR(\mO_\mathbb R)$ and 
$G_\bbR(\mathbb R[t^{-1}])_1=e$. The lemma follows.
\end{proof}

Consider the group ind-scheme 
$\Omega K^{(2)}\subset LK^{(2)}$ that assigns to each affine scheme $S$ the set
\[\Omega K^{(2)}(S)=\{(x,\phi)|x\in\bC(S),\phi\in K(\mathbb P^1\times S-\Gamma_{x\cup -x}), \phi(\{\infty\}\times S)=e\}.\]
The involution on $LK^{(2)}$ restricts to an involution on 
$\Omega K^{(2)}$ and we write $\Omega K^{(2)}_\mathbb R$ for the 
corresponding real analytic ind-space of real points.
\begin{lemma}\label{based loops for K_c}
We have the following:
\begin{enumerate}
\item
There are canonical isomorphisms
\[LK^{(2)}_\mathbb R|_{0}\is K_c ,\ LK^{(2)}_\mathbb R|_{i\mathbb R^\times}\is LK_c\times i\mathbb R^\times.\]
\item
There are canonical isomorphisms
\[\Omega K^{(2)}_\mathbb R|_{0}\is e,\ \Omega K^{(2)}_\mathbb R|_{i\mathbb R^\times}\is \Omega K_c\times i\mathbb R^\times.\]
\end{enumerate}

\end{lemma}
\begin{proof}
It follows directly from Lemma \ref{real forms of various loop groups} and Lemma \ref{poly in K_c}.
\end{proof}
 

\quash{
\subsection{$K(\calK)^{(2)}_\mathbb R$-orbits}
Choose  
$g_\lambda\in\mO_\mathbb R^\lambda\cap\mO_K^\lambda$ for each
$\lambda\in\mathcal L$.
The map
\[s^0_{\lambda}:i\mathbb R^\times\ra\Gr\times i\mathbb R^\times\is\Gr^{(2)}|_{i\mathbb R^\times},
\ x\ra (g_\lambda,x)\]
defines a section of $\Gr^{(2)}_\mathbb R\ra i\mathbb R$ over $i\mathbb R^\times$.
Since the morphism $\Gr^{(2)}_\mathbb R\ra i\mathbb R$ is ind-proper, the map above extends to a section 
\[s_{\lambda}:i\mathbb R\ra\Gr^{(2)}_\mathbb R.\]
\quash{
The section $s_{\lambda,g}$ is compatible with the real structures on 
$\bC$ and $\Gr^{(2)}$, hence induces a section 
\[s_{\lambda}:i\mathbb R\ra\Gr^{(2)}_\mathbb R.\]}
The group ind-scheme $K(\calK)^{(2)}_\mathbb R$
acts naturally on $\Gr_\mathbb R^{(2)}$ and we denote by
\[S_{\lambda}^{(2)}:=K(\calK)^{(2)}_\mathbb Rs_{\lambda}\]
the $K(\calK)^{(2)}_\mathbb R$-orbit on $\Gr^{(2)}_\mathbb R$ through $s_{\lambda}$.
We define $\Gr_\mathbb R^{(2),0}$ to be the union of $S_{\lambda}^{(2)},
\lambda\in\mathcal L$.
\begin{lemma}
1) We have $S_\lambda^{(2)}\cap S_{\lambda'}^{(2)}=\phi$ if 
$\lambda\neq\lambda'$.
2)
There are canonical isomorphisms
\[\Gr_\mathbb R^{(2),0}|_{0}\is\Gr_\mathbb R^0,\ 
\Gr_\mathbb R^{(2),0}|_{i\mathbb R^\times}\is\Gr\times i\mathbb R^\times.
\]
Here $\Gr_\mathbb R^0$ is the union of component of $\Gr_\mathbb R$ in \ref{dd}.
\end{lemma}

The family of groups 
$\Omega K^{(2)}_\mathbb R, LK_\mathbb R^{(2)}$ 
act naturally on $\Gr_\mathbb R^{(2)}$, $S_\lambda^{(2)}$,
and we define 
\beq
\Gr_\mathbb R^{(\eta)}:=\Omega K_\mathbb R^{(2)}\backslash\Gr_\mathbb R^{(2),0},\
S_\lambda^{(\eta)}:=\Omega K_\mathbb R^{(2)}\backslash S_\lambda^{(2)}.
\eeq

Notice that the quotient 
$\Omega K^{(2)}_\mathbb R\backslash LK_\mathbb R^{(2)}\is K_\mathbb R\times
i\mathbb R$ is the constant group $K_\mathbb R$ and it acts 
naturally on $\Gr_\mathbb R^{(\eta)}$ and $S_\lambda^{(\eta)}$.  

\begin{proposition}
1) Each $S_\lambda^{(\eta)}$ is a smooth manifold and the restriction of the projection 
$\pi:\Gr_\mathbb R^{(\eta)}\ra i\mathbb R$ to
$S_\lambda^{(\eta)}\subset\Gr_\mathbb R^{(\eta)}\ra i\mathbb R$ is a submersion.
2) The decomposition $\Gr_\mathbb R^{(\eta)}=\cup_{\lambda\in\mathcal L} S_\lambda^{(\eta)}$
is a Whitney stratification of $\Gr_\mathbb R^{(\eta)}$.
\end{proposition}
\begin{proof}
Proof of 1).

Proof of 2). By 1) and lemma \ref{}, it suffices to show that the decomposition 
$\Gr_\mathbb R^{(\eta)}=\cup_{\lambda\in\mathcal L} S_\lambda^{(\eta)}$
induces a Whitney stratification of each fiber $\Gr_x^{(\eta)}:=\pi^{-1}(x), x\in\mathbb R$.

\end{proof}

\subsection{$LG^{(2)}_\mathbb R$-orbtis}
The group scheme $LG^{(2)}_\mathbb R$
acts naturally on $\Gr_\mathbb R^{(2)}$ and we denote by
\[T_{\lambda}^{(2)}:=LG^{(2)}_\mathbb Rs_{\lambda}\]
the $LG^{(2)}_\mathbb R$-orbit on $\Gr^{(2)}_\mathbb R$ through $s_{\lambda}$.
We define $\Gr_\mathbb R^{(2),0}$ to be the union of $T_{\lambda}^{(2)},
\lambda\in\mathcal L$.
}

\section{Uniformizations of real bundles}\label{uniform}
In this section we 
study uniformizations of the stack of real bundles on 
$\mathbb P^1$ and use it to provide a moduli interpretation for  
the quotient
$LG_\bbR\backslash\Gr$.

In the rest of the paper, 
all the (ind-)stacks are of Bernstein-Lunts type, that is, 
they are unions of open substacks 
$\sX=\bigcup\sX_i$
, each $\sX_i$ being a quotient stack  
$G\backslash X$ of finite type and the 
bounded derived category of $\bC$-constructible
sheaves on $D_c(\sX)$ is the limit of   
$D_c(G\backslash X)$,
where each $D_c(G\backslash X)$ can be defined as an equivariant
derived category in the sense of Bernstein-Lunts (see Appendix \ref{Real stacks}).

\subsection{Stack of real bundles}\label{Stack of real bundles}
Let $\Bun_\bbG(\bbP^1)$ be the moduli stack of $\bbG$-bundles on 
the complex projective line $\mathbb P^1$. The standard complex conjugation 
$z\to\bar z$
on $\mathbb P^1$  together with 
the involution $\eta$ of $\bbG$ defines a real structure $c:\Bun_\bbG(\bbP^1)\to
\Bun_\bbG(\bbP^1)$ on 
$\Bun_\bbG(\bbP^1)$
with real form 
$\Bun_{\bbG_\bbR}(\bbP_\bbR^1)$, the real algebraic stack of $\bbG_\bbR$-bundles on the projective real line $\bbP^1_\bbR$.
We write 
$\Bun_{\bbG}(\bbP^1)_\bbR$ for the 
real analytic stack of real points of $\Bun_{\bbG_\bbR}(\bbP_\bbR^1)$.
By definition, we have $\Bun_{\bbG}(\bbP^1)_\bbR\is \Gamma_\bbR\backslash Y_\bbR$ where $Y\to\Bun_{\bbG_\bbR}(\bbP_\bbR^1)$ is a $\bbR$-surjective presentation of 
the real algebraic stack $\Bun_{\bbG_\bbR}(\bbP_\bbR^1)$\footnote{A presentation of a real algebraic stack is $\bbR$-surjective if it induces a surjective map
on the isomorphism classes of $\bbR$-points.},  
$\Gamma=Y\times_{\Bun_{\bbG_\bbR}(\bbP_\bbR^1)}Y$ is the corresponding groupoid, 
and $X_\bbR,\Gamma_\bbR$ are the real analytic spaces of 
real points of $X,\Gamma$
(see Appendix \ref{Real stacks}).

A point of $\Bun_{\bbG}(\bbP^1)_\bbR$ is 
a $\bbG_\bbR$-bundle $\mE_\bbR$ on $\bbP_\bbR^1$ and, by descent,
corresponds to  
a pair $(\mE,\gamma)$
where $\mE$ is a $\bbG$-bundle on $\bbP^1$ and 
$\gamma:\mE\is c(\mE)$ is an isomorphism such that the induced 
composition is the identity
\[\mE\stackrel{\gamma}\ra c(\mE)\stackrel{c(\gamma)}\ra c(c(\mE))=\mE.\]
We call such pair $(\mE,\gamma)$ a real bundle on $\bbP^1$ and 
$\Bun_{\bbG}(\bbP^1)_\bbR$ the stack of real bundles on $\bbP^1$.

For any $\bbG_\bbR$-bundle $\mE_\bbR$, the restriction of 
$\mE_\bbR$ to the (real) point $\infty$ is a 
$\bbG_\bbR$-bundle on $\Spec(\bbR)$ and
the assignment $\mE_\bbR\to\mE_\bbR|_\infty$
defines a morphism
\[
\Bun_{\bbG_\bbR}(\bbP_\bbR^1)\lra\mathbb B\bbG_\bbR.
\]
For each $\alpha\in H^1(\Gal(\bC/\bbR),G)$, 
let $T_\alpha$ be a $\bbG_\bbR$-torsor on $\Spec(\bbR)$ in the isomorphism class of 
$\alpha$ and we define $\bbG_{\bbR,\alpha}=\Aut_{\bbG_\bbR}(T_\alpha)$.
The collection $\{\bbG_{\bbR,\alpha},\ \alpha\in H^1(\Gal(\bC/\bbR),G)\}$
is the set of pure inner forms of $\bbG_\bbR$.
Let $G_{\bbR,\alpha}=\bbG_{\bbR,\alpha}(\bbR)$ be the real analytic group associated to 
$\bbG_{\bbR,\alpha}$.
We denote by $\alpha_0$ the isomorphism class of trivial $\bbG_\bbR$-torsor. 
By Example \ref{BG},
the morphism above induces a morphism 
\[
cl_\infty:\Bun_{\bbG}(\bbP^1)_\bbR\lra\mathbb \bigsqcup_{\alpha\in
H^1(\Gal(\bC/\bbR),G)}\mathbb BG_{\bbR,\alpha}
\]
on the corresponding real analytic stacks.
Define 
\beq
\Bun_{\bbG}(\bbP^1)_{\bbR,\alpha}:=(cl_\infty)^{-1}(\mathbb BG_{\bbR,\alpha})
\eeq
for the inverse image of $\mathbb BG_{\bbR,\alpha}$ under $cl_\infty$. 
Note that each $\Bun_{\bbG}(\bbP^1)_{\bbR,\alpha}$ is 
an union of connected components of $\Bun_{\bbG}(\bbP^1)_\bbR$ and 
we obtain the following decomposition
of
the stack of real bundles 
\[
\Bun_\bbG(\bbP^1)_{\bbR}=\bigsqcup_{\alpha\in
H^1(\Gal(\bbC/\bbR),G)}\Bun_\bbG(\bbP^1)_{\bbR,\alpha}.
\]
We will call  $\Bun_\bbG(\bbP^1)_{\bbR,\alpha}$ the stack of 
real bundles of class $\alpha$.

\begin{example}
Consider $G=\bC^\times$. In the case $\eta$ is 
the split conjugation,
the cohomology group
$H^1(\Gal(\bC/\bbR),G)$ is trivial and we have
\[\Bun_\bG(\bP^1)_\bbR\is\bZ\times B\bbR^\times.\]
In the case $\eta=\eta_c$ is the compact conjugation, we have 
$H^1(\Gal(\bC/\bbR),G)=\{\alpha_0,\alpha_1\}\is\bZ/2\bZ$ and 
\[\Bun_\bG(\bP^1)_\bbR\is\Bun_\bG(\bP^1)_{\bbR,\alpha_0}\cup\Bun_\bG(\bP^1)_{\bbR,\alpha_1},\]
where $\Bun_\bG(\bP^1)_{\bbR,\alpha_i}\is BS^1$.
\end{example}
\quash{Consider the non-abelian cohomology $H^1(\Gal(\bC/\bbR),G)$ that 
classifies isomorphism classes of $\bbG_\bbR$-bundles over $\Spec(\bbR)$.
We write $\alpha_0\in H^1(\Gal(\bC/\bbR),G)$ for the class of trivial $\bbG_\bbR$-bundle.
and we write 
$cl_\infty(\mE_\bbR)\in H^1(\Gal(\bC/\bbR),G)$ for the corresponding isomorphism class. 
The assignment $\mE_\bbR\ra cl_\infty(\mE_\bbR)$
defines a map
\beq
cl_\infty:\Bun_{\bbG_\bbR}(\bbR\bbP^1)\lra H^1(\Gal(\bbC/\bbR),G)
\eeq
and for each class $\alpha\in H^1(\Gal(\bbC/\bbR),G)$
}

\subsection{Uniformizations of real bundles}
We shall introduce and study two kinds of uniformization 
of real bundles: one uses a real point of $\bbP^1$ called the real uniformization 
the other uses a complex point of $\bbP^1$ called the complex uniformization. 

\subsubsection{Real uniformizations}
The unifomization morphism 
\[u:\Gr\to\Bun_{\bG}(\bP^1)\] for $\Bun_{\bbG}(\bbP^1)$ exhibits $\Gr$ as a $G(\bC[t^{-1}])$-torsor over $\Bun_{\bbG}(\bbP^1)$, in particular,  we have an isomorphism 
\beq\label{uni iso}
G(\bC[t^{-1}])\backslash\Gr\is\Bun_{\bG}(\bP^1).
\eeq
The map $u$ is compatible with the real structures 
on $\Gr$ and $\Bun_{\bG}(\bP^1)$ and 
we denote by 
\beq\label{real uniformization}
u_{\bbR}:\Gr_{\bbR}\ra\Bun_\bbG(\bbP^1)_\bbR
\eeq
the associated map between the corresponding real analytic stacks of real points. 
We call the morphism $u_{\bbR}$ the real uniformization.
It follows from \eqref{uni iso} that 
$u_{\bbR}$ factors through an embedding 
\beq\label{open embedding 1}
G_\bbR(\bbR[t^{-1}])\backslash\Gr_{\bbR}\ra\Bun_\bbG(\bbP^1)_\bbR.
\eeq
We shall describe the image of $u_\bbR$.

\begin{prop}\label{uniformizations at real x}
The map $u_{\bbR}$ factors through 
\[
u_{\bbR}:\Gr_{\bbR}\ra\Bun_\bbG(\bbP^1)_{\bbR,\alpha_0}\subset\Bun_\bbG(\bbP^1)_{\bbR}
\]
and induces an isomorphism of real analytic stacks
\[G_\bbR(\bbR[t^{-1}])\backslash\Gr_{\bbR}\stackrel{\sim}\lra\Bun_G(\bbP^1)_{\bbR,\alpha_0}.\]

\end{prop}
\begin{proof}
Since every $\bbG_\bbR$-bundle 
$\mE_\bbR$ in the image of $u_\bbR$ is trivial over 
$\bP^1_\bbR-\{0\}$, in particular at $\infty$, we have 
$\mE_\bbR\in\Bun_\bbG(\bbP^1)_{\bbR,\alpha_0}$. 
Thus the map $u_\bbR$ factors through $\Bun_\bbG(\bbP^1)_{\bbR,\alpha_0}$.
We show that the resulting morphism $u_\bbR:\Gr_{\bbR}\to\Bun_\bbG(\bP^1)_{\bbR,\alpha_0}$ is
surjective. Let $f:S\to\Bun_\bbG(\bP^1)_{\bbR,\alpha_0}$ be a smooth presentation
(note that $S$ is smooth as $\Bun_\bbG(\bP^1)_{\bbR,\alpha_0}$ is smooth). It suffices 
to show that, \'etale locally on $S$, $f$ admits a lifting to $\Gr_{\bbR}$.
Consider the fiber product $Y:=S\times_{\Bun_\bbG(\bP^1)_\bbR}\Gr_{\bbR}$ 
and we denote by  $h:Y\to S$ the natural projection map. 
It suffices to show that $h$ is surjective and admits a section 
\'etale locally on $S$.
By Theorem 1.1 in \cite{MS},  
every $\bbG_\bbR$-bundle 
$\mE_\bbR$
on $\bbP_\bbR^1$ 
which is trivial at $\infty$ admits a trivialization on 
$\bP^1_\bbR-\{0\}$. It implies $h$ is surjective.
To show that $h$ admits a section, we observe that 
$Y$ is a real analytic ind-space smooth over $\Gr_{\bbR}$  
and, as $u_{\bbR}$ is formally smooth, for any $y\in Y$ and $s=h(y)\in S$,
the tangent map $dh_y:T_yY\to T_sS$ is surjective. 
Choose a finite dimensional subspace $W\subset T_yY$ such that 
$dh_y(W)=T_sS$. We claim that there exists a smooth real analytic space 
$U\subset Y$ such that $y\in U$ and $T_yU=W$. This implies 
$h|_U:U\to S$ is smooth around $y$, thus $f$ admits a section \'etale locally around 
$s=h(y)$. 
Finally, by \eqref{open embedding 1}, we obtain an isomorphism $G_\bbR(\bbR[t^{-1}])\backslash\Gr_\bbR\is
\Bun_\bbG(\bbP^1)_{\bbR,\alpha_0}$. 

To prove the claim, we observe that $Y$ is locally isomorphic to 
$\Gr_{\bbR}$ times a smooth real analytic space.  So it suffices to show for any 
finite dimensional subspace $W\subset T_e\Gr_{\bbR}$, there exists 
a smooth real analytic space $U$ such that $T_eU=W$. 
This
follows from the fact that 
the exponential map $\exp:T_e\Gr_{\bbR}\to\Gr_\bbR$ associated to the 
metric $g(,)|_{\Gr_\bbR}$ (here $g(,)$ is the metric on $\Gr$ in Section \ref{energy flow}) is a local diffeomorphism.

\end{proof}

\subsubsection{Generalization to other components $\Bun_{\bbG}(\bP^1)_{\bbR,\alpha}$}
In this section we briefly discuss generalization of Proposition \ref{uniformizations at real x} to the component
$\Bun_{\bbG}(\bP^1)_{\bbR,\alpha}, \alpha\in H^1(\Gal(\bG/\bbR),G)$.
Recall the $\bbG_{\bbR}$-torsor $T_\alpha$ and the corresponding pure inner form 
$\bbG_{\bbR,\alpha}$.
Note that for each $\bbG_\bbR$-bundle $\mE_\bbR$ the $T_\alpha$-twist 
$\mF_\bbR:=\mE_\bbR\times^{\bbG_\bbR}T_i$ is a $\bbG_{\bbR,\alpha}$-torsor and the assignment 
$\mE_\bbR\to\mF_\bbR$ defines an isomorphism 
\[\Bun_{\bbG_\bbR}(\bbP_\bbR^1)\is\Bun_{\bbG_{\bbR,\alpha}}(\bbP_\bbR^1)\]
of real algebraic stacks. Let $\Gr_{\bG_{\bbR,\alpha}}$ be the affine Grassmannian for $\bG_{\bbR,\alpha}$.
Consider the uniformization map 
\[u_{\alpha}:\Gr_{\bG_{\bbR,\alpha},x}\to\Bun_{\bbG_{\bbR,\alpha}}(\bbP_\bbR^1)\is
\Bun_{\bbG_{\bbR}}(\bbP^1_\bbR).\]
Let $\Gr_{\bbR,\alpha}:=\Gr_{\bG_{\bbR,\alpha}}(\bbR)$
and we 
denote by 
\[u_{\alpha,\bbR}:\Gr_{\bbR,\alpha}\to\Bun_{\bG}(\bP^1)_\bbR\]
the map associated to $u_{\alpha}$. Let 
$|u_{\alpha,\bbR}|:\Gr_{\bbR,\alpha}\to|\Bun_{\bG}(\bP^1)_\bbR|$
the associated map on the isomorphism classes of points.

\begin{lemma}
We have $|u_{\alpha,\bbR}|(\Gr_{\bbR,\alpha})=|\Bun_\bG(\bP^1)_\alpha|$. 
\end{lemma}
\begin{proof}
It suffices to show that 
every $\bbG_\bbR$-bundle 
$\mE_\bbR$
on $\bbP_\bbR^1$ 
such that $\mE_\bbR|_\infty\is T_\alpha$ is in the image 
$u_{\alpha,\bbR}(\Gr_{\bbR,\alpha})$.
By Theorem 1.1 in \cite{MS}, for any such bundle $\mE_\bbR$ the restriction 
of $\mE_\bbR$
to $U_\infty=\bbP_\bbR^1-\{\infty\}$ (resp. $U_0=\bbP_\bbR^1-\{0\}$) is isomorphic 
to $T_\alpha\times U_\infty$ (resp. $T_\alpha\times U_0$).
Since the image $u_{\alpha,\bbR}(\Gr_{\bbR,\alpha})$ consists of 
real bundles which can be obtained from glueing of 
$T_\alpha\times U_\infty$ and $T_\alpha\times U_0$ along the open subset 
$U_\infty\cap U_0=\bbP_\bbR^1-\{0,\infty\}$. It implies 
$\mE_\bbR\in u_{0,\alpha,\bbR}(\Gr_{\bbR,\alpha})$ and 
the proof is complete.

\end{proof}

The lemma above implies that 
the morphism $u_{\alpha,\bbR}$ factors 
through \[u_{\alpha,\bbR}:\Gr_{\bbR,\alpha}\to\Bun_{\bbG}(\bbP^1)_{\bbR,\alpha}
\subset
\Bun_{\bbG}(\bbP^1)_\bbR\]
and the same argument as in the proof of Proposition \ref{uniformizations at real x}
shows that 
\begin{proposition}\label{uniformization for inner forms}
The map $u_{\alpha,\bbR}$ induces an isomorphism 
\[G_{\bbR,\alpha}(\bbR[t^{-1}])\backslash\Gr_{\bbR,\alpha}\stackrel{\sim}\lra\Bun_{\bbG}(\bbP^1)_{\bbR,\alpha}\]
of real analytic stacks.

\end{proposition}

\quash{
Since there is a bijection 
between the set $\Lambda_S^+$ of real dominant coweights and the set 
of $G_\bbR[t_x^{-1}]$-orbits on $\Gr_{x,\bbR}$, we obtain the following: 
\begin{corollary}
There is a bijection
$\Lambda_S^+\leftrightarrow |G_\bbR(\bbR[t_x^{-1}])\backslash\Gr_{x,\bbR}|\leftrightarrow\on{}|\Bun_\bbG(\bbP^1)_{\bbR,\alpha_0}|$.

\end{corollary}}

\subsubsection{Complex uniformizations}
We now discuss complex uniformizations. 
The natural map 
\[u^{(2)}:\Gr^{(2)}\ra\Bun_\bbG(\bbP^1)\times(\bbP^1)^2\] exhibits 
$\Gr^{(2)}$ as a $LG^{(2)}$-torsor over $\Bun_\bbG(\bbP^1)\times(\bbP^1)^2$, that is, we have an isomorphism 
\[LG^{(2)}\backslash\Gr^{(2)}\is\Bun_\bbG(\bbP^1)\times(\bbP^1)^2\]
The morphism $u^{(2)}$ is compatible with 
the complex conjugations on 
$\Gr^{(2)}$ and $\Bun_\bbG(\bbP^1)\times(\bbP^1)^2$
and we denote by  
\beq\label{family of uniformizations}
u^{(2)}_\bbR:\Gr^{(2)}_\bbR\longrightarrow\Bun_\bbG(\bbP^1)_\bbR\times\bbP^1
\eeq
the map between the corresponding real analytic stacks. 
Note that the map 
above factors through an imbedding 
\beq\label{family of embeddings}
LG^{(2)}_\bbR\backslash\Gr_\bbR^{(2)}\to\Bun_\bbG(\bbP^1)_\bbR\times\bbP^1.
\eeq
Recall that, 
by Proposition \ref{real forms of various loop groups}, we have isomorphisms
\[
\Gr^{(2)}_\bbR|_0\is\Gr_{\bbR},\ \ LG^{(2)}_{\bbR}|_0\is G_\bbR(\bbR[t^{-1}])
\]
\[\Gr^{(2)}_\bbR|_{i\bbR^\times}\is\Gr\times i\bbR^\times,\ \ LG^{(2)}_{\bbR}|_{i\bbR^\times}\is LG_{\bbR}\times i\bbR^\times.\]
The restriction 
\[u_{0,\bbR}:=u^{(2)}_{\bbR}|_0:\Gr_\bbR\is\Gr^{(2)}_{\bbR}|_0\to\Bun_\bG(\bP^1)_\bbR\] of \eqref{family of uniformizations} to the real point
$0\in i\bbR$ is isomorphic to the real unifomization map in 
\eqref{real uniformization}. Consider the case when $x\in i\bbR^\times$.
It follows from the isomorphism above that 
there is a unique map
\beq\label{Complex uniformization}
u_{x,\bbC}:\Gr\longrightarrow\Bun_\bbG(\bbP^1)_\bbR
\eeq
making the following diagram commutative
\[\xymatrix{\Gr^{(2)}_\bbR|_x\ar[r]^\sim\ar[d]^{u^{(2)}_\bbR|_x}&\Gr\ar[d]^{u_{x,\bbC}}
\\\Bun_\bbG(\bbP^1)_\bbR\times\{x\}\ar[r]^{\ \ \sim}&\Bun_\bbG(\bbP^1)_\bbR}.\]
We call the map~\eqref{Complex uniformization} the complex uniformization 
associated to $x$. 
Note that, by~\eqref{family of embeddings}, the map $u_{x,\bC}$ induces an embedding 
\beq\label{open embedding 2}
LG_{\bbR}\backslash\Gr\lra\Bun_{\bbG}(\bbP^1)_\bbR.
\eeq

We shall give a description of $u_{x,\bbC}$.
Let $(\mE,\phi)\in\Gr$ where $\mE$
is a $G$-bundle on $\bP^1$ and $\phi:\mE|_{\bP^1-\{0\}}\is G\times(\bP^1-\{0\})$ is a trivialization of $\mE$ over $\bP^1-\{0\}$.  
Let $(\mE_x,\phi_x)$ be the pull back of $(\mE,\phi)$ along the 
isomorphism $\bP^1\is\bP^1, t\to z=\frac{t-x}{t-\bar x}$. 
So $\mE_x$ is a $G$-bundle on $\bP^1$ and 
$\phi_x$ is a trivialization of 
$\mE_x$ on $\bP^1-\{x\}$. 
Let $c(\mE_x)$ be complex conjugation of $\mE_x$ (see Sect.~\ref{Stack of real bundles}) and 
let $\mF$ be the $G$-bundle on 
$\bP^1$ obtained from gluing of 
$\mE|_{\bP^1-\{\bar x\}}$ and $c(\mE)|_{\bP^1-\{x\}}$
using the isomorphism $c(\phi_x)^{-1}\circ\phi_x:\mE|_{\bP^1-\{x,\bar x\}}\is c(\mE)|_{\bP^1-\{x,\bar x\}}$. By construction, there is a canonical isomorphism 
$\gamma:\mF\is c(\mF)$ and the resulting real bundle 
$(\mF,\gamma)\in\Bun_{\bG}(\bP^1)_\bbR$ is the image $u_{x,\bC}((\mE,\phi))$.
Note that the cohomology class in $H^1(\Gal(\bC/\bbR),G)$
given by 
the restriction of the real bundle 
$\mF$ to $\infty$ is represented by the co-boundary $c(\phi_x(v))^{-1}(\phi_x)(v)$ (here $v\in\mE_x|_{\infty}$), hence is trivial. Thus the complex uniformization $u_{x,\bC}$ factors as 
\[u_{x,\bC}:\Gr\to\Bun_{\bG}(\bP^1)_{\bbR,\alpha_0}.\]

We shall describe the image of $u_{x,\bbC}$. 
For each $z\in\bC^\times$ let $a_z:\mathbb P^1\ra\mathbb P^1$
be the multiplication map by $z$.
Consider the flows on $\Gr^{(2)}$ and $\Bun_\bG(\bP^1)$:
\beq\psi_z:\Gr^{(2)}\ra
\Gr^{(2)},\ 
(x,\mE,\phi)\ra (a_z(x),(a_{z^{-1}})^*\mE,(a_{z^{-1}})^*\phi).
\eeq
\[\psi_z:\Bun_\bG(\bP^1)\to\Bun_\bG(\bP^1),\ \ \mE\to (a_{z^{-1}})^*\mE\]
For $z\in\mathbb R_{>0}$ the flows above restrict to flows
\beq\label{flow on Gr_R^2}
\psi_z^1:\Gr^{(2)}_\bbR\ra
\Gr^{(2)}_\bbR,\ \ \psi_z^2:\Bun_\bG(\bP^1)_\bbR\to\Bun_\bG(\bP^1)_\bbR
\eeq
and we have the following commutative diagram
\beq\label{diagram for psi_z}
\xymatrix{\Gr^{(2)}_\bbR\ar[r]^{\psi_z^1}\ar[d]^{q}&\Gr^{(2)}_\bbR\ar[d]^{q}
\\\Bun_\bG(\bP^1)_{\bbR,\alpha_0}\ar[r]^{\psi_z^2}&\Bun_\bG(\bP^1)_{\bbR,\alpha_0}}.
\eeq
Here $q$ is the natural projection map.

\begin{lemma}\label{flows on Gr^(2)}
We have the following properties of the flows:
\begin{enumerate}
\item The critical manifold of the flow $\psi_z^1$ are the cores 
$C_\bbR^\lambda\subset \Gr_\bbR\is\Gr^{(2)}_{\bbR}|_0$ 
and the stable manifold for $C_\bbR^\lambda$ is the 
strata $S_\bbR^\lambda\subset\Gr_\bbR$. 
\item For each $\lambda\in\Lambda_S^+$, we denote by 
\[\tilde T_\bbR^\lambda=\{\gamma\in\Gr^{(2)}_\bbR|
\underset{z\ra 0}\lim\ \psi^1_z(\gamma)\in C_\bbR^\lambda\}\]
the corresponding unstable manifold. 
We have $\tilde T_\bbR^\lambda|_0\is T_\bbR^\lambda\subset\Gr_\bbR$ for 
$\lambda\in\Lambda_S^+$. The isomorphism 
$\Gr^{(2)}_\bbR|_x\is\Gr,\ x\in i\bbR_{>0}$, 
restricts to an isomorphism 
\[\tilde T_\bbR^\lambda|_x\is\mO_\bbR^\lambda\]
for $\lambda\in\mL$ and $\tilde T_\bbR^\lambda|_x$ is empty for 
$\lambda\in\Lambda_S^+-\mL$.
\quash{
\item
The $\bbR_{>0}$-acton $\bbR_{>0}\times\Gr^{(2)}_\bbR\to\Gr^{(2)}_\bbR,\ (z,x)\to\psi^1_z(x)$
 extends to a 
$\bbR_{\geq 0}$-action:
\[\psi^1:\bbR_{\geq 0}\times\Gr^{(2)}_\bbR\to\Gr^{(2)}_\bbR,\]
such that $\psi^1(0,x)=\underset{z\ra 0}\lim\ \psi^1_z(x)$.}
\end{enumerate}
\end{lemma}
\begin{proof}
This is proved in \cite[Proposition 8.4]{N1}. 
\end{proof}

\begin{lemma}\label{flows on Bun_R}
\begin{enumerate}
\item
For any $\gamma\in\Gr_\bbR^{(2)}$,
the action map $\bbR_{>0}\to\Gr_\bbR^{(2)},\ z\to\psi_z^1(\gamma)$
given by the flow $\psi_z^1$
extends to a map $a_\gamma:\bbR_{\geq 0}\to\Gr_\bbR^{(2)}$ such that 
$a_\gamma(0)=\underset{z\to 0}\lim\ \psi_z^1(\gamma)$.
\item
For any $\mE\in\Bun_\bG(\bP^1)_{\bbR,\alpha_0}$,
the action map $\bbR_{>0}\to\Bun_\bG(\bP^1)_{\bbR,\alpha_0},\ z\to\psi_z^2(\mE)$
given by the flow $\psi_z^2$
extends to a map 
\beq\label{extended action}
a_\mE:\bbR_{\geq 0}\to\Bun_\bG(\bP^1)_{\bbR,\alpha_0}.
\eeq
Moreover, we have 
$a_\mE(z)\is\mE$ for all $z\in\bbR_{\geq 0}$, and 
for any $\gamma\in\Gr_\bbR^{(2)}$,
there is a commutative 
diagram 
\beq\label{comm of psi^1}
\xymatrix{\bbR_{\geq 0}\ar[r]^{a_\gamma}\ar[rd]_{a_\mE}&\Gr^{(2)}_\bbR\ar[d]^q\\
&\Bun_\bG(\bP^1)_{\bbR,\alpha_0}}
\eeq
where $\mE=q(\gamma)\in\Bun_\bG(\bP^1)_{\bbR}$.

\end{enumerate}
\end{lemma}
\begin{proof}
Part (1) follows from Lemma \ref{flows on Gr^(2)} (2).
Proof of part (2).
Let $\gamma\in\Gr_\bbR$ and let 
$\mE=q(\gamma)\in\Bun_\bG(\bP^1)_{\bbR}$. Consider the 
the composed map 
\[
a_\mE:\bbR_{\geq 0}\stackrel{a_{\gamma}}\to\Gr_\bbR\to G_\bbR(\bbR[t^{-1}])\backslash\Gr_\bbR\is\Bun_\bG(\bP^1)_{\bbR,\alpha_0}
\]
where $a_\gamma$ is the map in part (1) and the last isomorphism is the real uniformization (see Prop.\ref{uniformizations at real x}).
It is elementary to check that the map
$a_\mE$ only depends on $\mE$ and 
$a_\mE(z)=\psi_z^2(\mE)$ for $z\in\bbR_{>0}$, hence 
 defines the desired map in \eqref{extended action}. Moreover,
since $G_\bbR(\bbR[t^{-1}])$-orbits $T_\bbR^\lambda$ on $\Gr_\bbR$ are unstable manifolds for the flow
$\psi_z^1$, we have $a_\gamma(\bbR_{\geq 0})\subset T_\bbR^\lambda$ if $\gamma\in T_\bbR^\lambda$, and it implies 
$a_\mE(z)\is\mE$ for all $a\in\bbR_{\geq 0}$.
The commutativity of diagram \eqref{comm of psi^1} follows from
the construction of $a_\mE$.

\end{proof}

Recall the 
components $\Gr_\bbR^0=\bigcup_{\mL\in\mL} S_\bbR^\lambda$ in Section \ref{orbits}.  We define 
\[\Bun_\bbG(\bbP^1)_{\bbR,0}\] 
be the image of 
$\Gr_\bbR^0$ under the real uniformization $u_\bbR:
\Gr_\bbR\ra\Bun_\bbG(\bbP^1)_{\bbR,\alpha_0}$.
Note that 
\beq\label{Bun_0}
\Bun_\bbG(\bbP^1)_{\bbR,0}
\is G_\bbR(\bbR[t^{-1}])\backslash\Gr^0_{\bbR}
\subset\Bun_\bbG(\bbP^1)_{\bbR,\alpha_0}\is
G_\bbR(\bbR[t^{-1}])\backslash\Gr_{\bbR}
\eeq 
is a union 
of components of $\Bun_\bbG(\bbP^1)_{\bbR,\alpha_0}$.

\begin{proposition}\label{uniformizations at complex x}
The complex uniformization $u_{x,\bbC}:\Gr\to\Bun_\bbG(\bbP^1)_{\bbR,\alpha_0}$ 
factors as
\[u_{x,\bbC}:\Gr\lra\Bun_\bbG(\bbP^1)_{\bbR,0}\] and
induces an isomorphism  
\[LG_{\bbR}\backslash\Gr\stackrel{\sim}\longrightarrow\on{}\Bun_\bbG(\bbP^1)_{\bbR,0}\]
of real analytic stacks.
\end{proposition}
\begin{proof}
Let $\gamma\in
\Gr$ and $\gamma_x\in\Gr^{(2)}_\bbR|_x$ be the
image of $\gamma$ under the isomorphism $\Gr\is\Gr_\bbR^{(2)}|_x$.
Let $\mE=u_{x,\bbC}(\gamma)=q(\gamma_x)\in\Bun_\bG(\bP^1)_{\bbR,\alpha_0}$
be the image of the complex uniformization map. By Lemma \ref{flows on Bun_R}(2)
we have 
\beq\label{eq 1}
|\mE|=|a_\mE(0)|=|
q(a_{\gamma_x}(0)|=|q(\underset{z\to 0}\lim\ \psi_z^1(\gamma_x))|,
\eeq
\beq\label{eq 2}
\underset{z\to 0}\lim\ \psi_z^1(\Gr\is\Gr^{(2)}_\bbR|_x)=\bigcup_{\lambda\in\mL}C_\bbR^\lambda.
\eeq
 As the image $\bigcup_{\lambda\in\mL}|q(C_\bbR^\lambda)|$ of the critical manifolds under 
 $q$
is equal to $|\Bun_\bbG(\bbP^1)_{\bbR,0}|$, equations \eqref{eq 1} and \eqref{eq 2} imply that 
$u_{x,\bC}$ factors through $u_{x,\bC}:\Gr\to\Bun_\bbG(\bbP^1)_{\bbR,0}$ and induces a surjection 
between the sets of isomorphism classes of objects.
Now a similar argument as in the proof Proposition \ref{uniformizations at real x} shows that 
$u_{x,\bC}:\Gr\to\Bun_\bbG(\bbP^1)_{\bbR,0}$ is surjective and, by \eqref{open embedding 2}, we obtain  
an isomorphism $LG_{\bbR}\backslash\Gr\is\Bun_\bbG(\bbP^1)_{\bbR,0}$.
\end{proof}

\begin{remark} We have 
$\Bun_\bbG(\bbP^1)_{\bbR,0}=\Bun_\bbG(\bbP^1)_{\bbR,\alpha_0}$ if and only if $K$ is connected.
So in the case when $K$ is disconnected, 
the map 
$u_{x,\bC}:\Gr\to\Bun_\bG(\bP^1)_{\bbR,\alpha_0}$ is not 
surjective, that is, 
not 
every real bundle of class $\alpha_0$ admits a complex uniformization. 
\end{remark}

\begin{example}
In the case $G=\bC^\times$ with split conjugation, we have 
\[\Bun_\bG(\bP^1)_{\bbR,0}\is2\bZ\times B\bbR^\times\subset \Bun_\bG(\bP^1)_{\bbR}
\is\bZ\times B\bbR^\times\]
and the complex uniformizatoin 
is given by
\[u_{x,\bC}:\Gr\is\bZ\times\{\on{pt}\}\stackrel{(\times 2,p)}\lra\Bun_\bG(\bP^1)_{\bbR,0}\is2\bZ\times B\bbR^\times.\]
Here $p:\{\on{pt}\}\to B\bbR^\times$ is the quotient map.
\end{example}

\subsection{Categories of sheaves on $LG_\bbR\backslash\Gr$}
Since $\Bun_{\bbG}(\bbP^1)_\bbR$ is 
a real analytic stack of finite type, by the propositions above, the components $\Bun_G(\bbP^1)_{\bbR,\alpha},\Bun_G(\bbP^1)_{\bbR,0}$, 
and the quotients stacks $LG_\bbR\backslash\Gr, G_\bbR(\bbR[t^{-1}])\backslash\Gr_\bbR$, and
 $G_\bbR(\bbR[t^{-1}])\backslash\Gr_\bbR^0$ are also of finite type
 and there are well-defined categories of sheaves on them.

\begin{definition}
We define $D_c(LG_\bbR\backslash\Gr)$ 
to be the bounded derived category of $\bbC$-constructible sheaves on $LG_\bbR\backslash\Gr$. We define $D_!(LG_\bbR\backslash\Gr)$ 
to the be full subcategory of 
$D_c(LG_\bbR\backslash\Gr)$ 
consisting of all constructible complexes that are extensions by zero off of finite type substacks of 
$LG_\bbR\backslash\Gr$.
We denote by $D_c(G_\bbR(\bbR[t^{-1}])\backslash\Gr_\bbR)$, 
$D_!(G_\bbR(\bbR[t^{-1}])\backslash\Gr_\bbR)$, etc, for the similar defined 
categories.

 \end{definition}

\subsection{Uniformizations in family}
Consider the open subset $\Gr^{(2),0}_\bbR\subset\Gr^{(2)}_\bbR|_{i\bbR}$ 
such that $\Gr^{(2),0}_\bbR|_x=\Gr^{(2)}_\bbR|_x$ for $x\neq 0$
and $\Gr^{(2),0}_\bbR|_0\is\Gr^0_\bbR\subset\Gr^{(2)}_\bbR|_0\is\Gr_{\bbR}$.
Let \[u_\bbR^{(2),0}:\Gr^{(2),0}_\bbR\to\Bun_{\bbG}(\bbP^1)_{\bbR}\times i\bbR\] be the restriction of \eqref{family of uniformizations}
to $\Gr^{(2),0}_\bbR$.

\begin{prop}\label{an open family}
The map 
$u_\bbR^{(2),0}$
factors through 
\[u_\bbR^{(2),0}:\Gr^{(2),0}_\bbR\to\Bun_{\bbG}(\bbP^1)_{\bbR,0}\times i\bbR\]
and induces an isomorphism 
\[LG_\bbR^{(2)}\backslash\Gr^{(2),0}_\bbR\stackrel{\sim}\lra\Bun_{\bbG}(\bbP^1)_{\bbR,0}\times i\bbR.\]
of real analytic stacks.
\end{prop}
\begin{proof}
This follows from Proposition~\ref{uniformizations at real x} and~\ref{uniformizations at complex x}.
\end{proof}


\section{Quasi-maps}\label{QMaps}
In this section we 
study the stack of quasi-maps 
and use it to provide moduli interpretation for  
the quotient 
$LK_c\backslash\Gr$. 

\subsection{Definition of quasi-maps}
Let $\Sigma$ be a smooth complex projective curve.
For $n\geq 0$, define the stack of quasi-maps with poles
$\QM^{(n)}(\Sigma,G,K)$ to classify triples $(x, \calE, \sigma)$ comprising a point $x = (x_1, \ldots, x_n)\in \Sigma^n$, a  $G$-torsor $\calE$ over $\Sigma$, and a section $\sigma:\Sigma \setminus |x|\to \mE\times^GG/K$ where we write $|x| = \cup_{i =1}^n x_i \subset \Sigma$. 
According to \cite{GN1}, $\QM^{(n)}(\Sigma,G,K)$ is an ind-stack of ind-finite type.
Note the natural  maps  $\QM^{(n)}(\Sigma,G,K) \to\Sigma^n$ and $\QM^{(n)}(\Sigma,G,K) \to \Bun_G(\Sigma)$. For any $x\in\Sigma^n$, we will write 
$\QM^{(n)}_G(\Sigma,x,G,K)$ for the fiber $\QM^{(n)}(\Sigma,G,K)\times_{\Sigma^n} \{x\}$.

\subsection{Real forms of quasi-maps}\label{real form of QM}
 Now specialize to $\Sigma = \bP^1$ and $n=2$. 
The standard conjugation of $\bP^1$, denoted by $x\mapsto \bar x$, induces a twisted conjugation 
on $(\bP^1)^2$, defined by $c(x_1, x_2) = (\bar x_2, \bar x_1)$ with real points isomorphic to  $(\bP^1)^2_\bbR\is\bP^1$ regarded as a real variety. Let us fix the isomorphism given by the choice of $x_1$.
Together  with the conjugation of $G$ preserving $K$, the twisted conjugation of $(\bP^1)^2$
 induces a conjugation of $\QM^{(2)}(\bP^1,G, K)$. Let us denote its real points by $\QM^{(2)}(\bP^1,G,K)_\bbR$.
Note there are natural  maps  \[\QM^{(2)}(\bP^1,G,K)_\bbR \to(\bP^1)^2_\bbR\is\bP^1(\bC),\text{\ \ \ }\QM^{(2)}(\bP^1,G,K)_\bbR \to \Bun_G(\bbP^1)_\bbR.\] 
Define 
$\QM^{(2)}(\bP^1,G,K)_{\bbR,\alpha_0}$ 
(resp. $\QM_G^{(2)}(\bP^1,G,K)_{\bbR,0}$)
be the pre-image of 
$\Bun_\bbG(\bbP^1)_{\bbR,\alpha_0}$ (resp. 
$\Bun_\bbG(\bbP^1)_{\bbR,0}$) under the morphism 
$\QM^{(2)}(\bP^1,G,K)_{\bbR}\to\Bun_\bbG(\bbP^1)_\bbR$.
For any $x\in\bP^1(\bC)$ we have the fiber 
$\QM^{(2)}(\bP^1,x,G,K)_{\bbR}$ and the intersections 
$\QM^{(2)}(\bP^1,x,G,K)_{\bbR,\alpha_0}$ and $\QM^{(2)}(\bP^1,x,G,X)_{\bbR,0}$.

\subsection{Uniformizations of quasi-maps}
We have a natural uniformization map 
\beq\label{uni for QM}
\Gr^{(2)}\to\QM^{(2)}(\bP^1,G,K)
\eeq 
exhibits 
$\Gr^{(2)}$ as 
a $LK^{(2)}$-torsor on $\QM^{(2)}(\bP^1,G,K)$.
In particular, there is a canonical isomorphism of ind-stacks
\beq\label{iso for QM}
 q^{(2)}:LK^{(2)}\backslash\Gr^{(2)}\stackrel{\sim}\lra \QM^{(2)}(\bP^1,
G,K).
\eeq
The morphism in \eqref{uni for QM} is compatible with the real structures
and we denote by \[\Gr^{(2)}_{\bbR}\lra 
\QM^{(2)}(\bP^1,G,K)_\bbR\]
the associated map on the corresponding real algebraic stacks of real points.
It follows from \eqref{iso for QM} that the map above factors through an embedding 
\[q^{(2)}_\bbR:LK_\bbR^{(2)}\backslash\Gr^{(2)}_{\bbR}\lra 
\QM^{(2)}(\bP^1,G,K)_\bbR.\]
By Lemma \ref{based loops for K_c}, there are natural isomorphisms 
\[LK^{(2)}_{\bbR}\backslash\Gr^{(2)}_{\bbR}|_x\is
LK_{c}\backslash\Gr,\ \ x\in i\bbR^\times,\ \ LK^{(2)}_{\bbR}\backslash\Gr^{(2)}_{\bbR}|_0\is
K_{c}\backslash\Gr_\bbR
\] and the map $q^{(2)}_\bbR$ gives rise to maps  
\[q_x:LK_{c}\backslash\Gr\lra\QM^{(2)}(\bbP,^1x,G,K)_\bbR,\ x\in i\bbR^\times\]
\[q_0:K_{c}\backslash\Gr_\bbR\lra\QM^{(2)}(\bbP^1,0,G,K)_\bbR.\]

\begin{lemma}\label{uni for real quasi-maps}
We have the following:
\begin{enumerate}
\item

The map $q_x$ induces an isomorphism
\[q_x:LK_{c}\backslash\Gr\stackrel{\sim}\lra\QM^{(2)}(\bbP^1,x,G,K)_{\bbR,0}.\]
\item
The map $q_0$ induces an isomorphism
\[q_0:K_{c}\backslash\Gr_\bbR\stackrel{\sim}\lra\QM^{(2)}(\bbP^1,0,G, K)_{\bbR,\alpha_0}.\]
\end{enumerate}
\end{lemma}
\begin{proof}
This follows from Proposition \ref{uniformizations at complex x} (resp. Proposition \ref{uniformizations at real x}) that every 
real bundle $\mE$ in $\Bun_\bbG(\bbP^1)_{\bbR,0}$ (resp. $\Bun_\bbG(\bbP^1)_{\bbR,\alpha_0}$)
admits a complex uniformization (resp. a real uniformization).

\end{proof}

Recall the open family $\Gr_\bbR^{(2),0}\to i\bbR$ and the family of uniformizations 
$LG_\bbR^{(2)}\backslash\Gr_\bbR^{(2),0}\is\Bun_\bbG(\bP^1)_{\bbR,0}\times i\bbR$
in Proposition \ref{an open family}. The above lemma implies the following.

\begin{proposition}\label{open sub families of quasi maps}
The natural map $\Gr_\bbR^{(2),0}\to QM^{(2)}(\bP^1,G,K)_{\bbR}$ induces an isomorphism 
\[LK_\bbR^{(2)}\backslash\Gr_\bbR^{(2),0}\is QM^{(2)}(\bP^1,G,K)_{\bbR,0}|_{i\bbR}\]
and we have the following commutative diagram 
\[\xymatrix{LK_\bbR^{(2)}\backslash\Gr_\bbR^{(2),0}\ar[d]\ar[r]&QM^{(2)}(\bP^1,G,K)_{\bbR,0}|_{i\bbR}\ar[d]\\
LG_\bbR^{(2)}\backslash\Gr_\bbR^{(2),0}\ar[r]&\Bun_\bbG(\bP^1)_{\bbR,0}\times i\bbR}.\]
where the vertical maps are the natural quotient and projection maps.
In addition, there are canonical isomorphisms 
\[LK_c\backslash\Gr\times i\bbR^\times\is QM^{(2)}(\bP^1,G,K)_{\bbR,0}|_{i\bbR^\times}\]
\[K_c\backslash\Gr^0_\bbR\is QM^{(2)}(\bbP^1,0,G,K)_{\bbR,0}\]
\end{proposition}

\subsection{Categories of sheaves on $LK_c\backslash\Gr$}
By Lemma \ref{uni for real quasi-maps}, the real analytic ind-stack
$LK_c\backslash\Gr$ is of ind-finite type and we have a well-defined 
category of sheaves on it.
Introducing  the stratification $\calS$ of $LK_c\backslash\Gr$ with strata the
$LK_c$-quotients of $K(\calK)$-orbits. 

\begin{definition}
Let $D_c(LK_c\backslash\Gr)$ 
be the bounded constructible derived category of  sheaves on $LK_c\backslash\Gr$. 
We set $D_c(K(\calK)\backslash\Gr)$ to be the full subcategory
of $D_c(LK_c\backslash\Gr)$
of complexes constructible with 
respect to the stratification $\calS$.
\end{definition}

\subsection{Rigidified quasi-maps}
Let $\QM^{(n)}(\bbP^1,G,K,\infty)$ be the ind-scheme classifies 
quadruple $(x,\mE,\phi,\iota)$ where $x\in\bC^n$,
$\mE$ is a $G$-bundle on $\bbP^1$, $\phi:\bbP^1-|x|\ra\mE\times^GX$, and 
$\iota:\mE_K|_\infty\is K$, here $\mE_K$ is the $K$-reduction of
$\mE$ on $\bbP^1-|x|$ given by $\phi$. We have a natural map 
$\QM^{(n)}(\bbP^1,G,K,\infty)\ra\bC^n$. 
The ind-scheme $\QM^{(n)}(\bbP^1,G,K,\infty)$ is called rigidified quasi-maps.
Note that we have natural map 
\[
\QM^{(n)}(\bbP^1,G,K,\infty)\to\QM^{(n)}(\bbP^1,G,K)
\]
sending $(x,\mE,\phi,\iota)$ to $(x,\mE,\phi)$
and it induces an isomorphism 
\beq\label{uni for QM(infty)}
K\backslash\QM^{(n)}(\bbP^1,G,K,\infty)\is
\QM^{(n)}(\bbP^1,G,K)|_{\bC^n}
\eeq
where the group $K$ acts on 
$\QM^{(n)}(\bbP^1,G,K,\infty)$
by changing the trivialization $\iota$.

The twisted conjugation 
on $(x_1,x_2)\ra (\bar x_2,\bar x_1)$ together with the involution 
$\eta$ on $G$ defines a real form 
$\QM^{(2)}(\bbP^1,G,K,\infty)_\bbR$ of $\QM^{(2)}(\bbP^1,G,K,\infty)$. 
We have a natural map 
$\QM^{(2)}(\bbP^1,G,K,\infty)_\bbR\ra
\Bun_\bbG(\bP^1)_\bbR$ and we denote by 
$\QM^{(2)}(\bbP^1,G,K,\infty)_{\bbR,0}$ the pre-image of the components
$\Bun_\bbG(\bP^1)_{\bbR,0}$. 
The isomorphism (\ref{uni for QM(infty)}) induces  an embedding 
\[K_c\backslash\QM^{(2)}(\bbP^1,G,K,\infty)_{\bbR}\to
\QM^{(2)}(\bbP^1,G,K)_{\bbR}|_{\bC}
.\]
It follows from Proposition \ref{open sub families of quasi maps} that the above embedding restricts to an isomorphism 
\beq\label{QM=QM(infty)/K}
K_c\backslash\QM^{(2)}(\bbP^1,G,K,\infty)_{\bbR,0}\is
\QM^{(2)}(\bbP^1,G,K)_{\bbR,0}|_\bC
\eeq
and there are canonical isomorphisms  
\beq\label{factorization}
\Omega K_\bbR^{(2)}\backslash\Gr_\bbR^{(2),0}
\is\QM^{(2)}(\bbP^1,G,K,\infty)_{\bbR,0}|_{i\bbR}
\eeq
\[\Omega K_c\backslash\Gr\times i\bbR^\times\is QM^{(2)}(\bP^1,G,K,\infty)_{\bbR,0}|_{i\bbR^\times}\]
\[\Gr^0_\bbR\is QM^{(2)}(\bbP^1,G,K,\infty)_{\bbR,0}|_0\]

Consider the stratifications 
$\calS_1=\{\Omega K_c\backslash\mO_K^\lambda\}_{\lambda\in\mL}$ 
of $\Omega K_c\backslash\Gr^0$ and $\calS_2=\{S_\bbR^\lambda\}_{\lambda\in\mL}$
of $\Gr_\bbR^0$.
By\eqref{factorization}, the union 
$\calS^{(2)}=\calS_1\times i\bbR^\times\cup\calS_2\times\{0\}$ forms a stratification of 
$\QM^{(2)}(\bbP^1,G,X,\infty)_{\bbR,0}|_{i\bbR}$.
In section \ref{nearby cycles and Radon TF}, we will need following technical lemma, which is proved in \cite[Proposition 6.7]{CN2}.

\begin{lemma}\label{Whitney for QM_R}
The stratification $\calS^{(2)}$ above
is Whitney and the natural map
\[\QM^{(2)}(\bbP^1,G,K,\infty)_{\mathbb R,0}|_{i\bbR}\lra i\bbR\] is
a Thom stratified map. Here $i\bbR$ is equipped with the 
stratification
$i\bbR=i\bbR^\times\cup\{0\}$.
\end{lemma}

\begin{remark}
In fact, in \emph{loc. cit.}, we show that the quasi-maps family 
above 
admits a $K_c$-equivariant topological trivialization.
\end{remark}
\quash{
The natural map 
$\Gr^{(2)}|_{\bC^2}\to\QM^{(2)}(\bP^1,G,,\infty)$
induces an open embedding 
\[p^{(2)}_\bbR:(\Omega K^{(2)}_\bbR\backslash\Gr^{(2)}_{\bbR})|_\bC\lra 
\QM^{(2)}(\bP^1,G,X)_\bbR.\]
By Lemma \ref{based loops for K_c}, 
we have natural isomorphisms \[\Omega K^{(2)}_{\bbR}\backslash\Gr^{(2)}_{\bbR}|_i\is
\Omega K_{c}\backslash\Gr,\ \ \Omega K^{(2)}_{\bbR}\backslash\Gr^{(2)}_{\bbR}|_0\is
\Gr_\bbR
\] and the morphism above restricts to open embeddings 
\[p_i:\Omega K_{c}\backslash\Gr\to\QM^{(2)}(i,G,X,\infty)_\bbR,\]
\[p_0:\Gr_\bbR\to\QM^{(2)}(0,G,X,\infty)_\bbR.\]
}

\subsection{Flows on quasi-maps}\label{flows}
For each $z\in\bC^\times$ we have the following flow
\beq\psi_z:\QM^{(2)}(\bbP^1,G,K,\infty)\ra
\QM^{(2)}(\bbP^1,G,K,\infty),\ 
(x,\mE,\psi,\iota)\ra (a_z(x),(a_{z^{-1}})^*\mE,(a_{z^{-1}})^*\psi,\iota).
\eeq
\quash{
We have the following commutative diagrams 
\[\xymatrix{\QM^{(2)}(\bbP^1,G,K)\ar[r]^{\psi_z}\ar[d]&\QM^{(2)}(\bbP^1,G,K)\ar[d]
\\(\bbP^1)^2\ar[r]^{a_z}&(\bbP^1)^2}\ \ \ \ \ \ 
\xymatrix{\QM^{(2)}(\bbP^1,G,K,\infty)\ar[r]^{\psi_z}\ar[d]&\QM^{(2)}(\bbP^1,G,K,\infty)\ar[d]
\\\bC^2\ar[r]^{a_z}&\bC^2}.\]
}
For $z\in\mathbb R_{>0}$ the flow $\phi_z$ restricts to a flow
\[\psi^3_z:\QM^{(2)}(\bbP^1,G,K,\infty)_\bbR\ra\QM^{(2)}(\bbP^1,G,K,\infty)_\bbR,\]
and we have the following commutative diagrams 
\[\xymatrix{\QM^{(2)}(\bbP^1,G,K,\infty)_\bbR\ar[r]^{\psi_z^3}\ar[d]&\QM^{(2)}(\bbP^1,G,K,\infty)_\bbR\ar[d]
\\\bbP^1\ar[r]^{a_z}&\bbP^1}\]

\begin{lemma}\label{properties of flows}
We have the following properties of the flows:
\begin{enumerate}
\item The flow $\psi_z$ on $\QM^{(2)}(\bbP^1,G,K,\infty)_\bbR$ is 
$K_c$-equivariant.
\item
Recall the flow 
$\psi_z^1$ on $\Gr^{(2)}_\bbR$~\eqref{flow on Gr_R^2}.
We have the following commutative diagram
\beq\label{compatibility with flows}
\xymatrix{\Gr^{(2)}_\bbR|_\bC\ar[r]^{\psi_z^1}\ar[d]&\Gr^{(2)}_\bbR|_\bC\ar[d]\\
\QM^{(2)}(\bbP^1,G,K,\infty)_\bbR\ar[r]^{\psi_z^3}&\QM^{(2)}(\bbP^1,G,K,\infty)_\bbR}.
\eeq

\item For each $\lambda\in\Lambda_S^+$, the core $C_\bbR^\lambda\subset \Gr_\bbR
\subset\QM^{(2)}(\bP^1,G,K,\infty)_\bbR|_0$ is a union of components of 
the critical manifold of the flow $\psi_z^3$ on $\QM^{(2)}(\bbP^1,G,K,\infty)_\bbR$ 
and the stable manifold for $C_\bbR^\lambda$ is the 
strata $S_\bbR^\lambda\subset\Gr_\bbR$. 
\item For each $\lambda\in\Lambda_S^+$, we denote by 
\[\tilde T_\bbR^\lambda=\{x\in\QM^{(2)}(\bbP^1,G,K,\infty)_\bbR|
\underset{z\ra 0}\lim\psi_z^3(x)\in C_\bbR^\lambda\}\]
the corresponding unstable manifold. 
We have $\tilde T_\bbR^\lambda|_0\is T_\bbR^\lambda\subset\Gr_\bbR$ for 
$\lambda\in\Lambda_S^+$. The open embedding 
$\Omega K_c\backslash\Gr\to\QM^{(2)}(\bP^1,G,K,\infty)_\bbR|_i$ restricts to an isomorphism 
\[\Omega K_c\backslash\mO_\bbR^\lambda\is\tilde T_\bbR^\lambda|_i\]
for $\lambda\in\mL$.

\end{enumerate}
\end{lemma}
\begin{proof}
Part (1) and (2) follows from the construction of the flows. 
Part (3) and (4) 
follows from Lemma \ref{flows on Gr^(2)} and diagram (\ref{compatibility with flows}).

\end{proof}

\quash{
\subsection{Stratifications of quasi-maps}\label{St of quasi-maps}
For any partition $p(n)$ of $\{1,...,n\}$ and a map $\Theta:p(n)\ra
X(\calK)/G(\mO)$, we say that a quasi map 
$(x,\mE,\sigma)\in\QM^{(n)}(\Sigma,G,X)$ is of type $(p(n),\Theta)$
if the followings hold. First, the coincidences among 
$|x|$ are given by the partition $p(n)$. Second, 
the restriction to the formal neighborhood of 
$x_i\in\bP^1$ provides a map 
 \beq\label{ev}
 \xymatrix{
 ev_{x_i}:\QM^{(n)}(\Sigma,G,X) \ar[r] &  X(\calK_{})/G(\mO_{})\is\Lambda_S^+
 }
\eeq
and we required $ev_{x_i}((x,\mE,\sigma))=\Theta(i)$.
We define the local stratum 
$\QM^{p(n),\Theta}(\Sigma,G,X)\subset \QM^{(n)}(\Sigma,G,X)$
to consist of those quasi maps of type $(p(n),\Theta)$.
For any $p(n)$ we define the local stratum 
$\Sigma^{p(n)}:=\{(x_1,...,x_n)|x_i=x_j\text{\ if\ } i,j\text{\ are in the same part of\ } p(n)\}$.
Let $\calS^{(n)}$ (resp. $\calV^{(n)}$) be the stratification of 
$\QM^{(n)}(\Sigma,G,X)$ (resp. $\Sigma^n$) forms by the 
local stratum $\QM^{p(n),\Theta}(\Sigma,G,X)$ (resp. $\Sigma^{p(n)}$). 
For any $x\in\Sigma^{p(n)}$, let
$\calS^{(n)}_x$ be the stratification of 
$\QM^{(n)}(x,G,X)$ forms by the stratum
$\QM^{p(n),\Theta}(x,G,X)=\QM^{p(n),\Theta}(\Sigma,G,X)\times_{\Sigma^{n}}\{x\}$.

\quash{
\begin{proposition}\label{Whitney}
The stratifications $\calS^{(n)}$ and $\calV^{(n)}$ 
of $\QM^{(n)}_G(\Sigma,X)$ and $\Sigma^{n}$
are
Whitney. The natural projection map 
$\QM^{(n)}_G(\Sigma,X)\ra\Sigma^n$ is a Thom 
stratified map.
\end{proposition}}

\subsection{Moduli interpretations for $K(\calK)$-orbits}
Recall the stratum 
$\QM^{p(2),\Theta}(\bP^1,G,X)$ and stratification 
$\calS^{(2)}=\{\QM^{p(2),\Theta}(\bP^1,G,X)\}$ of $\QM^{(2)}(\bP^1,G,X)$
in Section \ref{St of quasi-maps}.
In the case $p(2)=\{1\}\cup\{2\}$ (resp. $p(2)=\{1,2\}$), we let 
$\lambda_k=\Theta(\{k\})\in\Lambda_S^+$ (resp. $\lambda=\Theta(\{1,2\})\in\Lambda_S^+$) and write 
$\QM^{\lambda_1,\lambda_2}(\bP^1,G,X)$ (resp. $\QM^{\lambda}(\bP^1,G,X)$) for $\QM^{p(2),\Theta}(x,G,X)$. 
The complex conjugation on $\QM^{(2)}(\bP^1,G,X)$ 
preserves the stratum $\QM^{\lambda}(\bbP^1,G,X)$ and
maps the stratum $\QM^{\lambda_1,\lambda_2}(\bbP^1,G,X)$ to 
$\QM^{\lambda_2,\lambda_1}(\bbP^1,G,X)$. In particular, the 
complex conjugation preserves the 
stratification $\calS^{(2)}$ 
of $\QM(\bbP^1,G,X)$. We denote by 
$\calS^{(2)}_\bbR$ the stratification of 
$\QM^{(2)}(\bP^1,G,X)_\bbR$ forms by the fixed 
points of the strata.  
For any $x\in\bP^1$ we denote by 
$\calS^{(2)}_{\bbR,x}$ 
the stratification of $\QM^{(2)}(x,G,X)_\bbR$ 
forms by the 
intersection.
Let
\[ev_{x}^\bbR:\QM^{(2)}(x,G,X)_\bbR\to
\QM^{(2)}(\bP^1,G,X)\stackrel{ev_{x}}\to
 \Lambda_S^+\] 
where the first arrow is the natural map and $ev_{x}$ is the evaluation map in (\ref{ev}).
Define 
\beq\label{local stratum at x}
\QM^{\lambda}(x,G,X)_\bbR:=(ev_{x}^\bbR)^{-1}(\lambda).
\eeq
Then the stratification 
$\calS^{(2)}_{\bbR,x}$ of $\QM^{(2)}(x,G,X)_\bbR$ is given by 
\[\calS^{(2)}_{\bbR,x}=\{\QM^{\lambda}(x,G,X)_\bbR\}_{\lambda\in\Lambda_S^+}.\]

We shall study the relation 
between $\calS^{(2)}_{\bbR,x}$ and the
$K(\calK)$-orbits stratification on $\Gr$ and $G_\bbR(\mO_\bbR)$-orbits stratification on 
$\Gr_\bbR$. We begin with the following lemma:

\begin{lemma}\label{ev for Gr}
We have the following:
\begin{enumerate}
\item We have the following commutative diagram
\[\xymatrix{
LK_c\backslash \Gr\ar[r]^{q_i\ \ \ }\ar[d] & \QM^{(2)}(i,G,X)_\bbR\ar[d]^{ev_i^\bbR}\\
|K(\calK)\backslash\Gr|\is\mL\ar[r]&|X(\calK)/G(O)|\is\Lambda_S^+}.\]
Here 
the vertical left arrow is the natural quotient map, and the 
lower horizontal arrow is the natural inclusion map.
\item  
We have the following commutative diagram
\[\xymatrix{K_c\backslash \Gr_\bbR\ar[r]^{q_0\ \ }\ar[d] & \QM^{(2)}(0,G, X)_\bbR\ar[d]^{ev_0^\bbR}\\
|K_c(\calK_\bbR)\backslash\Gr_\bbR|\ar[r]&|X(\calK)/G(O)|\is\Lambda_S^+}.\]
Here 
the vertical left arrow is the natural quotient map, and the 
lower horizontal arrow is the natural inclusion map.
\
\item
The $K_\bbR(\calK_\bbR)$-orbits in $\Gr_\bbR$
coincide with 
$G_\bbR(\mO_\bbR)$-orbits. Moreover,
under the 
identifications 
$|K_\bbR(\calK_\bbR)\backslash\Gr_\bbR|\is |G_\bbR(\mO_\bbR)\backslash\Gr_\bbR|\is\Lambda_S^+$, $|X(\calK)/G(O)|\is\Lambda_S^+$, the inclusion map
\[|K_c(\calK_\bbR)\backslash\Gr_\bbR|\ra |X(\calK)/G(O)|\]
in (2)
becomes the identity map 
$\on{id}_{\Lambda_S^+}$.

\end{enumerate}
\end{lemma}
\begin{proof}
Part (1) and (2) follows from the definition and part (3) is proved in \cite[Lemma 9.1]{N2}.
\end{proof}

Recall the stratum $QM^{\lambda}(x,G,X)_{\bbR}$
in (\ref{local stratum at x}).
Define $QM^{\lambda}(x,G,X)_{\bbR,\alpha_0}$ (resp. 
$QM^{\lambda}(x,G,X)_{\bbR,0}$ be the pre-image of 
$\Bun_\bbG(\bP^1)_{\bbR,\alpha_0}$ (resp. $\Bun_\bbG(\bP^1)_{\bbR,0}$)
along the natural projection map $QM^{\lambda}(x,G,X)_{\bbR}\to\Bun_\bbG(\bP^1)_{\bbR}$.
The following proposition follows from 
Proposition \ref{parametrization of K(K)-orbits} and 
Lemma \ref{ev for Gr}:
\begin{prop}\label{moduli int for orbits}
We have the following:
\begin{enumerate}
\item
The collection $\{\QM^{\lambda}(i,G,X)_{\bbR,0}\}_{\lambda\in\mL}$
forms a stratification of 
$\QM^{(2)}(i,G,X)_{\bbR,0}$ and the 
isomorphism 
$q_i:LK_c\backslash\Gr\is\QM^{(2)}(i,G,X)_{\bbR,0}$ induces
a stratified isomorphism 
\[(LK_c\backslash\Gr,\{LK_c\backslash\mO_K^\lambda\}_{\lambda\in\mL})\is(\QM^{(2)}(i,G,X)_{\bbR,0},\{\QM^{\lambda}(i,G,X)_{\bbR,0}\}_{\lambda\in\mL}).\]

\item
The collection $\{\QM^{\lambda}(0,G,X)_{\bbR,\alpha_0}\}_{\lambda\in\Lambda_S^+}$
forms a stratification of 
$\QM^{(2)}(0,G,X)_{\bbR,\alpha_0}$ and the 
isomorphism 
$q_0:LK_c\backslash\Gr\is\QM^{(2)}(i,G,X)_{\bbR,\alpha_0}$ induces
a stratified isomorphism 
\[(K_c\backslash\Gr_\bbR,\{K_c\backslash S_\bbR^\lambda\}_{\lambda\in\Lambda_S^+})\is(\QM^{(2)}(0,G,X)_{\bbR,\alpha_0},\{\QM^{\lambda}(0,G,X)_{\bbR,\alpha_0}\}_{\lambda\in\Lambda_S^+}).\]

\end{enumerate}
\end{prop}

}

\quash{
\subsection{Some auxiliary stacks}
We introduce some auxiliary ind-stacks that will be used later in the paper.
Let $\QM^{(2)}_G(\bbP^1,X,\infty)$ be the ind-stack classifies 
quadruple $(x,\mE,\phi,\iota)$ where $x=(x_1,x_2)\in\bC^2$,
$\mE$ is a $G$-bundle on $\bbP^1$, $\phi:\bbP^1-|x|\ra\mE\times^KX$, and 
$\iota:\mE_K|_\infty\is K$, here $\mE_K$ is the $K$-reduction of
$\mE$ on $\bbP^1-|x|$ given by $\phi$. We have a natural map 
$\QM^{(2)}_G(\bbP^1,X,\infty)\ra\bC^2$. The twisted conjugation 
on $(x_1,x_2)\ra (\bar x_2,\bar x_1)$ together with the involution 
$\eta$ on $G$ defines a real form 
$\QM^{(2)}_G(\bbP^1,X,\infty)_\bbR$ of $\QM^{(2)}_G(\bbP^1,X,\infty)$. Note the natural map $\QM^{(2)}_G(\bbP^1,X,\infty)_\bbR\ra\bC$.
The group $K$ (resp. $K_c$) acts naturally on 
$\QM^{(2)}_G(\bbP^1,X,\infty)$ (resp. $\QM^{(2)}_G(\bbP^1,X,\infty)_\bbR$)
by changing the trivialization $\iota$ and we have natural isomorphisms
\beq\label{QM=QM(infty)/K}
\QM^{(2)}_G(\bbP^1,X)\is K\backslash\QM^{(2)}_G(\bbP^1,X,\infty),\ \ 
\QM^{(2)}_G(\bbP^1,X)_\bbR\is
K_c\backslash\QM^{(2)}_G(\bbP^1,X,\infty)_\bbR.
\eeq

The pre-image of the stratification
$\calS$ (resp. $\calS_\bbR$) along the quotient map
$\QM^{(2)}_G(\Sigma,X,\infty)\ra\QM^{(2)}_G(\Sigma,X)$ (resp.
$\QM^{(2)}_G(\Sigma,X,\infty)_\bbR\ra\QM^{(2)}_G(\Sigma,X)_\bbR$) 
defines a stratification $\calS'$ (resp. $\calS_\bbR'$) of 
$\QM^{(2)}_G(\Sigma,X,\infty)$ (resp. $\QM^{(2)}_G(\Sigma,X,\infty)_\bbR$).

\begin{proposition}\label{Whitney for QM_R}
The stratification $\calS'$ (resp. $\calS'_\bbR$)
of $\QM^{(2)}_G(\bbP^1,X,\infty)$ (resp. the real form 
$\QM^{(2)}_G(\bbP^1,X,\infty)_\mathbb R$)
is Whitney. The maps 
$\QM^{(2)}_G(\bbP^1,X,\infty)\ra\bC^2$ 
and $\QM^{(2)}_G(\bbP^1,X,\infty)_\mathbb R\ra \bC$ are
Thom stratified maps, where $\bC^2$ and $\bC$ are equipped with the 
stratifications $\bC^2=(\bC^2-\Delta\bC)\cup\Delta\bC$ 
and $\bC=(\bC-\bbR)\cup\bbR$.
\end{proposition}

\quash{
\begin{lemma}
It follows from proposition \ref{Whitney} and the fact that if 
a Thom stratified map is equivariant under an action of a compact group $H$, then 
the induced map on the $H$-fixed points is again a Thom stratified map
(see \cite[Lemma 4.5.1]{N1}).
\end{lemma}}

\subsection{Complex groups}\label{complex}
The complex conjugations on $\Gr^{(2)}$ and $\Omega G_c^{(2)}$
induce an involution 
on the quotient $\Omega G_c^{(2)}\backslash\Gr^{(2)}$ denoted by $\eta^{(2)}$.
}

\quash{
\section{Stratifications}

\subsection{A stratified family}
Consider the left copy embedding $\iota:\Gr^{(1)}\ra\Gr^{(2)}$ sending 
$(x,\mE,\phi)$ to $(x,-x,\mE,\phi|_{\bP^1-\pm x})$. 
The composition 
$\Gr^{(1)}\to\Gr^{(2)}\to\Omega G_c^{(2)}\backslash\Gr^{(2)}$ is
an isomorphism of ind-varieties
\[\Gr^{(1)}\is \Omega G_c^{(2)}\backslash\Gr^{(2)}\]
and the involution $\eta^{(2)}$ 
on $\Omega G_c^{(2)}\backslash\Gr^{(2)}$ in section \ref{complex}. 
induces an involution on $\Gr^{(1)}$ denoted by $\eta^{(1)}$.
For $x=0$, the involution $\eta^{(1)}$ restricts to the complex conjugation 
on $\Gr\is\Gr^{(1)}|_0$. 
Let $x=ir\neq 0$.
We shall give a description of the involution $\eta^{(1)}$ on
$\Gr^{(1)}|_{x}$. 
Consider the 
coordinate $t$ of $\mathbb P^1$ sending 
$\infty$ to $1$, $x$ to $0$, and $-x$ to $\infty$ (i.e., $t=\frac{z-x}{z+x}$). 
It induces isomorphisms 
\[
\Gr^{(2)}|_x\is G[t,t^{-1}]/G[t]\times G[t,t^{-1}]/G[t^{-1}],\ \ 
\Omega G_c^{(2)}|_i\is\Omega G_c\subset G[t,t^{-1}].\] Moreover, under the 
isomorphisms above the involution 
$\eta^{(2)}$ and 
the action of $\gamma\in\Omega G_c^{(2)}|_x\is\Omega G_c$ on $\Gr^{(2)}|_i$
are given by the formulas 
$\eta^{(2)}(g_1,g_2)=(\eta^\tau(g_2),\eta^\tau(g_1))$ 
and $\gamma(g_1,g_2)=(\gamma g_1,\gamma g_2)$.
From this, we see that under the 
homeomorphism $\Gr^{(1)}|_x\is\Omega G_c\subset G[t,t^{-1}]$ (induced by the above chosen coordinate)  
the involution  
$\eta^{(1)}$ is given by the formula 
$\gamma\to\tilde\eta^\tau(\gamma)(=\eta(\gamma(\bar t^{-1}))^{-1}),\ \gamma\in\Omega G_c$. 
We define $\calX:=(\Gr^{(1)})^{\eta^{(1)}}$ to be the fixed point of $\eta^{(1)}$ on $\Gr^{(1)}$ and we write 
\[q:\calX\ra i\bbR\]
for the restriction of the projection map $\Gr^{(1)}\to i\bbR$ to $\calX$.  
From the above discussion we see that the special fiber $\calX_0:=q^{-1}(0)$ is isomorphic to 
\[\calX_0\is\Gr_\bbR\text{\ \ (as real varieties)}\]
and the generic fiber $\calX_x:=q^{-1}(x)$ is homeomorphic to 
\beq\label{X_x}
\calX_x\is\Omega X_c=(\Omega G_c)^{\tilde\eta^\tau}=(\Omega G_c)^{\tilde\theta}.
\eeq
 Recall the $G(\mO)^{(1)}$-orbtis stratification $\{S^{\lambda,(1)}\}_{\lambda\in\Lambda_T^+}$ and $L^-G^{(1)}$-orbits 
stratification $\{T^{\lambda,(1)}\}_{\lambda\in\Lambda_T^+}$
on $\Gr^{(1)}$.
It is known that the stratification 
$\{S^{\lambda,(1)}\}_{\lambda\in\Lambda_T^+}$ 
is Whitney and the projection map
$\Gr^{(1)}\to i\bbR$ is a stratified submersion (here we equip $i\bbR$ with the trivial stratification 
$\{i\bbR\}$). Our first goal is to show that 
the intersection of $\{S^{\lambda,(1)}\cap\calX\}_{\lambda\in\Lambda_T^+}$  
forms a Whitney stratification of $\calX$ and the projection map 
$q:\calX\to i\bbR$ is a stratified submersion.

We will use the following lemma:
\begin{lemma}\cite[Lemma 4.5.1]{N1}\label{fixed points}
Let $f:M\to N$ be a stratified submersion between two 
Whitney stratified manifolds. Assume there is 
a compact group $H$ acting on $M$ and $N$ such that 
the actions preserve the stratifications 
and $f$ is 
$H$-equivariant. Then the 
fixed point manifold $M^H$ and $N^H$ are Whitney stratified by the 
fixed points of the strata and the induced map 
$f^H:M^H\to N^H$ is a stratified submersion.
\end{lemma}

Since $\calX$ is the fixed point manifold of an involution on 
$\Gr^{(1)}$ and the projection map $\Gr^{(1)}\to i\bbR$ is compatible with the involution, to achieve our goal 
it suffices to show that 
$\{S^{\lambda,(1)}\}$ is $\eta^{(1)}$-invariant. More precisely, we have the following:

\begin{lemma}\label{stable under eta}
The involution $\eta^{(1)}$ 
preserves the stratification $\{S^{\lambda,(1)}\}$ (resp. $\{T^{\lambda,(1)}\}$) and 
maps the stratum $S^{\lambda,(1)}$ (resp. $T^{\lambda,(1)}$) to 
the stratum containing the coweight $\eta(\lambda)\in\Gr\is\Gr^{(1)}|_0$.

\end{lemma}
\begin{proof}
Note that under the  
the isomorphism 
$w:\Gr^{(1)}|_{i\bbR^\times}\is\Omega G_c\times i\bbR^\times$
 the involution $\eta^{(1)}$ is given by the formula 
$\eta^{(1)}(\gamma,x)=(\tilde\eta^{\tau}(\gamma),x)$
and we have  
$w(S^{\lambda,(1)}|_{i\bbR^\times})=S^\lambda\times i\bbR^\times$. 
Since $\tilde\eta^{\tau}$ map $S^\lambda$ to the stratum 
$S^{\lambda'}:=\tilde\eta^{\tau}(S^\lambda)$
containing 
$\tilde\eta^{\tau}(\lambda)=\eta(\lambda)$ and the closure 
of $S^{\lambda,(1)}|_{i\bbR^\times}$ in $\Gr^{(1)}$ is equal to 
$\overline{S^{\lambda,(1)}|_{i\bbR^\times}}=\cup_{\mu\leq\lambda} S^{\mu,(1)}$, we have 
\[\bigcup_{\mu\leq\lambda}\eta^{(1)}(S^{\mu,(1)})=\eta^{(1)}(\overline{S^{\lambda,(1)}|_{i\bbR^\times}})=\overline{\eta^{(1)}(S^{\lambda,(1)}|_{i\bbR^\times})}=\overline{S^{\lambda',(1)}}=
\bigcup_{\mu'\leq\lambda'}S^{\lambda',(1)}
.\]
Now by induction on the dimension of 
the stratum $S^{\mu,(1)}$, we conclude that 
$\eta^{(1)}(S^{\lambda,(1)})=S^{\lambda',(1)}$. The desired claim for 
the stratification $\{S^{\lambda,(1)}\}$ follows. 

Consider the flow $\phi_x^{(1)}:\Gr^{(1)}\to\Gr^{(1)}, x\in\bbR_{>0}$ in section \ref{}. 
The critical manifolds of the flow are the cores $C^\lambda\subset\Gr\is\Gr^{(1)}|_0$
and the corresponding descending spaces are $T^{\lambda,(1)}$.
Note that the flow $\phi_x^{(1)}$ is compatible with the involution, that is,
$\phi_x^{(1)}\circ\eta^{(1)}=\eta^{(1)}\circ\phi_x^{(1)}$. Since the 
involution $\eta^{(1)}$ maps 
$C^\lambda$ to $C^{\eta(\lambda)}$, it implies $\eta^{(1)}$ preserves 
the the stratification $\{T^{\lambda,(1)}\}$ (since they are the descending spaces of the flow) and 
the image $\eta^{(1)}(T^{\lambda,(1)})$ is a stratum containing $\eta(\lambda)$.

\quash{
Note that the 
Since the flow $\phi_x^{(1)}$ on $\Gr^{(1)}$ is $K_c$-equivariant, it descends to a flow
on the quotient $K_c\backslash\Gr^{(1)}$. We claim that the resulting flow on 
$K_c\backslash\Gr^{(1)}$, denoted by $\tilde\phi_x^{(1)}$,
Consider the partial compactification 
$\Gr^{(1)}_{\bP^1}|_{ i\bbR\cup\infty}\to i\bbR\cup\infty$ of 
$\Gr^{(1)}\to i\bbR$.
The flow 
$\phi_x^{(1)}$ and the involution $\eta^{(1)}$ on $\Gr^{(1)}$ 
extend compatibly to $\Gr^{(1)}_{\bP^1}|_{ i\bbR\cup\infty}$ and the critical manifolds for the flow are 
$C^\lambda\sqcup C^\mu\subset\Gr\sqcup\Gr\is\Gr^{(1)}_{\bP^1}|_{0\cup\infty}$. 
Note that 
the stratum $S^{\lambda,(1)}$ is the ascending space in $\Gr^{(1)}$ for the  
critical manifold $C^\lambda\sqcup C^\lambda$. Since 
$\eta^{(1)}$ maps $C^\lambda\sqcup C^\lambda$ to 
$C^{\eta(\lambda)}\sqcup C^{\eta(\lambda)}$, the involution $\eta^{(1)}$
must preserve the stratification $\{S^{\lambda,(1)}\}$ and maps
$S^{\lambda,(1)}$ to the stratum containing $\eta(\lambda)$.}

\end{proof}

\begin{lemma}
The intersection $\mathcal S^\lambda:=\calX\cap S^{\lambda,(1)}$ (resp. $\mathcal T^\lambda:=\calX\cap T^{\lambda,(1)}$) is non-empty if and only if $\lambda\in\Lambda_S^+$. 
\end{lemma}
\begin{proof}
Since the special fiber 
$S^\lambda_\bbR$ (resp. $T_\bbR^\lambda$) is non-empty if and only if $\lambda\in\Lambda_S^+$, it suffices to show that 
$\calS^\lambda$ (resp. $\calT^\lambda$) is non-empty if and only if 
the special fiber $S^\lambda_\bbR$ (resp. $T^\lambda_\bbR$) is non-empty. 
For the case $\calS^\lambda$, we observe that, by lemma \ref{fixed points}
and lemma \ref{stable under eta}
, the stratification 
$\{\calS^\lambda\}$ of $\calX$ is Whitney and the 
map $q:\calX\to i\bbR$ is a ind-proper stratified submersion. 
Thus the stratum $\calS^\lambda$ is non-empty if and only if 
the special fiber $S^\lambda_\bbR$ is non-empty. 
For the case $\calT^\lambda$, we note that 
$\calX$ is closed in $\Gr^{(1)}$ and 
$\calT^{\lambda}$ is the descending space for the 
critical manifold $C^\lambda_\bbR$ 
of the flow 
$\phi_x^{(1)}$. Thus 
the intersection  
$\mathcal T^\lambda$
is non-empty 
if and only if the special fiber $T^\lambda_\bbR$ 
is non-empty. 
The lemma follows.
\end{proof}

We have proved the following: 
\begin{proposition}\label{Whitney}
The stratification $\{\calS^\lambda\}_{\lambda\in\Lambda_S^+}$
on $\calX$ is a Whitney stratification and the 
projection map $q:\calX\to i\bbR$ is a stratified submersion.
\end{proposition}

Since the projection $q$ is ind-proper and the $K_c$-action on $\calX$ preserves each
strata $\calS^\lambda$, 
by Thom's first isotopic lemma \cite{M}, we obtain:
\begin{proposition}\label{isotopic}
For any $x\in i\bbR$ we write 
$\calX_x=\calX|_x$ and $\calS^{\lambda}_x=\calS^\lambda|_x$.
There is a $K_c$-equivariant 
stratum preserving homeomorphism 
\[(\calX,\{\calS^\lambda\}_{\lambda_S^+})\is(\calX_x,\{\calS_x^\lambda\}_{\lambda\in\Lambda_S^+})\times i\bbR\]
which is real analytic on each stratum. 
In particular, for $x\in i\bbR^\times$, 
there is a $K_c$-equivariant 
stratum preserving homeomorphism 
\[(\calX_x,\{\calS_x^\lambda\}_{\lambda\in\Lambda_S^+})\is(\Gr_\bbR,\{S_\bbR^\lambda\}_{\lambda\in\Lambda_S^+})\]
which is real analytic on each stratum.

\end{proposition}

\subsection{Relation with real quasi-maps}
Observe that the inclusion map 
$\Omega K_c^{(2)}\backslash\Gr^{(2)}_\bbR\to\Omega G_c^{(2)}\backslash\Gr^{(2)}$
in section \ref{} factors through the fixed points $\calX$: 
\[\Omega K_c^{(2)}\backslash\Gr^{(2)}_\bbR\stackrel{\iota}\to\calX\to\Omega G_c^{(2)}\backslash\Gr^{(2)}.\]
For 
any $x\in i\bbR^\times$
we have an isomorphism 
$\Omega K_c^{(2)}\backslash\Gr^{(2)}_\bbR|_x\is\Omega K_c\backslash\Gr$ and we write \[\iota_x:\Omega K_c\backslash\Gr\is\Omega K_c^{(2)}\backslash\Gr^{(2)}_\bbR|_x\stackrel{}\to\calX_x\] for the induced map on fibers.  

\begin{lemma}\label{S=O}
For $\lambda\in\mL\subset\Lambda_S^+$, 
the map $\iota_x$ induces 
a $K_c$-equivariant isomorphisms of real varieties 
\[\Omega K_c\backslash\mO_K^\lambda\is\mathcal S^{\lambda}_x,\ \ \ 
\Omega K_c\backslash\mO_\bbR^\lambda\is\mathcal T^{\lambda}_x.\]
\end{lemma}
\begin{proof}
Note that we have the following commutative diagram
\[\xymatrix{\Omega K_c\backslash\Gr\ar[r]^{\iota_x}\ar[d]&\calX_x\ar[d]^{v}\\
\Omega K_c\backslash\Omega G_c\ar[r]^\pi&\Omega X_c}\]
where the left vertical arrow $v$ is the homeomorphism in (\ref{X_x}) 
and the map $\pi$ is given by $\gamma\to \theta(\gamma)^{-1}\gamma$.
In addition, we have 
$v(\calS^\lambda_x)=S^\lambda\cap\Omega X_c$ (resp. $v(\calT^\lambda_x)=T^\lambda\cap\Omega X_c$).
So it suffices to show that 
$\pi$ maps $\Omega K_c\backslash\mO_K^\lambda$ (resp. $\Omega K_c\backslash\mO_\bbR^\lambda$) homeomorphically  onto $S^\lambda\cap\Omega X_c$ (resp.
$T^\lambda\cap\Omega X_c$). This follows from the result in section \ref{orbits}.
\end{proof}

Let $\calX^0\subset\calX$ be the union of the strata $\calS^\lambda,\ \lambda\in\mL$. 
Since $\Gr_\bbR^0:=\cup_{\lambda\in\mL} S_\bbR^\lambda$ is an union of certain connected components 
of $\Gr_\bbR$,
it follows from proposition \ref{isotopic} that $\calX^0$ is also an union of certain connected components
of $\calX$. Moreover, by the lemma above we have 
\[\calX^0\subset\Omega K_c^{(2)}\backslash\Gr^{(2)}_\bbR\subset\calX.\]
\begin{remark}
If $K$ is connected, then $\mL=\Lambda_S^+$ and we have 
$\calX^0=\calX=\Omega K_c^{(2)}\backslash\Gr^{(2)}_\bbR$.
\end{remark}
Let $q^0:\calX^0\to i\bbR$ be the restriction of $q$ to $\calX^0$.
The following proposition follow from proposition \ref{Whitney}, proposition \ref{isotopic}, and lemma \ref{S=O}.

\begin{proposition}
The stratification $\{\calS^\lambda\}_{\lambda\in\mL}$ on $\calX^0$
is a Whitney stratification and the map 
\beq\label{X^0}
q^0:\calX^0\to i\bbR
\eeq is 
an ind-proper stratified submersion. 
For any $x\in i\bbR$ we write 
$\calX^0_x=\calX^0|_x$ and $\calS^{\lambda}_x=\calS^\lambda|_x$.
There is a $K_c$-equivariant 
stratum preserving homeomorphism 
\[(\calX^0,\{\calS^\lambda\}_{\lambda\in\mL})\is(\calX^0_x,\{\calS_x^\lambda\}_{\lambda\in\mL})\times i\bbR\]
which is real analytic on each stratum. 
In addition, for $x\in i\bbR^\times$ (resp. $x=0$)
there is a real analytic $K_c$-equivariant 
stratum preserving isomorphism
\[(\calX_x^0,\{\calS_x^\lambda\}_{\lambda\in\mL})\is(\Omega K_c\backslash\Gr,\{\Omega K_c\backslash\mO_K^\lambda\}_{\lambda\in\mL})\ \ 
(resp.\ \ 
(\calX_0^0,\{\calS_0^\lambda\}_{\lambda\in\mL})\is
(\Gr_\bbR^0,\{S_\bbR^\lambda\}_{\lambda\in\mL}))
.\]
In particular, we obtain 
a $K_c$-equivariant 
stratum preserving homeomorphism 
\[(\Omega K_c\backslash\Gr,\{\Omega K_c\backslash\mO_K^\lambda\}_{\lambda\in\mL})\is(\Gr^0_\bbR,\{S_\bbR^\lambda\}_{\lambda\in\mL})\]
which is real analytic on each stratum.

\end{proposition}

\subsection{A generalization}
Consider the stratification 
$\{V^{\lambda,\mu,(1)}:=S^{\lambda,(1)}\cap T^{\lambda,(1)}\}_{\lambda,\mu\in\Lambda_T^+}$ of $\Gr^{(1)}$.
Since the two stratifications $\{S^{\lambda,(1)}\}$ and $\{T^{\lambda,(1)}\}$ are transversal and both are stable under the involution 
$\eta^{(1)}$,
the results in the previous section imply
\begin{proposition}
The fixed points 
stratification $\{\calV^{\lambda,\mu}=V^{\lambda,\mu,(1)}\cap\calX^0\}_{\lambda,\mu\in\mL}$ of $\calX^0$ is a Whitney stratification and 
the map 
\beq\label{X^0}
q^0:\calX^0\to i\bbR
\eeq is 
an ind-proper stratified submersion. 
For any $x\in i\bbR$,
there is a $K_c$-equivariant 
stratum preserving homeomorphism 
\[(\calX^0,\{\calV^{\lambda,\mu}\}_{\lambda,\mu\in\mL})\is(\calX^0_x,\{\calV_x^{\lambda,\mu}\}_{\lambda,\mu\in\mL})\times i\bbR\]
which is real analytic on each stratum. 
In particular, we obtain 
a $K_c$-equivariant 
stratum preserving homeomorphism 
\[(\Omega K_c\backslash\Gr,\{\Omega K_c\backslash(\mO_K^\lambda\cap\mO_\bbR^\mu)\}_{\lambda\in\mL})\is(\Gr^0_\bbR,\{S_\bbR^\lambda\cap T^\mu_\bbR\}_{\lambda,\mu\in\mL})\]
which is real analytic on each stratum.

\end{proposition}

}

\section{Affine Matsuki correspondence for sheaves}\label{Affine Matsuki}
In this section we prove the 
affine Matsuki correspondence for sheaves.

\subsection{The functor $\Upsilon$}

Let
$u:LK_c\backslash\Gr\ra
LG_\bbR\backslash\Gr$ be the quotient map. 
Define 
\[
\Upsilon:D_c(K(\calK)\backslash\Gr)\ra D_c(LG_\bbR\backslash\Gr)\]
to be the restriction of 
$u_!:D_c(LK_\bbR\backslash\Gr)\to D_c(LG_\bbR\backslash\Gr)$
to $D_c(K(\calK)\backslash\Gr)\subset D_c(LK_\bbR\backslash\Gr)$.

\begin{thm}[Affine Matsuki correspondence for sheaves]\label{AM}
The functor $\Upsilon$ defines an equivalence of categories
\[\Upsilon:D_c(K(\calK)\backslash\Gr)\stackrel{\sim}\lra D_!(LG_\bbR\backslash\Gr).\]
\end{thm}

The rest of the section is devoted to the proof of Theorem \ref{AM}.

\subsection{Bijection between local systems}
Write $[\mO_K^\lambda]=LK_c\backslash\mO_K^\lambda$, $[\mO_\bbR^\lambda]=LK_c\backslash\mO_\bbR^\lambda$, 
$[\mO_c^\lambda]=LK_c\backslash\mO_c^\lambda$, and 
$[\mE^\lambda]=LG_\bbR\backslash\mO_\bbR^\lambda\in LG_\bbR\backslash\Gr$.
Recall the Matsuki flow $\phi_t:\Gr\to
\Gr$ in Theorem \ref{flow}. 
As $\phi_t$ is $LK_c$-equivariant, it descends to a flow
$\tilde\phi_{t}:LK_c\backslash\Gr\to LK_c\backslash\Gr$ and we define
\[\phi_\pm:LK_c\backslash\Gr\ra\bigsqcup_{\lambda\in\calL}\ [\mO_c^\lambda]\subset LK_c\backslash\Gr,\ \  \gamma\ra\underset{t\to\pm\infty}\lim\tilde\phi_t(\gamma).\]
Consider the following Cartesian diagrams:
\[\xymatrix{[\mO_K^\lambda]\ar[r]^{i_+^\lambda\ \ \ }\ar[d]^{\phi_+^\lambda}&L K_c\backslash\Gr\ar[d]^{\phi_+}
\\[\mO_c^\lambda]\ar[r]^{j_+^\lambda\ \ \ }&\bigsqcup_{\lambda\in\mL}[\mO_c^\lambda]}\ \ \ \ \ 
\xymatrix{[\mO_\bbR^\lambda]\ar[r]^{i_-^\lambda\ \ \ }\ar[d]^{\phi_-^\lambda}&L K_c\backslash\Gr\ar[d]^{\phi_-}
\\[\mO_c^\lambda]\ar[r]^{j_+^\lambda\ \ \ }&\bigsqcup_{\lambda\in\mL}[\mO_c^\lambda]}\ \ \ \ \xymatrix{[\mO_\bbR^\lambda]\ar[r]^{i_-^\lambda\ \ \ }\ar[d]^{u^\lambda}&LK_c\backslash\Gr\ar[d]^{u}
\\[\mE^\lambda]\ar[r]^{j_-^\lambda\ \ \ }&LG_\bbR\backslash\Gr}
\]
Here  
$i^\lambda_\pm$ and $j^\lambda_\pm$ are the natural 
embeddings and $\phi^\lambda_\pm$ (resp. $u^\lambda$) is the restriction of 
$\phi_\pm$ (resp. $u$) along $j^\lambda_+$ (resp. $j^\lambda_-$). 

\begin{lemma}\label{bijection}
We have the following:
\begin{enumerate}
\item
There is a bijection between isomorphism classes of 
local systems $\tau^+$ on $[\mO_K^\lambda]$,
local systems $\tau^-$ on $[\mO_\bbR^\lambda]$, 
local systems $\tau$ on $[\mO_c^\lambda]$, and local systems $\tau_\bbR$ on $[\mE^\lambda]=[LG_\bbR\backslash\mO_\bbR^\lambda]$, characterizing by the property that 
$\tau^\pm\is(\phi_\pm^\lambda)^*\tau$ and $
\tau^-\is(u^\lambda)^*\tau_\bbR$. 
\item
The map $u^\lambda$ factors as 
\beq\label{p^lambda}
u^\lambda:[\mO_\bbR^\lambda]
\stackrel{\phi_-^\lambda}\ra[\mO_c^\lambda]\stackrel{p^\lambda}\ra[\mE^\lambda]\eeq
where $p^\lambda$ is smooth of relative dimension $\on{dim}[\mE^\lambda]-\on{dim}[\mO_c^\lambda]$. Moreover, we have 
$(p^\lambda)^*\tau_\bbR\is\tau$. 
\end{enumerate}
\end{lemma}
\begin{proof}

Since the fibers of 
$\phi_\pm$ 
are contractible, pull-back along $\phi^\lambda_+$ (resp. $\phi^\lambda_-$) defines 
an equivalence between
$LK_c$-equivariant local systems on 
$\mO_c^\lambda$ 
and $LK_c$-equivariant local systems 
on $\mO_K^\lambda$ (resp. $\mO_\bbR^\lambda$).
We show that the fiber of $u^\lambda$ is contractible, hence pull back along 
$u^\lambda$ defines an equivalence between local systems on 
$[\mE^\lambda]$ and $LK_c$-equivariant local systems on $\mO_\bbR^\lambda$.
Pick $y\in\mO_c^\lambda$ and let
$LK_c(y)$, $LG_\bbR(y)$ be the stabilizers of 
$y$ in $LK_c$ and $LG_\bbR$ respectively.
The group $LK_c(y)$ acts on the fiber 
$l_y:=(\phi_-^\lambda)^{-1}(y)$ and we have 
$\mO_\bbR^\lambda\is LK_c\times^{LK_c(y)}l_y$. 
Moreover, under the isomorphism 
$[\mO_\bbR^\lambda]\is LK_c\backslash\mO_\bbR^\lambda\is
LK_c(y)\backslash l_y$, $[\mO_c^\lambda]\is LK_c(y)\backslash y$, and
$[\mE^\lambda]\is LG_\bbR(y)\backslash y$, 
the map $u^\lambda$ 
takes the form
\[
u^\lambda:[\mO_\bbR^\lambda]\is LK_c(y)\backslash l_y\stackrel{\phi_-^\lambda}\ra 
[\mO_c^\lambda]\is LK_c(y)\backslash y\stackrel{p^\lambda}\ra 
[\mE^\lambda]\is LG_\bbR(y)\backslash y,\]
where the first map is induced by the projection 
$l_y\ra y$ and 
the second map is induced by the 
inclusion $LK_c(y)\ra LG_\bbR(y)$.
We claim that the quotient $LK_c(y)\backslash LG_\bbR(y)$ is contractible, 
hence 
$u^\lambda$
has contractible fibers and $p^\lambda$ is smooth of relative dimension $\on{dim}[\mE^\lambda]-\on{dim}[\mO_c^\lambda]$. Part (1) and (2) follows.

Proof of the claim. Pick $y'\in C_\bbR^\lambda\subset\Gr_\bbR$ and let
$K_c(y')$ and $G_\bbR(\bbR[t^{-1}])(y')$ be the stabilizers of 
$y'$ in $K_c$ and $G_\bbR(\bbR[t^{-1}])$ respectively.
The composition of the complex and real uniformizations of $\Bun_G(\bP^1)_{\bbR,0}$
 \[LG_\bbR\backslash\Gr\stackrel{\on{Prop.} \ref{uniformizations at complex x}}\is \Bun_G(\bP^1)_{\bbR,0}
 \stackrel{\on{Prop.} \ref{uniformizations at real x}}\is G_\bbR(\bbR[t^{-1}])\backslash\Gr_\bbR^0\] identifies 
\[LG_\bbR(y)\backslash y\is [\mE^\lambda]\is G_\bbR(\bbR[t^{-1}])(y')\backslash y'.\] 
Hence we obtain a natural isomorphism 
\[LG_\bbR(y)\is\Aut([\mE^\lambda])\is G_\bbR(\bbR[t^{-1}])(y')\]
sending $LK_c(y)=K_c(y)\subset LG_\bbR(y)$ to $K_c(y')\subset G_\bbR(\bbR[t^{-1}])(y')$.
Thus we reduce to show that the quotient $K_c(y')\backslash G_\bbR(\bbR[t^{-1}])(y')$ is contractible.
This follows from the fact that evaluation map 
$K_c(y')\backslash G_\bbR(\bbR[t^{-1}])(y')\to K_c(y')\backslash G_\bbR(y'),\ \ \gamma(t^{-1})\to\gamma(0)$
has contractible fibers and the quotient $K_c(y')\backslash G_\bbR(y')$ is contractible as 
$K_c(y')$ is a maximal compact subgroup of the Levi subgroup of $G_\bbR(y')$.

\end{proof}


\subsection{Proof of Theorem \ref{AM}}
\quash{
Consider a diagram of closed substacks of $LK_c\backslash\Gr$
\quash{
\[\xymatrix{U_0\ar[r]\ar[drrr]&U_1\ar[r]\ar[drr]&\cdot\cdot\cdot&U_k\ar[d]\ar[r]&\cdot\cdot\cdot\\
&&&LK_c\backslash\Gr}\]}
\[
U_0\stackrel{j_0}\ra U_1\stackrel{j_1}\ra U_2\ra\cdot\cdot\cdot\ra U_k\ra\cdot\cdot\cdot\]
such that 
1) $\cup_{i} U_i=LK_c\backslash\Gr$, 
2) $U_i$ is a finite union of $[\mO_K^\lambda]$, 3) each $j_k$ is closed embedding. Let $f_i:U_i\ra LK_c\backslash\Gr$ be the natural embedding and 
we define \[s_i=u\circ f_i:U_i\ra\Bun_G(\mathbb P^1)_\bbR.\]
Note that each $s_i$ is of finite type 
and it follows from the definition that 
$f_i=f_{i+1}\circ j_i$, 
$s_{i}=s_{i+1}\circ j_i$.
We 
define 
\[u_i^!:=(f_i)_*s_i^!:D(LG_\bbR\backslash\Gr)\ra D(LK_c\backslash\Gr)\]
Let $\mF\in D(LG_\bbR\backslash\Gr)$.
The unit morphism $id\ra (j_{i-1})_*j_{i-1}^{*}\is (j_{i-1})_*j_{i-1}^{!}$ defines a map
\[p_i:u_i^!(\mF)\is(f_i)_*s_i^!\mF\ra(f_i)_*(j_{i-1})_*j_{i-1}^{!}s_i^!(\mF)\is (f_{i-1})_*s_{i-1}^!(\mF)=u_{i-1}^!(\mF).\]
Consider the following functor 
\[u^!:D(LG_\bbR\backslash\Gr)\ra\on{pro}(D(LK_c\backslash\Gr)),\ \ 
\mF\ra u^!(\mF):=\lim u_i^!(\mF)\in\on{pro}(D(LK_c\backslash\Gr)).\]
Here $u^!(\mF)$ is 
the pro-object in $D(LK_c\backslash\Gr)$
associated to the projective system 
\[u_0^!(\mF)\stackrel{\ p_1}\leftarrow u_1^!(\mF)\stackrel{p_2}\leftarrow u_2^!(\mF)\leftarrow\cdot\cdot\cdot.\]

\begin{lemma}
For any $\mM\in D(K(\calK)\backslash\Gr))$ and $\mF\in D(LG_\bbR\backslash\Gr)$, we have 
\[\on{Hom}_{D(LG_\bbR\backslash\Gr)}(u_!\circ\on{For}_+(\mM),\mF)\is
\on{Hom}_{\on{pro}(D(LK_c\backslash\Gr))}(\on{For}_+(\mM),u^!(\mF)).\]
\end{lemma}
\begin{proof}
\end{proof}
}

For each $\lambda\in\calL$ and a local system 
$\tau$ on $[\mO_c^\lambda]$
one has the standard sheaves
\beq\label{standard}
\calS_*^+(\lambda,\tau):=(i_{+}^\lambda)_*(\tau^+) \text{\ \ and\ \ } 
\calS_*^-(\lambda,\tau):=(j_{-}^\lambda)_*(\tau_\bbR)
\eeq
and co-standard sheaves
\beq\calS_!^+(\lambda,\tau):=(i_{+}^\lambda)_!(\tau^+)
\text{\ \ and\ \ } 
\calS_!^-(\lambda,\tau):=(j_-^\lambda)_!(\tau_\bbR).
\eeq
Here $\tau^+$ and $\tau_\bbR$ are local system 
on $[\mO_K^\lambda]$ and $[\mE^\lambda]$
corresponding to $\tau$
as in Lemma \ref{bijection}.
Let $d_\lambda:=\on{dim}\Bun_{G}(\mathbb P^1)_\bbR-\on{dim}[\mO_K^\lambda]$.

Write
\beq\label{iota}
\iota_+^\lambda:[\mO_c^\mu]\to[\mO_K^\mu],\ \ \iota_-^\lambda:[\mO_c^\mu]\to[\mO_\bbR^\mu]
\eeq
for the natural embeddings.
We recall the following fact, see \cite[Lemma 5.4]{MUV}.
\begin{lemma}\label{Braden}
\begin{enumerate}
\item
Consider $[\mO_c^\mu]\stackrel{\iota_+^\mu}\to[\mO_K^\mu]\stackrel{\phi_+^\mu}\to[\mO_c^\mu]$.
Let $\mF\in D_c([\mO_K^\mu])$. If $\mF$ is 
smooth (= locally constant) on the trajectories of the flow $\tilde\phi_t$, then we have canonical isomorphisms 
$(\iota_+^\mu)^!\mF\is(\phi_+^\mu)_!\mF$  and 
$(\iota_+^\mu)^*\mF\is(\phi_+^\mu)_*\mF$.
\item
Consider $[\mO_c^\mu]\stackrel{\iota_-^\mu}\to[\mO_\bbR^\mu]\stackrel{\phi_-^\mu}\to[\mO_c^\mu]$ where $\iota_-^\mu$ is the natural embedding. 
Let $\mF\in D_c([\mO_\bbR^\mu])$. If $\mF$ is 
smooth (= locally constant) on the trajectories of the flow $\tilde\phi_t$ and is supported on a finite dimensional substack $\sY\subset [\mO_\bbR^\mu]$, then we have canonical isomorphisms 
$(\iota_-^\mu)^!\mF\is(\phi_-^\mu)_!\mF$  and 
$(\iota_-^\mu)^*\mF\is(\phi_-^\mu)_*\mF$.
\end{enumerate}

\end{lemma}

We shall show that the functor $\Upsilon$ sends 
standard sheaves to co-standard sheaves.
Introduce the following local system on $[\mO_c^\lambda]$
\beq\label{L_lambda}
\mL_\lambda:=(\iota_-^\lambda)^*\mL_\lambda'\otimes\mL_\lambda''\otimes\on{or}_{p^\lambda}^\vee
\eeq
where 
\beq
(\mL'_\mu)^\vee:=(i_-^\mu)^!(\bC)[\on{codim}[\mO_\bbR^\mu]] \ \text{\ and\ \ }  
\mL_\lambda'':=(\iota_+^\lambda)^!\bC[\on{codim}_{[\mO_K^\lambda]}[\mO_c^\lambda]]
\eeq
are local systems 
on $[\mO_\bbR^\lambda]$ and $[\mO_c^\lambda]$ respectively and 
$\on{or}_{p^\lambda}:=(p^\lambda)^!\bC[-\on{dim}[\mE^\lambda]+\on{dim}[\mO_c^\lambda]]$ is the orientation sheaf for the smooth map $p^\lambda:[\mO_c^\lambda]\to[\mE^\lambda]$ in \eqref{p^lambda}.

\begin{lemma}\label{surjective}
For any local system 
$\tau$ on $[\mO_c^\lambda]$ we have 
\[\Upsilon(S^+_*(\lambda,\tau))\is S^-_!(\lambda,\tau\otimes\mL_\lambda)[d_\lambda].\]

\end{lemma}
\begin{proof}
Let $\lambda,\mu\in\mL$.
Consider the following diagram 
\beq\label{key diagram}
\xymatrix{[\mO_K^\lambda\cap\mO_\bbR^\mu]\ar[r]^s\ar[d]^\iota
&[\mO_\bbR^\mu]\ar[r]^{u^\mu\ \ \ \ \ \ \ }\ar[d]^{i_-^\mu}&[\mE^\mu]=LG_\bbR\backslash\mO_\bbR^\mu\ar[d]^{j_-^\mu}\\
[\mO_K^\lambda]\ar[r]^{i_+^\lambda}&LK_c\backslash\Gr\ar[r]^{u\ }&LG_\bbR\backslash\Gr}.
\eeq
Let 
$\mG=(j_-^\mu)^*\Upsilon(S_*(\lambda,\tau))\is (j_-^\mu)^*u_!(i_+^\lambda)_*(\tau^+)\is(u^\mu)_!(i_-^\mu)^*(i_+^\lambda)_*(\tau^+)$.
It suffices to show that $\mG\is 0$ if $\lambda\neq\mu$ and $\mG\is(\tau\otimes\mL_\lambda)_\bbR$ if $\lambda=\mu$.

By Corollary \ref{transversal}, the orbits $\mO_\bbR^\mu$ and $\mO_K^\lambda$ are trasversal to each 
other, hence we have 
\beq\label{upper shriek}
(i_-^\mu)^*(i_+^\lambda)_*(\tau^+)\is(i_-^\mu)^!(i_+^\lambda)_*(\tau^+)\otimes\mL'_\mu[\on{codim}[\mO_\bbR^\mu]].
\eeq
where 
\beq\label{L'}
(\mL'_\mu)^\vee=(i_-^\mu)^!(\bC)[\on{codim}[\mO_\bbR^\mu]]
\eeq
is a local system on $[\mO_\bbR^\mu]$.
Thus 
\[
\mG\is(u^\mu)_!(i_-^\mu)^*(i_+^\lambda)_*(\tau^+)\stackrel{\eqref{upper shriek}}\is(u^\mu)_!((i_-^\mu)^!(i_+^\lambda)_*(\tau^+)\otimes\mL_\mu')[\on{codim}[\mO_\bbR^\mu]]\is
\]
\[
(u^\mu)_!(s_*\iota^!(\tau^+)\otimes\mL_\mu')[\on{codim}[\mO_\bbR^\mu]].
\]

According to Lemma \ref{bijection} the map $u^\mu$ factors as 
\[u^\mu:[\mO_\bbR^\mu]\stackrel{\phi_-^\mu}\ra[\mO_c^\mu]\stackrel{p^\mu}\ra[\mE^\mu]\]
where $p^\mu$ is smooth of relative dimension $\on{dim}[\mE^\mu]-\on{dim}[\mO_c^\mu]$. Since $s_*\iota^!(\tau^+)[\on{codim}[\mO_\bbR^\mu]]\in D_c([\mO_\bbR^\mu])$
is smooth on the trajectories of the flow $\tilde\phi_t$, by Lemma \ref{Braden}, we have 
\beq\label{mG}
\mG\is u^\mu_!(s_*\iota^!(\tau^+)\otimes\mL_\mu')[\on{codim}[\mO_\bbR^\mu]]
\is p^\mu_!(\phi_-^\mu)_!(s_*\iota^!(\tau^+)\otimes\mL_\mu')[\on{codim}[\mO_\bbR^\mu]]
\stackrel{\on{Lem}\ref{Braden}}\is
\eeq
\[\is p^\mu_!(\iota_-^\mu)^!(s_*\iota^!(\tau^+)\otimes\mL_\mu')[\on{codim}[\mO_\bbR^\mu]].\]
Here $\iota_-^\mu:[\mO_c^\mu]\ra[\mO_\bbR^\mu]$ is the embedding.

If $\lambda\neq\mu$ then $[\mO_\bbR^\mu]\cap[\mO_c^\lambda]$ is empty, thus we have \[(\iota_-^\mu)^!s_*\iota^!(\tau^+)[\on{codim}[\mO_\bbR^\mu]]=0\]
and \eqref{mG} implies $\mG=0$. 

If $\lambda=\mu$, then 
$[\mO_\bbR^\lambda]\cap[\mO_K^\lambda]=[\mO_c^\lambda]$,
$s=\iota_-^\lambda$, $\iota=\iota_+^\lambda$ are closed embeddings 
and by Lemma \ref{bijection} we have 
\[
(u^\lambda)_!(s)_*(\tau)\is (p^\lambda)_!(\tau)\is\tau_\bbR\otimes (p^\lambda)_!(\bC)\is
\tau_\bbR\otimes(\on{or}_{p^\lambda}^{\vee})_\bbR[\on{dim}[\mE^\mu]-\on{dim}[\mO_c^\mu]],
\]
\[\iota^!(\tau^+)\is \iota^*(\tau^+)\otimes \iota^!\bC\is\tau\otimes (\iota_+^\lambda)^!\bC\is\tau\otimes\mL''_\lambda[-\on{codim}_{[\mO_K^\lambda]}[\mO_c^\lambda]]\]
where $\on{or}_{p^\lambda}$ is the relative orientation sheaf on $[\mO_c^\lambda]$ associated to $p^\lambda:[\mO_c^\lambda]\to[\mE^\lambda]$
and 
\beq\label{L''}
\mL_\lambda'':=(\iota_+^\lambda)^!\bC[\on{codim}_{[\mO_K^\lambda]}[\mO_c^\lambda]]
\eeq
is a local system on $[\mO_c^\lambda]$.
Now an elementary calculation shows that 
\[\mG\stackrel{\eqref{mG}}\is (u^\lambda)_!(s_*\iota^!(\tau^+)\otimes\mL_\lambda')[\on{codim}[\mO_\bbR^\lambda]]\is
(u^\lambda)_!(s_*(\tau\otimes\mL_\lambda'')\otimes\mL_\lambda')[\on{codim}[\mO_\bbR^\lambda]-\on{codim}_{[\mO_K^\lambda]}[\mO_c^\lambda]]\is\]
\[\is(u^\lambda)_!(\iota_-^\lambda)_*(\tau\otimes\mL_\lambda''\otimes (\iota_-^\lambda)^*\mL_\lambda')[\on{codim}[\mO_\bbR^\lambda]-\on{codim}_{[\mO_K^\lambda]}[\mO_c^\lambda]]\is
(\tau\otimes\mL_\lambda)_\bbR[d_\lambda]
,\]
where 
\beq\label{L_lambda}
\mL_\lambda:=(\iota_-^\lambda)^*\mL_\lambda'\otimes\mL_\lambda''\otimes\on{or}_{p^\lambda}^\vee
\eeq
is a local sytem on $[\mO_c^\lambda]$ and
\[d_\lambda=\on{codim}[\mO_\bbR^\lambda]-
\on{codim}_{[\mO_K^\lambda]}[\mO_c^\lambda]+\on{dim}[\mE^\lambda]-\on{dim}[\mO_c^\lambda]=\on{dim}\Bun_{G}(\mathbb P^1)_\bbR-\on{dim}[\mO_K^\lambda].\]
The lemma follows.

\end{proof}

We shall show that $\Upsilon$ is fully-faithful.
Consider a diagram of closed substacks of $LK_c\backslash\Gr$
\quash{
\[\xymatrix{U_0\ar[r]\ar[drrr]&U_1\ar[r]\ar[drr]&\cdot\cdot\cdot&U_k\ar[d]\ar[r]&\cdot\cdot\cdot\\
&&&LK_c\backslash\Gr}\]}
\[
U_0\stackrel{j_0}\ra U_1\stackrel{j_1}\ra U_2\ra\cdot\cdot\cdot\ra U_k\ra\cdot\cdot\cdot\]
such that 
\begin{enumerate}
\item $\bigcup_{i} U_i=LK_c\backslash\Gr$, 
\item Each $U_i$ is a finite union of $[\mO_K^\lambda]$, 
\item Each $j_k$ is closed embedding. 
\end{enumerate}

Let $f_i:U_i\ra LK_c\backslash\Gr$ be the natural embedding and 
we define \[s_i=u\circ f_i:U_i\ra LG_\bbR\backslash\Gr.\]
Note that each $s_i$ is of finite type.

\begin{lemma}\label{full}
For any $\mF,\mF'\in D_c(K(\calK)\backslash\Gr)$
we have \[\on{Hom}_{D_c(K(\calK)\backslash\Gr)}(\mF,\mF')
\is\on{Hom}_{D_!(LG_\bbR\backslash\Gr)}(\Upsilon(\mF),\Upsilon(\mF')).\]

\end{lemma}
\begin{proof}
Choose $k$ such that 
$\mF=(j_k)_*\mF_k$ and $(j_k)_*\mF'_k$ for 
$\mF_k,\mF'_k\in D_c(K(\calK)\backslash U_k)$.
We have \[\on{Hom}_{D_c(K(\calK)\backslash\Gr)}(\mF,\mF')\is
\on{Hom}_{D_c(K(\calK)\backslash U_k)}(\mF_k,\mF'_k)\]
and \[\on{Hom}_{D_!(LG_\bbR\backslash\Gr)}(\Upsilon(\mF),\Upsilon(\mF'))\is
\on{Hom}_{D_c(K(\calK)\backslash U_k)}((s_k)_!\mF_k,(s_k)_!\mF_k')\is\]
\[\is
\on{Hom}_{D_c(K(\calK)\backslash U_k)}(\mF_k,(s_k)^!(s_k)_!\mF_k').\]
We have to show that the map 
\beq\label{iso}
\on{Hom}_{D_c(K(\calK)\backslash U_k)}(\mF_k,\mF'_k)\ra
\on{Hom}_{D_c(K(\calK)\backslash U_k)}(\mF_k,(s_k)^!(s_k)_!\mF_k')
\eeq
is an isomorphism.
Since $D_c(K(\calK)\backslash U_k)$ is generated by 
$w_!(\tau_\lambda^+)$ (resp. $w_*(\tau_\lambda^+)$) for 
$[\mO_K^\lambda]\subset U_k$ (here $w_\lambda:[\mO_K^\lambda]\ra U_k$ is 
natural inclusion), it suffices to verify 
(\ref{iso}) for 
\[\mF_k=(w_\lambda)_!(\tau_\lambda^+)\text{\ \ and \ \ }\mF'_k
\is (w_\mu)_*(\tau_\mu^+).\] 
Note that in this case the left hand side of \eqref{iso} becomes
\beq\label{iso Hom trivial, K}
\on{Hom}_{D_c(K(\calK)\backslash U_k)}(\mF_k,\mF'_k)=0\text{\ \ if \ \ }\lambda\neq\mu
\eeq
\beq\label{iso Hom, K}
\on{Hom}_{D_c(K(\calK)\backslash U_k)}(\mF_k,\mF'_k)\is
\on{Hom}_{D_c([\mO_c^\lambda])}(\tau_\lambda,\tau_\lambda)
\text{\ \ if \ \ }\lambda=\mu.
\eeq
By Lemma \ref{surjective} we have 
\[(s_k)_!((w_\mu)_*(\tau_\mu^+))
\is u_!(j_k)_*(w_\mu)_*(\tau_\mu^+))\is 
\Upsilon(S_*^+(\mu,\tau_\mu))\is
(j_-^\mu)_!(\tilde\tau_{\mu,\bbR})[d_\mu],\] 
where $\tilde\tau_{\mu,\bbR}=\tau_{\mu,\bbR}\otimes\mL_{\mu,\bbR}$.
Therefore the right hand side of \eqref{iso} becomes
\beq\label{iso on hom,1}
\on{Hom}_{D_c(K(\calK)\backslash U_k)}(\mF_k,
(s_k)^!(s_k)_!\mF_k')\is
\eeq
\beqn
\is\on{Hom}_{D_c(K(\calK)\backslash U_k)}((w_\lambda)_!(\tau_\lambda^+),
(s_k)^!(s_k)_!((w_\mu)_*(\tau_\mu^+))\is
\eeqn
\beqn
\is
\on{Hom}_{D_c([\mO_K^\lambda]}(\tau^+_\lambda,
w_\lambda^!(s_k)^!(j_-^\mu)_!(\tilde\tau_{\mu,\bbR})[d_\mu])\is
\eeqn
\beqn
\is\on{Hom}_{D_c([\mO_K^\lambda])}(\tau_\lambda^+,
(u\circ i_+^\lambda)^!(j_-^\mu)_!(\tilde\tau_{\mu,\bbR})[d_\mu]).
\eeqn
Since $u\circ i_+^\lambda$ and $j_-^\mu$
are transversal, 
we have 
\[(u\circ i_+^\lambda)^!(j_-^\mu)_!\tau_{\mu,\bbR}\is
(u\circ i_+^\lambda)^*(j_-^\mu)_!\tau_{\mu,\bbR}\otimes\mL_\lambda'''[-d_\lambda]\] 
where 
\beq\label{L}
\mL_\lambda'''=(u\circ i_+^\lambda)^!\bC[d_\lambda]
\eeq
is a local system on $[\mO_K^\lambda]$
and, in view of the diagram \eqref{key diagram}, we get 
\beq\label{iso hom, 2}
\on{Hom}_{D_c(K(\calK)\backslash U_k)}(\mF_k,
(s_k)^!(s_k)_!\mF_k')\stackrel{\eqref{iso on hom,1}}\is\on{Hom}_{D_c([\mO_K^\lambda])}(\tau_\lambda^+,
(u\circ i_+^\lambda)^!(j_-^\mu)_!\tilde\tau_{\mu,\bbR}[d_\mu])\is
\eeq
\beqn
\is
\on{Hom}_{D_c([\mO_K^\lambda])}(\tau_\lambda^+,
(u\circ i_+^\lambda)^*(j_-^\mu)_!\tilde\tau_{\mu,\bbR}\otimes\mL[d_\mu-d_\lambda])
\eeqn
\beqn
\is
\on{Hom}_{D_c([\mO_K^\lambda])}((\phi_+^\lambda)^*\tau_\lambda,
\iota_!(u^\mu\circ s)^*\tilde\tau_{\mu,\bbR}\otimes\mL[d_\mu-d_\lambda])
\eeqn
\beqn
\is\on{Hom}_{D_c([\mO_K^\lambda])}((\phi_+^\lambda)^*\tau_\lambda,
\iota_!(u^\mu\circ s)^*\tilde\tau_{\mu,\bbR}\otimes\mL[d_\mu-d_\lambda])
\eeqn
\beqn
\on{Hom}_{D_c([\mO_c^\lambda])}(\tau_\lambda,
(\phi_+^\lambda)_*(\iota_!(u^\mu\circ s)^*\tilde\tau_{\mu,\bbR}\otimes\mL)[d_\mu-d_\lambda]).
\eeqn
\\
Consider the case $\lambda\neq\mu$. Then 
by Lemma \ref{Braden} we have 
\[(\phi_+^\lambda)_*(\iota_!(u^\mu\circ s)^*\tilde\tau_{\mu,\bbR}\otimes\mL)[d_\mu-d_\lambda])\is
(\phi_+^\lambda)_*\iota_!((u^\mu\circ s)^*\tilde\tau_{\mu,\bbR}\otimes\iota^*\mL))[d_\mu-d_\lambda])\is\]
\[\is
(\iota^\lambda_+)^*
\iota_!((u^\mu\circ s)^*\tilde\tau_{\mu,\bbR}\otimes\iota^*\mL))[d_\mu-d_\lambda])=
0,\] 
here $\iota_+^\lambda:[\mO_c^\lambda]\to[\mO_K^\lambda]$,
and it follows from \eqref{iso hom, 2} that 
\beq\label{iso Hom trivial, R}
\on{Hom}_{D_c(K(\calK)\backslash U_k)}(\mF_k,
(s_k)^!(s_k)_!\mF_k')=0.
\eeq
Hence we have
\beqn
\on{Hom}_{D_c(K(\calK)\backslash U_k)}(\mF_k,\mF_k')\stackrel{}\is
\on{Hom}_{D_c(K(\calK)\backslash U_k)}(\mF_k,
(s_k)^!(s_k)_!\mF_k')\is 0 \text{\ \ if\ \ }\lambda\neq\mu.
\eeqn
\\
Consider the case  $\lambda=\mu$. We have 
\[(\phi_+^\lambda)_*(\iota_!(u^\lambda\circ s)^*\tilde\tau_{\lambda,\bbR}\otimes\mL)[d_\lambda-d_\lambda])\is
(u^\lambda\circ s)^*\tilde\tau_{\lambda,\bbR}\otimes\iota^*\mL_\lambda'''
\is\tau_\lambda\otimes\mL_{\lambda}\otimes \iota^*\mL'''_\lambda.\] 
We claim that $\mL_{\lambda}\otimes\iota^*\mL'''_\lambda\is\bC$ is the trivial local system
hence above isomorphism implies 
\[(\phi_+^\lambda)_*(\iota_!(u^\lambda\circ s)^*\tilde\tau_{\lambda,\bbR}\otimes\mL)[d_\lambda-d_\lambda])\is\tau_\lambda,\ \text{\ if\ }
\lambda=\mu,\]
and by \eqref{iso hom, 2}, we obtain 
\beq\label{iso Hom, R}
\on{Hom}_{D_c(K(\calK)\backslash U_k)}(\mF_k,
(s_k)^!(s_k)_!\mF_k')\is\on{Hom}_{D_c([\mO_c^\lambda])}(\tau_\lambda,\tau_\lambda)
.
\eeq
By unwinding the definition of the map in \eqref{iso}, we obtain that \eqref{iso} satisfies
\beqn
\xymatrix{
\on{Hom}_{D_c(K(\calK)\backslash U_k)}(\mF_k,
\mF_k')\ar[rr]^{\eqref{iso}}\ar[rd]^\sim_{\eqref{iso Hom, K}}&&\on{Hom}_{D_c(K(\calK)\backslash U_k)}(\mF_k,
(s_k)^!(s_k)_!\mF_k')\ar[ld]_\sim^{\eqref{iso Hom, R}}\\
&\on{Hom}_{D_c([\mO_c^\lambda])}(\tau_\lambda,\tau_\lambda)&},
\eeqn
hence is an isomorphism.
The lemma follows.

To prove the claim, we observe that, up to cohomological shifts, we have 
\[\mL_\lambda\stackrel{\eqref{L_lambda}}\is(s)^*\mL_\lambda'\otimes\mL_\lambda''\otimes\on{or}_{p^\lambda}^\vee
\is (p^\lambda)^*((j_-^\lambda)^!\bC)^\vee)\otimes\iota^!\bC\otimes\on{or}_{p^\lambda}^\vee[-]
\]
\[\iota^*\mL_\lambda'''\stackrel{\eqref{L}}\is\iota^*((u\circ i_+^\lambda)^!\bC)[-].\]
Using the canonical isomorphisms $\iota^!(-)\is\iota^*(-)\otimes\iota^!\bC$ and 
$(p^\lambda)^!(-)\is (p^\lambda)^*(-)\otimes\on{or}_{p^\lambda}[-]$, we see that 
\[\mL_\lambda\otimes\iota^*\mL_\lambda'''\is (p^\lambda)^*((j_-^\lambda)^!\bC)^\vee)\otimes\on{or}_{p^\lambda}^\vee\otimes
\iota^!((u\circ i_+^\lambda)^!\bC)[-]\is\]
\[\is(p^\lambda)^*((j_-^\lambda)^!\bC)^\vee)\otimes\on{or}_{p^\lambda}^\vee\otimes
(p^\lambda)^!((j_-^\lambda)^!\bC))[-]\is
\]
\[\is(p^\lambda)^*((j_-^\lambda)^!\bC)^\vee)\otimes(p^\lambda)^*((j_-^\lambda)^!\bC))[-]
\is\bC[-].\]
The claim follows.
\end{proof}

It follows from 
Lemma \ref{surjective} and Lemma \ref{full} that the image of $\Upsilon$ is equal to $D_!(LG_\bbR\backslash\Gr)$ and the resulting functor 
$\Upsilon:D_c(K(\calK)\backslash\Gr)\ra D_!(LG_\bbR\backslash\Gr)$
is fully-faithful, hence an equivalence. 
This finishes the proof of Theorem \ref{AM}.

\quash{
\begin{lemma}
\begin{enumerate}
\item
The functor 
$\on{For}_+$ admits a right adjoint functor 
which we denote it by 
$\on{Av}_+^r:D(LK_c\backslash\Gr)\ra D(K(\calK)\backslash\Gr)$. 
That is, we have 
\[
\on{Hom}_{D(LK_c\backslash\Gr)}(\on{For}_+(\mM),\mF)\is\on{Hom}_{D(K(\mathcal K)\backslash\Gr)}(\mM,\on{Av}_+^r\mF).\]
 \item The right adjoint $\on{Av}_+^r$ has a unique extension  $\on{Av}_+^r:\on{pro}(D(LK_c\backslash\Gr))
\ra\on{pro}(D(K(K)\backslash\Gr))$ such that 
\[\on{Hom}_{\on{pro}(D(LK_c\backslash\Gr))}(\on{For}_+(\mM),\mF)\is\on{Hom}_{\on{pro}(D(K(\mathcal K)\backslash\Gr))}(\mM,\on{Av}_+^r\mF)\]

\end{enumerate}
\end{lemma}

The lemmas above imply that the functor 
\[\Upsilon_-:=\on{Av}^r_+\circ u^!:D_!(LG_\bbR\backslash\Gr)\ra D(K(\calK)\backslash\Gr)\]
defines a right adjoint of $\Upsilon_+$.
\begin{proposition}
The functor $\Upsilon_-$ is the inverse equivalence of 
$\Upsilon_+$.
\end{proposition}}

\quash{
\begin{lemma}
The left adjoint functor of $\on{For}_-$ exists on the full subcategory 
$\on{For}_+(D(K(\calK)\backslash\Gr))\subset D(LK_c\backslash\Gr)$. 
More precisely, there exists a functor 
\[\on{Av}_-^l:\on{For}_+(D(K(\calK)\backslash\Gr))\ra D_!(LG_\bbR\backslash\Gr)\]
such that for
any $\mM\in D(K(\calK)\backslash\Gr)$, $\mF\in D_!(LG_\bbR\backslash\Gr)$
we have 
\[\on{Hom}_{D(LK_c\backslash\Gr)}(\on{For}_+(\mM),\on{For}_-(\mF))\is\on{Hom}_{D_!(LG_\bbR\backslash\Gr)}(\on{Av}_-^l\circ\on{For}_+(\mM),\mF).\]
\end{lemma}}

\section{Nearby cycles functors and the Radon transform}\label{nearby cycles and Radon TF}
We study the nearby cycles functors associated to the 
quasi-maps in Section \ref{QMaps} and the Radon transform 
for the real affine Grassmannian.

\subsection{A square of equivalences}
Recall the quasi-map family 
$QM^{(2)}(\bP^1,G,K)_{\bbR,0}\to\bP^1(\bC)$ in Section \ref{real form of QM}.
By 
Proposition \ref{open sub families of quasi maps},
we have the following cartesian diagram
\[\xymatrix{(LK_c\backslash\Gr)\times i\mathbb R_{>0}\ar[r]^j\ar[d]^{f^0}&
QM^{(2)}(\bP^1,G,K)_{\mathbb R,0}|_{i\bbR_{\geq 0}}\ar[d]^{f}&K_c\backslash\Gr^0_\mathbb R\ar[d]^{f_0}\ar[l]_{\ \ \ \ \ \ i}\\
LG_\bbR\backslash\Gr\times i\mathbb R_{>0}\ar[r]^{\bar j}\ar[d]&\Bun_G(\mathbb P^1)_{\bbR,0}\times i\mathbb R_{\geq 0}\ar[d]&G_\bbR(\bbR[t^{-1}])\backslash\Gr^0_\bbR\ar[l]_{\bar i}\ar[d]\\
i\mathbb R_{>0}\ar[r]&i\mathbb R_{\geq 0}&\{0\}\ar[l]}.\]
Define the following nearby cycles functors 
\beq
\Psi:D_c(LK_c\backslash\Gr)\ra D_c(K_c\backslash\Gr_\bbR^0),\ \ \mF\ra\Psi(\mF):=i^!j_!(\mF\boxtimes\bC_{i\bbR_{>0}}),
\eeq
\beq
\Psi_{\bbR}:D_c(LG_\mathbb R\backslash\Gr)\ra D_c(G_\mathbb R(\mathbb R[t^{-1}])\backslash\Gr^0_\mathbb R),\ \ 
\mF\ra \Psi_{\bbR}(\mF)=(\bar i)^!(\bar j)_!(\mF\boxtimes\bC_{i\bbR_{>0}}).
\eeq
We also have the Radon transform 
\beq\label{Radon TF}
\Upsilon_\bbR: D_c(G_\bbR(\mO_\bbR)\backslash\Gr_\bbR)\ra
D_c(G_\bbR(\bbR[t^{-1}])\backslash\Gr_\bbR)
\eeq
given by the restriction to $D(G_\bbR(\mO_\bbR)\backslash\Gr_\bbR)\subset
D_c(G_\bbR\backslash\Gr_\bbR)$
of the 
push-forward $p_!:D_c(G_\bbR\backslash\Gr_\bbR)\ra
D_c(G_\bbR(\bbR[t^{-1}])\backslash\Gr_\bbR)$
along the quotient map 
$p:G_\bbR\backslash\Gr\to G_\bbR(\bbR[t^{-1}])\backslash\Gr_\bbR$.

Here are the main results of this section.
\begin{thm}\label{equ}
The nearby cycles functors and the Radon transform
induce equivalences of categories:
\[\Psi:D_c(K(\calK)\backslash\Gr)\stackrel{\sim}\lra D_c(G_\bbR(\mO_\bbR)\backslash\Gr^0_\bbR),\] 
\[\Psi_{\bbR}:D_!(LG_\mathbb R\backslash\Gr)\stackrel{\sim}\lra D_!(G_\mathbb R(\mathbb R[t^{-1}])\backslash\Gr^0_\mathbb R),\]
\[\Upsilon_\bbR:D_c(G_\bbR(\mO_\bbR)\backslash\Gr_\bbR)\stackrel{\sim}\lra
D_!(G_\bbR(\bbR[t^{-1}])\backslash\Gr_\bbR).\]

\end{thm}

\begin{thm}\label{diagram}
We have a commutative square of equivalences 
\[
\xymatrix{D_c(K(\calK)\backslash\Gr)\ar[r]^{\Psi}\ar[d]^{\Upsilon}&D_c(G_\mathbb R(\mO_\mathbb R)\backslash\Gr^0_\mathbb R)
\ar[d]^{\Upsilon_\bbR}\\
D_!(LG_{\mathbb R}\backslash\Gr)\ar[r]^{\Psi_{\bbR}\ \ }&D_!(G_\mathbb R(\mathbb R[t^{-1}])\backslash\Gr^0_\mathbb R).}
\]
Here $\Upsilon$ is the affine Matsuki correspondence for sheaves.
\end{thm}

The rest of the section is devoted to the proof of Theorem \ref{equ} and Theorem \ref{diagram}.

\subsection{Images of standard sheaves under $\Psi$}
We begin with the following constructibility result for $\Psi$.
\begin{lemma}\label{const}
We have $\Psi(\mF)\in D_c(G_\bbR(\mO_\bbR)\backslash\Gr^0_\bbR)$
for any $\mF\in D_c(K(\calK)\backslash\Gr)$.

\end{lemma}
\begin{proof}
Consider the stratification of $\calS_1$ of 
$\Omega K_c\backslash\Gr$ with strata 
the $\Omega K_c$-quotients of $K(\calK)$-orbits.
By Lemma \ref{Whitney for QM_R}, 
the pull-back of $\calS_1$
under the projection map 
 \[\QM^{(2)}(\bbP^1,G,K,\infty)_{\bbR,0}|_{i\bbR_{>0}}\is \Omega K_c\backslash\Gr\times i\bbR_{>0}\ra
\Omega K_c\backslash\Gr\] 
together with the 
$G_\bbR(\mO_\bbR)$-orbits stratification of
$\Gr_\bbR^0\is\QM^{(2)}(\bbP^1,G,K)_{\bbR,0}|_0$ 
forms a Whitney stratification of $\QM^{(2)}(\bbP^1,G,K)_{\bbR,0}|_{i\bbR_{\geq 0}}$.
Moreover, 
the map 
\[\QM^{(2)}(\bbP^1,G,K,\infty)_{\bbR,0}|_{i\bbR_{\geq 0}}\ra i\bbR_{\geq 0}\] is 
Thom-stratified with respect to 
the above stratification of $\QM^{(2)}(\bbP^1,G,K,\infty)_{\bbR,0}|_{i\bbR_{\geq 0}}$
and the stratification 
$i\bbR_{\geq 0}=i\bbR_{>0}\cup\{0\}$. 
Since \[K_c\backslash\QM^{(2)}(\bbP^1,G,K,\infty)_{\bbR,0}|_{i\bbR_{\geq 0}}\is \QM^{(2)}(\bbP^1,G,K)_{\bbR,0}|_{i\bbR_{\geq 0}}\] as stacks over $i\bbR_{\geq 0}$ (see (\ref{QM=QM(infty)/K})), the nearby cycles functor $\Psi$ takes 
$\{LK_c\backslash\mO_K^\lambda\}_{\lambda\in\mL}$-constructible complexes on 
$LK_c\backslash\Gr$ to $\{K_c\backslash S_\bbR^\lambda\}_{\lambda\in\mL}$-constructible complexes on $K_c\backslash\Gr_\bbR^0$. The lemma follows.

\end{proof}

By the lemma above, the nearby cycles functor $\Psi$ restricts to a functor 
\[\Psi:D_c(K(\calK)\backslash\Gr)\ra D_c(G_\bbR(\mO_\bbR)\backslash\Gr_\bbR^0).\]
We shall show that $\Psi$ sends standard sheaves to standard sheaves.
Recall the flow \[\psi_z^3:\QM^{(2)}(\bbP^1,G,,X,\infty)_\bbR\ra\QM^{(2)}(\bbP^1,G,X,\infty)_\bbR\] in \S\ref{flows}.
For $\lambda\in\Lambda_S^+$, we have the 
critical manifold $C_\bbR^\lambda$, the stable manifold 
$S_\bbR^\lambda$, and the unstable manifold $\tilde T_\bbR^\lambda$.
We write 
\beqn 
s_\lambda^+:S_\bbR^\lambda\ra\QM^{(2)}(\bbP^1,G,K,\infty)_\bbR,\ \ \ \ 
\tilde t_\lambda:\tilde T_\bbR^\lambda\ra\QM^{(2)}(\bbP^1,G,K,\infty)_\bbR
\eeqn
for the inclusion maps and we write 
\beqn
c_\lambda^+:S_\bbR^\lambda\ra C_\bbR^\lambda,\ \ \ \ \tilde d_\lambda:\tilde T_\bbR^\lambda\ra C_\bbR^\lambda
\eeqn
for the contraction maps. 
Note that all the  
maps above are $K_c$-equivalent with respect to natural $K_c$-actions.
The following lemma follows from a topological version of Braden's theorem,
see \cite[Theorem 9.2]{N2}.
\begin{lemma}\label{hyperbolic 1}
For every $\mF\in D_c(\QM^{(2)}(\bbP^1,G,K,\infty)_\bbR)$ which is $\bbR_{>0}$-constructible 
with respect to the flow $\psi_z$, we have 
\[(c^+_\lambda)_*(s^+_\lambda)^!\mF\is(\tilde d_\lambda)_!(\tilde t_\lambda)^*\mF.\]
\end{lemma}

Recall that, by Lemma \ref{properties of flows}, we have isomorphisms 
$\tilde T_\bbR^\lambda|_i\is\Omega K_c\backslash\mO_\bbR^\lambda$, $\tilde T_\bbR^\lambda|_0\is T_\bbR^\lambda$,
for $\lambda\in\mL$ and we write 
\beqn
s_\lambda^-:T_\bbR^\lambda\ra\QM^{(2)}(\bbP^1,G,K,\infty)_\bbR,\ \ \ 
t_\lambda:\Omega K_c\backslash\mO_\bbR^\lambda\ra\QM^{(2)}(\bbP^1,G,K,\infty)_\bbR
\eeqn
for the restriction of $\tilde t_\lambda$ and 
\beqn
c_\lambda^-:T_\bbR^\lambda\ra C_\bbR^\lambda,\ \ \ 
d_\lambda:\Omega K_c\backslash\mO_\bbR^\lambda\ra C_\bbR^\lambda
\eeqn
for the restriction of the contractions $\tilde d_\lambda$. 

\begin{lemma}\label{hyperbolic 2}
For every $\mF\in D_c(\Omega K_c\backslash\Gr)$ which is $\bbR_{>0}$-constructible 
with respect to the flow $\psi_z$, we have 
\[(c^+_\lambda)_*(s^+_\lambda)^!\Psi(\mF)\is(d_\lambda)_!(t_\lambda)^*\mF,\ \ \ \ \text{if\ }\lambda\in\mL.\]
\end{lemma}
\begin{proof}
Same argument as in \cite[Corollary 9.2]{N2}
\end{proof}

We write 
$
k_\lambda:\Omega K_c\backslash\mO_c^\lambda\ra C_\bbR^\lambda$
for the restriction of $d_\lambda$ and 
$
p_\lambda: T_\bbR^\lambda\ra G_\bbR(\bbR[t^{-1}])\backslash T_\bbR^\lambda$ 
for the natural quotient map.

\quash{
We have the following diagrams
\beq
\xymatrix{C_\bbR^\lambda&T_\bbR^\lambda\ar[d]^{c_\lambda}\ar[l]
\\ S_\bbR^\lambda\ar[r]^{}\ar[u]&\Gr_\bbR}\ \ \ \ \ 
\xymatrix{\Omega K_c\backslash\mO_c^\lambda\ar[r]^{}\ar[d]^{k_\lambda}&\Omega K_c\backslash\mO_\bbR^\lambda\ar[d]^{u}
\\ C_\bbR^\lambda\ar[r]^{j_-^\lambda\ \ \ }&\Gr_\bbR}\ \ \ \ 
\xymatrix{T_\bbR^\lambda\ar[r]^{i_-^\lambda\ \ \ }\ar[d]^{\phi_-^\lambda}&
\Gr_\bbR\ar[d]^{\phi_-}
\\ G_\bbR(\bbR[t^{-1}])\backslash T_\bbR^\lambda\ar[r]^{}&G_\bbR(\bbR[t^{-1}])\backslash\Gr_\bbR}
\eeq}

\begin{lemma}\label{bijection 2}
The map $k_\lambda:\Omega K_c\backslash\mO_c^\lambda\ra C_\bbR^\lambda$ is a $K_c$-equivariant isomorphism. 
There is a bijection between isomorphism classes of 
$K_c$-equivariant local systems $\omega^+$ on $S_\bbR^\lambda$,
$K_c$-equivariant local systems $\omega^-$ on $T_\bbR^\lambda$, 
$K_c$-equivariant local systems $\omega$ on $C_\bbR^\lambda$, 
$K_c$-equivariant local systems $\tau$ on $\Omega K_c\backslash\mO_c^\lambda$, 
and local system $\omega_\bbR$ on $G_\bbR(\bbR[t^{-1}])\backslash T_\bbR^\lambda$,
characterizing by the property that 
$\omega^\pm\is(c_\lambda^\pm)^*\omega$, $
\tau\is(k_\lambda)^*\omega$
, and $(p_\lambda)^*\omega_\bbR\is (c_\lambda^-)^*\omega$
\end{lemma}
\begin{proof}
The first claim follows from the fact that 
$\Omega K_c\backslash\mO_c^\lambda\is C_\bbR^\lambda\is
K_c(\lambda)\backslash K_c$, where $K_c(\lambda)$
is the stabilizer of $\lambda$ in $K_c$, and the $K_c$-equivariant property of $k_\lambda$. 
The second claim follows from the facts that the 
contraction maps $c_\lambda^\pm$ are
$K_c$-equivariant and the 
fibers of $c_\lambda^\pm$ and the quotient
$K_c\backslash G_\bbR(\bbR[t^{-1}])$
are contractible.

\end{proof}

For any $\lambda\in\mL$ and a $K_c$-equivaraint local system $\omega$ on $C_\bbR^\lambda$, 
one has the standard and co-standard sheaves 
\[\calT_*^+(\lambda,\omega):=(s_\lambda^+)_*(\omega^+) \text{\ \ and\ \ \ } 
\calT_!^+(\lambda,\omega):=(s_{\lambda}^+)_!(\omega^+)\]
in $D_c(G_\bbR(\mO_\bbR)\backslash\Gr_\bbR)$.
Recall the standard sheaf
$\calS_*^+(\lambda,\tau)$ in $D_c(K(\calK)\backslash\Gr)$ (see (\ref{standard})).

\begin{proposition}\label{Psi_K}
We have $\Psi(\calS_*^+(\lambda,\tau))\is\calT_*^+(\lambda,\omega)$
\end{proposition}
\begin{proof}
It suffices to show that 
\[
(a)\ (s_\lambda^+)^!\Psi(\calS_*^+(\lambda,\tau))\is\omega^+
\ \ \ \ \text{and}\ \ \ \ \ (b)
\ (s_\mu^+)^!\Psi(\calS_*^+(\lambda,\tau))=0\ \text{for}\ \mu\neq\lambda.\]
Proof of (a). 
By Lemma \ref{bijection 2}, it suffices to show that 
$(c_\lambda^+)_*(s_\lambda^+)^!\Psi(\calS_*^+(\lambda,\tau))\is\omega$.
But it follows from 
Lemma \ref{hyperbolic 2} and 
Lemma \ref{bijection}, indeed, we have 
\[(c_\lambda^+)_*(s_\lambda^+)^!\Psi(\calS_*^+(\lambda,\tau))\stackrel{\on{Lem} \ref{hyperbolic 2}}\is
(d_\lambda)_!(t_\lambda)^*\calS_*^+(\lambda,\tau)\is 
(k_\lambda)_!(\calS_*^+(\lambda,\tau)|_{\Omega K_c\backslash\mO_c^\lambda})\is(k_\lambda)_!\tau\stackrel{\on{Lem} \ref{bijection}}\is\omega.\]
Proof of (b). 
It suffices to show that 
$\mF_{\mu}:=(c_\mu^+)_*(s_\lambda^+)^!\Psi(\calS_*^+(\lambda,\tau))=0$ for 
$\mu\neq\lambda$. For this, it is enough to show that 
\[H_c^*(\mF_{\mu}\otimes\mL)=0\]
for any $K_c$-equivariant local system $\mL$ on $C_\bbR^\mu$.
Recall the contraction map 
$\phi_\mu^-:\Omega K_c\backslash\mO_\bbR^\mu\ra\Omega K_c\backslash\mO_c^\mu$ coming from the Matsuki flow $\phi_z:\Gr\ra
\Gr$ in \S\ref{Morse flow}.
By Lemma \ref{bijection} and Lemma \ref{bijection 2}, for any 
such $\mL$, there is a $K_c$-equivariant local system 
$\mL'$ on $\Omega K_c\backslash\mO_c^\mu$ satisfying 
\beq\label{compatibility}
(d_\mu)^*\mL\is(\phi_\mu^-)^*\mL'\ \ (\text{as equivaraint local systems on}\ 
\Omega K_c\backslash\mO_\bbR^\mu).
\eeq
Thus we have 
\beq\label{iso 1}
 H_c^*(\mF_{\mu}\otimes\mL)\is H_c^*((c_\mu^+)_*(s_\mu^+)^!\Psi_K(\calS_*^+(\lambda,\tau))\otimes\mL)
\eeq
\[\is
H_c^*((d_\mu)_!(t_\mu)^*\calS_*^+(\lambda,\tau)\otimes\mL)
\ \ (\text{by Lemma}\ \ref{hyperbolic 2})
\]
\[
\is H_c^*((d_\mu)_!((t_\mu)^*\calS_*^+(\lambda,\tau)\otimes (d_\mu)^*\mL))
\ \ (\text{by projection formula})\]
\[
\is H_c^*((\phi_\mu^-)_!((t_\mu)^*\calS_*^+(\lambda,\tau)\otimes (\phi_\mu^-)^*\mL'))\ \ (by\ \  (\ref{compatibility}))
\]
\[
\is H_c^*((\phi_\mu^-)_!(t_\mu)^*\calS_*^+(\lambda,\tau)\otimes \mL')
\ \ (\text{by projection formula})
\]
Note that $t_\mu=i_-^\mu:\Omega K_c\backslash\mO_\bbR^\mu\ra
\Omega K_c\backslash\Gr$ is the embedding for the 
unstable manifold of the Morse flow $\phi_z$ on $\Omega K_c\backslash\Gr$
and we have 
\beq\label{iso 2}
(\phi_\mu^-)_!(t_\mu)^*\calS_*^+(\lambda,\tau)\is 
(\phi_\mu^-)_!(i_-^\mu)^*\calS_*^+(\lambda,\tau)\is
(\phi_\mu^+)_*(i_+^\mu)^!\calS_*^+(\lambda,\tau),
\eeq
here $i_+^\mu:\Omega K_c\backslash\mO_K^\mu\ra\Omega K_c\backslash\Gr $ is the embedding for the stable manifold of the flow $\phi_z$ and 
$\phi_\mu^+:\Omega K_c\backslash\mO_K^\mu\ra\Omega K_c\backslash\mO_c^\mu$ is the contraction map.
Since $\mu\neq\lambda$, we have $(i_+^\mu)^!\calS_*^+(\lambda,\tau)=0$ and 
it implies 
\[H_c^*(\mF_{\mu}\otimes\mL)\stackrel{(\ref{iso 1})}\is H_c^*((\phi_\mu^-)_!(t_\mu)^*\calS_*^+(\lambda,\tau)\otimes \mL')\stackrel{(\ref{iso 2})}\is H_c^*(((\phi_\mu^+)_*(i_+^\mu)^!\calS_*^+(\lambda,\tau)\otimes \mL')=0.\]
Claim (b) follows and the proposition is proved.

\end{proof}

\subsection{The Radon transform}
Recall the flow $\psi_z^1:\Gr_\bbR^{(2)}\to\Gr_\bbR^{(2)}$ 
in  \eqref{flow on Gr_R^2}. By Lemma \ref{flows on Gr^(2)}, it
restricts to a 
flow on the special fiber $\Gr_\bbR\is\Gr_\bbR^{(2)}|_0$ with 
critical manifolds $\bigcup_{\lambda\in\Lambda_S^+} C_\bbR^\lambda$ and 
$S_\bbR^\lambda$, respectively $T_\bbR^\lambda$, is the stable manifold, respectively unstable manifold, of $C_\bbR^\lambda$.
Let 
$t_\lambda^-:G_\bbR(\bbR[t^{-1}])\backslash T_\bbR^\lambda\ra G_\bbR(\bbR[t^{-1}])\backslash\Gr_\bbR 
$ be the natural inclusion map. 
According to Lemma \ref{bijection 2}, for any $K_c$-equivariant local system 
$\omega$ on $C_\bbR^\lambda$,
we have the standard and 
co-standard sheaves
\[\calT_*^-(\lambda,\omega):=(t_\lambda^-)_*(\omega_\bbR) \text{\ \ and\ \ \ } 
\calT_!^-(\lambda,\omega):=(t_{\lambda}^-)_!(\omega_\bbR).\]
Recall the Radon transform \[\Upsilon_\bbR:
D_c(G_\bbR(\mO_\bbR)\backslash\Gr_\bbR)\lra
D_c(G_\bbR(\bbR[t^{-1}])\backslash\Gr_\bbR)\] in \eqref{Radon TF}.
The same argument as in the proof of Theorem \ref{AM}, replacing the 
Matsuki flow $\phi_t:\Gr\to\Gr$ by the
$\bbR_{>0}$-flow $\psi_z^1:\Gr_\bbR\to\Gr_\bbR$ and Lemma \ref{bijection} by Lemma \ref{bijection 2},
gives us:
\begin{prop}\label{real RT}
The Radon transform defines an equivalence of categories 
\[\Upsilon_\bbR:D_c(G_\bbR(\mO_\bbR)\backslash\Gr_\bbR)\stackrel{\sim}\lra
D_!(G_\bbR(\bbR[t^{-1}])\backslash\Gr_\bbR).\]
Moreover, for any $K_c$-equivariant local system $\omega$
on $C_\bbR^\lambda$ we have 
\[\Upsilon_\bbR(\calT_*^+(\lambda,\omega))\is\calT^-_!(\lambda,\omega\otimes\mL_\lambda)[d_\lambda].\]
Here we regard the local system $\mL_\lambda$ in \eqref{L_lambda} as 
a local system on $C_\bbR^\lambda$ via the isomorphism 
$k_\lambda:\Omega G_c\backslash\mO_c^\lambda\is C_\bbR^\lambda$
in Lemma \ref{bijection 2}.

\end{prop}

\subsection{The functor $\Psi_\bbR$}
Note that, by Proposition \ref{an open family},  the map $LG_\bbR^{(2)}\backslash\Gr^{(2),0}_\bbR\is
\Bun_G(\mathbb P^1)_{\bbR,0}\times i\mathbb R_{}
\ra i\mathbb R_{}$ is isomorphic to a constant family. It implies

\begin{prop}\label{Psi_R}
The nearby cycles functor 
\[\Psi_{\bbR}:D_!(LG_\mathbb R\backslash\Gr)\stackrel{}\lra D_!(G_\mathbb R(\mathbb R[t^{-1}])\backslash\Gr^0_\mathbb R)\]
is a t-exact equivalence (with respect to the natural $t$-structures) 
satisfying 
$\Psi_{\bbR}(\calS_!^-(\lambda,\tau))\is \calT_!^-(\lambda,\omega)$.
\end{prop}


\subsection{Proof of Theorem \ref{equ} and Theorem \ref{diagram}}
It remains to prove that $\Psi:D_c(K(\calK)\backslash\Gr)\ra
D_c(G_\bbR(\mO_\bbR)\backslash\Gr^0_\bbR)$ is an equivalence and 
Theorem \ref{diagram}.
Note that for $\mF\in D_c(K(\calK)\backslash\Gr)$
there is a natural transformation (induced by the natural transformation 
$(f_0)_!i^!\ra(\bar i)^!f_!$)
\beq\label{natural}
\Upsilon_\bbR\circ\Psi(\mF)=(f_0)_!i^!j_!(\mF\boxtimes\bC_{i\bbR_{>0}})\ra (\bar i)^!f_!j_!(\mF\boxtimes\bC_{i\bbR_{>0}})
\is(\bar i)^!(\bar j)_!(f^0)_!(\mF\boxtimes\bC_{i\bbR_{>0}})=
\eeq
\[
=\Psi_{\bbR}\circ\Upsilon(\mF).
\]
Moreover, it follows from Lemma \ref{surjective}, Proposition \ref{Psi_K},  
Proposition \ref{real RT}, and Proposition \ref{Psi_R} that (\ref{natural}) is an isomorphism for
the standard sheaf $\calS_*^+(\lambda,\tau)$. 
Since the category $D_c(K(\calK)\backslash\Gr)$ is generated by 
$\calS_*^+(\lambda,\tau)$, it implies (\ref{natural}) is an isomorphism.
By Theorem \ref{AM}, 
Proposition \ref{real RT}, and Proposition \ref{Psi_R},
the functors $\Psi_\bbR$, $\Upsilon$, and $\Upsilon_\bbR$ are equivalences
and (\ref{natural})
implies $\Psi:D_c(K(\calK)\backslash\Gr)\ra
D_c(G_\bbR(\mO_\bbR)\backslash\Gr^0_\bbR)$ is an equivalence.
This finishes the proof of Theorem \ref{equ} and Theorem \ref{diagram}


\section{Compatibility of Hecke actions}\label{s:hecke}
Recall the derived Satake category $D_c(G(\mO)\backslash \Gr)$ is naturally monoidal  with respect to convolution.
We will write $\calF_1 \star\calF_2$ for the convolution product of $\calF_1, \calF_2\in D_c(G(\mO)\backslash \Gr)$.

Here we enhance the equivalences and commutative square of Theorems~\ref{equ} and~\ref{diagram} to  $D_c(G(\mO)\backslash \Gr)$-modules. 
 Roughly speaking, we will  take advantage of  the natural right actions on the categories involved, whereas the prior Radon transforms were performed on the left.


\subsection{Hecke actions}\label{ss:conv}

First,    the affine Matsuki correspondence for sheaves 
\[
\Upsilon:D_c(K(\calK)\backslash\Gr)  \stackrel{\sim}\lra
D_!(LG_{\mathbb R}\backslash\Gr)
\]
is naturally an equivalence of $D_c(G(\mO)\backslash \Gr)$-modules by convolution on the right. To see this, 
recall $\Upsilon$ is the restriction to $D_c(K(\calK)\backslash\Gr)\subset D_c(LK_\bbR\backslash\Gr)$ of the push-forward
$u_!:D_c(LK_\bbR\backslash\Gr)\to D_c(LG_\bbR\backslash\Gr)$
along the quotient map $u:LK_c\backslash\Gr\ra
LG_\bbR\backslash\Gr$.
We can equip this construction with compatibility with convolution on the right by using the commutative action diagram
\[
\xymatrix{
\ar[d] LK_c\backslash G(\calK) \times^{G(\mO)} \Gr  \ar[r]^-{u \times\id} & 
LG_\bbR\backslash G(\calK) \times^{G(\mO)} \Gr \ar[d] \\
LK_c\backslash\Gr \ar[r]^-u & 
LG_\bbR\backslash\Gr
}
\]
and its natural iterations.

Similarly,
   the Radon equivalence
\[\Upsilon_\bbR:D_c(G_\bbR(\mO_\bbR)\backslash\Gr_\bbR)\stackrel{\sim}\lra
D_!(G_\bbR(\bbR[t^{-1}])\backslash\Gr_\bbR)\]
is naturally an equivalence of $D_c(G_\bbR(\mO_\bbR)\backslash\Gr_\bbR)$-modules by convolution on the right.
 To see this, 
recall $\Upsilon_\bbR$ is
 the restriction to $D_c(G_\bbR(\mO_\bbR)\backslash\Gr_\bbR)\subset
D_c(G_\bbR\backslash\Gr_\bbR)$
of the 
push-forward $p_!:D_c(G_\bbR\backslash\Gr_\bbR)\ra
D_c(G_\bbR(\bbR[t^{-1}])\backslash\Gr_\bbR)$
along the quotient map 
$p:G_\bbR\backslash\Gr\to G_\bbR(\bbR[t^{-1}])\backslash\Gr_\bbR$.
We can equip this construction with compatibility with convolution on the right by using the commutative action diagram
\[
\xymatrix{
\ar[d] G_\bbR\backslash G(\calK_\bbR) \times^{G(\mO_\bbR)} \Gr_\bbR  \ar[r]^-{p \times\id} & 
G_\bbR(\bbR[t^{-1}])\backslash G(\calK_\bbR) \times^{G(\mO_\bbR)} \Gr_\bbR \ar[d] \\
G_\bbR\backslash\Gr_\bbR \ar[r]^-p & 
G_\bbR(\bbR[t^{-1}])\backslash\Gr_\bbR
}
\]
and its natural iterations.


\subsection{From complex to real kernels}\label{ss:kernels}

Following~\cite{N1}, nearby cycles in the real Beilinson-Drinfeld Grassmannian $\Gr^{(2)}_\mathbb R$ 
 over  $i \bbR_{\geq 0}$
gives a  functor
\[m:D_c(G(\calO) \backslash\Gr)\lra D_c(G_\bbR(\mO_\bbR)\backslash\Gr^0_\bbR)
\subset D_c(G_\bbR(\mO_\bbR)\backslash\Gr_\bbR)\]
Namely, by Lemma~\ref{real forms of various loop groups}, 
there is a  canonical diagram of  $G_\bbR$-equivariant maps
  \begin{equation}\label{c to r diag}
\xymatrix{
 \Gr & \ar[l]_-\pi \Gr\times i\bbR_{>0} \is \Gr^{(2)}_\bbR|_{
i\mathbb R_{>0}} \ar@{^(->}[r]^-j &  \Gr^{(2)}_\bbR|_{
i\mathbb R_{\geq 0}} & \ar@{_(->}[l]_-i \Gr^{(2)}_\mathbb R|_{0}\is\Gr_\mathbb R
}
  \end{equation}
where we view $G_\bbR \subset LG^{(2)}_\bbR$ as the constant group-scheme.
One defines $m=i^*j_*\pi^*f_\bbR$ where we write  $f_\bbR:D_c(G(\calO) \backslash\Gr) \to D_c(G_\bbR \backslash\Gr)$ 
for the forgetful functor. 

Note the domain and codomain of $m$ both have natural convolution monoidal structures. 
To equip $m$ with a monoidal structure, we proceed as follows. 

Let $\Gr^{(2)} \tilde\times \Gr^{(2)}$ be  the moduli of $x_1, x_2 \in \bbP^1$, $\calE_1, \calE_2$ $G$-torsors on $\bbP^1$, $\phi$ a trivialization of $\calE_1$ over $\bbP^1 \setminus \{x_1, x_2\}$, and $\alpha$ an isomorphism from $\calE_1$ to $\calE_2$ over  $\bbP^1 \setminus \{x_1, x_2\}$. 
Let $\Gr^{(2)}_\bbR \tilde\times \Gr^{(2)}_\bbR$ be the real form of $\Gr^{(2)} \tilde\times \Gr^{(2)}$ with respect to the twisted conjugation that exchanges $x_1$ and $x_2$. 

Then
there is a  canonical diagram of  $G_\bbR$-equivariant maps
\begin{equation} \label{eq: nearby + conv}
  \xymatrix{
 G(\calK) \times^{G(\calO)} \Gr & \ar[l]_-\pi G(\calK) \times^{G(\calO)} \Gr\times i\bbR_{>0} \is\Gr^{(2)}_\bbR \tilde\times \Gr^{(2)}_\bbR|_{
i\mathbb R_{>0}} \ar@{^(->}[r]^-j &  }
\end{equation}
$$
\xymatrix{
\Gr^{(2)}_\bbR \tilde\times \Gr^{(2)}_\bbR|_{i\mathbb R_{\geq 0}} & \ar@{_(->}[l]_-i 
\Gr^{(2)}_\bbR \tilde\times \Gr^{(2)}_\bbR |_{0}\is   G_\bbR(\calK_\bbR)  \times^{G_\bbR(\calO_\bbR)} \Gr_\mathbb R
}
$$
Moreover, the  convolution maps on the end terms naturally extend to the entire diagram. By standard identities, we arrive at a canonical isomorphism $m(\calF_1 \star \calF_2) \simeq m(\calF_1) \star m(\calF_2)$. By using iterated versions of the above moduli spaces, we may likewise equip $m$ with the associativity constraints of a monoidal structure.


\subsection{Compatibility of actions}\label{ss:actions}

Note we can view the  Radon equivalence  $\Upsilon_\bbR$ as an equivalence of $D_c(G(\calO) \backslash\Gr)$-modules 
via the monoidal functor 
$$
m:D_c(G(\calO) \backslash\Gr)\lra D_c(G_\bbR(\mO_\bbR)\backslash\Gr^0_\bbR)
$$

Now we have the following further compatibility of our constructions.

\begin{thm}\label{conv comp}
Via the monoidal functor
\[m:D_c(G(\calO) \backslash\Gr)\lra D_c(G_\bbR(\mO_\bbR)\backslash\Gr^0_\bbR)
\]
the  equivalences 
\[\Psi:D_c(K(\calK)\backslash\Gr)\stackrel{\sim}\lra D_c(G_\bbR(\mO_\bbR)\backslash\Gr^0_\bbR),\] 
\[\Psi_{\bbR}:D_!(LG_\mathbb R\backslash\Gr)\stackrel{\sim}\lra D_!(G_\mathbb R(\mathbb R[t^{-1}])\backslash\Gr^0_\mathbb R)\]
of Theorem~\ref{equ} and commutative square
\[
\xymatrix{D_c(K(\calK)\backslash\Gr)\ar[r]^{\Psi}\ar[d]^{\Upsilon}&D_c(G_\mathbb R(\mO_\mathbb R)\backslash\Gr^0_\mathbb R)
\ar[d]^{\Upsilon_\bbR}\\
D_!(LG_{\mathbb R}\backslash\Gr)\ar[r]^{\Psi_{\bbR}\ \ }&D_!(G_\mathbb R(\mathbb R[t^{-1}])\backslash\Gr^0_\mathbb R).}
\]
of Theorem~\ref{diagram} are naturally of  $D_c(G(\calO) \backslash\Gr)$-modules.
\end{thm}

\begin{proof}
We will focus on the compatibility for the top row and indicate the moduli spaces needed. We  leave it to the reader to pass to sheaves and apply standard identities.  The  compatibility for the bottom row and entire square can be argued  similarly.

Let $\QM^{(2)}(\bP^1,G,K) \tilde\times \Gr^{(2)}$ be  the moduli of $x_1, x_2 \in \bbP^1$, $\calE_1, \calE_2$ $G$-torsors on $\bbP^1$, $\sigma$ a section of $\calE_1 \times^G G/K$ over $\bbP^1 \setminus \{x_1, x_2\}$, and $\alpha$ an isomorphism from $\calE_1$ to $\calE_2$ over  $\bbP^1 \setminus \{x_1, x_2\}$. 
Let $\QM^{(2)}(\bP^1,G,K)_{\bbR}  \tilde\times \Gr^{(2)}_{\bbR} $ be the real form of $\QM^{(2)}(\bP^1,G,K) \tilde\times \Gr^{(2)}$ with respect to the twisted conjugation that exchanges $x_1$ and $x_2$.

Then
there is a  canonical diagram of  $K_c$-equivariant maps
\begin{equation}\label{eq: nearby + action}
  \xymatrix{
LK_c\backslash G(\calK) \times^{G(\mO)} \Gr  & \ar[l]_-\pi  LK_c\backslash G(\calK) \times^{G(\mO)} \Gr  \times i\bbR_{>0} \is \QM^{(2)}(\bP^1,G,K)_{\bbR}  \tilde\times \Gr^{(2)}_{\bbR}|_{
i\mathbb R_{>0}} \ar@{^(->}[r]^-j &  }
\end{equation}
$$
\xymatrix{
\QM^{(2)}(\bP^1,G,K)_{\bbR}  \tilde\times \Gr^{(2)}_{\bbR}|_{i\mathbb R_{\geq 0}} & \ar@{_(->}[l]_-i 
\QM^{(2)}(\bP^1,G,K)_{\bbR}  \tilde\times \Gr^{(2)}_{\bbR} |_{0}\is   K_c\backslash G_\bbR(\calK_\bbR)  \times^{G_\bbR(\calO_\bbR)} \Gr_\mathbb R
}
$$
Note we could equivalently obtain diagram  \eqref{eq: nearby + action} by taking diagram  \eqref{eq: nearby + conv}
and quotienting by the left action of  the group-scheme $LK^{(2)}_\bbR$.

As with the convolution maps in diagram \eqref{eq: nearby + conv}, the  actions maps on the end terms of diagram  
 \eqref{eq: nearby + action} naturally extend to the entire diagram. By standard identities, we arrive at a canonical isomorphism $\Psi(\calM \star \calF) \simeq \Psi(\calM) \star m(\calF)$. By using iterated versions of the above moduli spaces, we may likewise equip $\Psi$ with the associativity constraints of a module map.
\end{proof}

\appendix
\section{Real analytic stacks}\label{Real stacks}
\subsection{Basic definitions}
Let $\on{RSp}$ be the site of real analytic spaces where the coverings are 
\'etale (=locally biholomorphic) maps $\{S_i\to S\}_{i\in I}$ such that the map $\bigsqcup S_i\to S$ is surjective. A real analytic pre-stack is a functor $\sX:\on{RSp}\to\on{Grpd}$ from 
$\on{RSp}$ to the category of groupoids $\on{Grpd}$ and a real analytic stack is a 
pre-stack which is a sheaf. 
Let $\Gamma\rightrightarrows X$ be a groupoid in real analytic spaces. 
We define $\Gamma\backslash X$ be the stack associated to the pre-stack 
$S\to \{\Gamma(S)\rightrightarrows X(S)\}$.
A morphism $\sX\to\sY$ between real analytic stacks is called representable if
for any morphism from a real analytic space 
$Y\to\sY$, the fiber product $\sX\times_\sY Y$ is representable by a real analytic space. We say that a representable morphism $\sX\to\sY$ has property 
P if it has property P after base change along any morphism 
from a real analytic space.

\subsection{From real algebraic stacks to 
real analytic stacks}
For any $\bbR$-scheme $X$ locally of finite type, its $\bbR$-points  
$X(\bbR)$ is naturally a real analytic space, denoted by $X_\bbR$, and the assignment 
$X\to X_\bbR$ defines a functor from the category of 
$\bbR$-scheme to the category of real analytic spaces.
We are going to extend the above construction to real algebraic stacks.
Let $\sX$ be a real algebraic stack. 
A presentation $f:X\to\sX$ of $\sX$ is called a 
$\bbR$-surjective presentation if it induces a surjective map  
$X(\bbR)\to |\sX(\bbR)|$ on the set of isomorphism classes of objects. 

\begin{lemma}\label{image}
Let $f_1:X_1\to\sX$ and $f_2:X_2\to\sX$ be 
two $\bbR$-surjective presentations of $\sX$.
Let $\Gamma_i=X_i\times_\sX X_i\rightrightarrows X_i$ be the corresponding 
groupoid. 
Then
there is a canonical isomorphism of real analytic stacks 
\[\Gamma_{1,\bbR}\backslash X_{i,\bbR}\is\Gamma_{2,\bbR}\backslash X_{2,\bbR}.\]

\end{lemma}
\begin{proof}
Let $Y=X_1\times_\sX X_2$ and $\Gamma=Y\times_\sX Y$
be the corresponding groupoid. As $X_i$ is a presentation of 
$\sX$ the natural map $Y\to X_i$ is smooth and one can check that 
the natural map $\Gamma_\bbR\backslash Y_\bbR\to\Gamma_{i,\bbR}\backslash X_{i,\bbR}$ is an isomorphism. The lemma follows.

\end{proof}

\begin{definition}
Given a real algebraic stack $\sX$ which admits a $\bbR$-surjective presentation, we define the associated 
real analytic stack to be 
\[\sX_\bbR:=\Gamma_\bbR\backslash X_\bbR\]
where $X\to\sX$ is a $\bbR$-presentation of $\sX$.
\end{definition}
By the lemma above $\sX_\bbR$ is well-defined and the assignment 
$\sX\to\sX_\bbR$ defines a functor from the 2-category of 
real algebraic stacks which admit $\bbR$-presentations to the 2-category of 
real analytic stacks.

\begin{example}\label{BG}
Let $X$ be a $\bbR$-scheme and $G$ be an algebraic group over $\bbR$ acting 
on $X$. Consider the algebraic stack $\sX=G\backslash X$.
Let $T_1,...,T_s$ be the isomorphism classes of 
$G$-torsors.  Define $G_i:=\on{Aut}_G(T_i)$ and the $\bbR$-scheme 
$X_i:=\Hom_{G}(T_i,X)$. Note that 
$G_i$ acts on $X_i$ and 
the collection $\{G_1,...,G_s\}$ gives all the pure-inner 
forms of $G$. 
Consider the real algebraic stack $G_i\backslash X_i$.  
We have $G_i\backslash X_i\is\sX$ and the map 
$\bigsqcup_{i=1}^sX_i\to\sX$ is a $\bbR$-presentation. In addition, 
the $\bbR$-presentation above induces an isomorphism of real analytic stacks 
$\bigsqcup_{i=1}^sG_{i,\bbR}\backslash X_{i,\bbR}\is \sX_\bbR$.

\end{example}

\begin{definition}
Let $\sX$ be a real analytic stack (resp. a real algebraic stack).
The stack $\sX$ is called of Bernstein-Lunts type (BL-type) if it 
 is an union of open substacks 
$\sX=\bigcup\sX_i$
, each $\sX_i$ being a quotient 
$G\backslash X$
where $X$ is a  real analytic space 
(resp. $\bbR$-scheme )
and $G$ is a real analytic group (resp. affine algebraic group over $\bbR$) acting on 
$X$. 
\end{definition}

Note that, by the example above, each real algebraic stack $\sX$ of BL-type admits 
a $\bbR$-surjective presentation and the corresponding real analytic stack 
$\sX_\bbR$ is also of BL-type.

The discussion above can be generalized to real ind-schemes and real ind-stacks.
Let $X_0\to X_1\to\cdot\cdot\cdot X_k\to\cdot\cdot\cdot$ be a diagram of 
closed embedding of 
$\bbR$-schemes. 
Let $X=\underset{\lra}\lim\ X_i$ be the corresponding ind-scheme over $\bbR$.
We define 
$X_\bbR=\underset{\lra}\lim\ X_{i,\bbR}$
to be real ind-analytic space associated to the diagram 
$X_{0,\bbR}\to X_{1,\bbR}\to\cdot\cdot\cdot X_{k,\bbR}\to\cdot\cdot\cdot$.
Similarly, let $\sX=\underset{\lra}\lim\sX_{i}$ be a real ind-stack 
associated to a diagram 
$\sX_{0}\to \sX_{1}\to\cdot\cdot\cdot \sX_{k}\to\cdot\cdot\cdot$
of real algebraic stacks which admit $\bbR$-presentations. 
We define $\sX_\bbR=\underset{\lra}\lim\ \sX_{i,\bbR}$.
An ind-algebaic stack $\sX=\underset{\lra}\lim\ \sX_i$  (resp. an real analytic ind-stack) is 
called of BL-type if each $\sX_i$ is of BL-type.

Let  $\sX=\underset{\lra}\lim\ \sX_i$ be a 
in real algebaic ind-stack
(resp. an real analytic ind-stack). By definition, a morphism $f:\sX\to\sY$ from 
$\sX$ to a real algebraic stack $\sX$ (resp. a real analytic stack)
is the limit of morphism $f_i:\sX_i\to\sY$. It is called representable if each 
$f_i$ is representable. 

\subsubsection{}
One can regard real algebraic stacks as complex algebraic stacks with real structures and the discussion above has an obvious generalization to 
this setting.  Let $\sX$ be a complex algebraic stack and let
$\sigma$ be a real structure on $\sX$, that is, a complex conjugation (or a semi-linear involution) $\sigma:\sX\to\sX$. Then 
a presentation $f:X\to\sX$ of $\sX$
is called a 
$\bbR$-surjective presentation if it satisfies the following properties.
(1) There is a real structure $\sigma$ on $X$ such that 
$f$ is compatible with the real structures on $X$ and $\sX$.
(2) The map $f$  
induces a surjective map  
$X(\bC)^{\sigma}\to |\sX(\bC)^{\sigma}|$. One can check that 
Lemma \ref{image} still holds in this setting, 
thus 
for a pair $(\sX,\sigma)$ as above which admits a $\bbR$-surjective presentation, there is a well-defined 
real analytic stack $\sX_\bbR$ given by 
$\sX_\bbR:=\Gamma(\bC)^\sigma\backslash X(\bC)^\sigma$, where 
$X\to\sX$ is a $\bbR$-surjective presentation, 
$\Gamma=X\times_{\sX}X$ is the corresponding groupoid (Note that 
$\Gamma$ has a canonical real structure $\sigma$ coming from 
$X$ and $\sX$).
Finally, 
the previous discussions about stacks of BL-type and ind-stacks can be easily generalized to this new setting. The details are left to the reader.

\subsection{Sheaves on real analytic stacks}
Let $X$ be a real analytic space. We will denote by 
$D_c(X)$ the corresponding bounded derived category of $\bC$-constructible sheaves. 
Let $\sX$ be a real analytic stack of BL-type.
We define 
$D_c(\sX)=\underset{\longleftarrow}\lim\ D_c(\sX_i)$, where each 
$D_c(\sX_i)$ is the bounded equivariant derived category in the sense of Bernstein-Lunts \cite{BL}.
Let $\sX=\underset{\lra}\lim\ \sX_i$ be a real analytic ind-stack of BL-type.
We define $D_c(\sX)=\underset{\lra}\lim\ D_c(\sX_i)$.

Let $f:\sX\to\sY$ be a representable morphism of real analytic stacks of BL-type. We
will denote by $f_*,f^*,f_!,f^!$ the
corresponding functors between $D(\sX)$ and $D(\sY)$ always understood in the derived sense.
Let $f:\sX\to\sY$ be a representable morphism from a real analytic ind-stack to 
a real analytic stack. Assume both $\sX$ and $\sY$ are of BL-type, then the functors
$f_*,f_!$ are well-defined.

\end{document}